\documentclass{amsart}
\usepackage{amssymb} 
\usepackage{graphicx}
\usepackage[curve]{xypic}
\usepackage{tikz}
\usetikzlibrary{arrows,snakes,shapes,shadows}
\usepackage{hyperref}

\newcommand{\grassman}{\mathbf G}
\newcommand\ghlim{{\lim}^{\text{\tiny GH}}}

\newbox\mybox
\def\overtag#1#2#3{\setbox\mybox\hbox{$#1$}\hbox to
  0pt{\vbox to 0pt{\vglue-#3\vglue-\ht\mybox\hbox to \wd\mybox
      {\hss$\ss#2$\hss}\vss}\hss}\box\mybox}
\def\undertag#1#2#3{\setbox\mybox\hbox{$#1$}\hbox to 0pt{\vbox to
    0pt{\vglue#3\vglue\ht\mybox\hbox to \wd\mybox
      {\hss$\ss#2$\hss}\vss}\hss}\box\mybox}
\def\lefttag#1#2#3{\hbox to 0pt{\vbox to 0pt{\vglue -6pt\hbox to
      0pt{\hss$\ss#2$\hskip#3}\vss}}#1}
\def\righttag#1#2#3{\hbox to 0pt{\vbox to 0pt{\vglue -6pt\hbox to
      0pt{\hskip#3$\ss#2$\hss}\vss}}#1}
\let\ss\scriptstyle
\def\splicediag#1#2{\xymatrix@R=#1pt@C=#2pt@M=0pt@W=0pt@H=0pt}
\def\Dot{\lower.2pc\hbox to 2pt{\hss$\bullet$\hss}}
\def\Circ{\lower.2pc\hbox to 2pt{\hss$\circ$\hss}}
\def\Vdots{\raise5pt\hbox{$\vdots$}}
\newcommand\lineto{\ar@{-}}
\newcommand\dashto{\ar@{--}}
\newcommand\dotto{\ar@{.}}

\newcommand\denom{{\operatorname{denom}}}
\let\cal\mathcal
\renewcommand{\setminus}{\smallsetminus}
\newcommand\Q{{\mathbb Q}}
\newcommand\R{{\mathbb R}}
\newcommand\C{{\mathbb C}}
\newcommand\Z{{\mathbb Z}}

\renewcommand\L{{$\cal L$}}
\newcommand\T{{$\cal T$}}
\newcommand\A{{$\cal A$}}
\newcommand\Nn{{\mathcal N}}

\newtheorem{theorem}{Theorem}[section]
\newtheorem{proposition}[theorem]{Proposition}
\newtheorem*{theorem*}{Theorem}
\newtheorem{corollary}[theorem]{Corollary}
\newtheorem{lemma}[theorem]{Lemma}

\theoremstyle{definition}
\newtheorem{example}[theorem]{Example}
\newtheorem{remark}[theorem]{Remark}
\newtheorem*{remark*}{Remark}
\newtheorem{definition}[theorem]{Definition}

\renewcommand{\int}{\operatorname{int}}
\newcommand{\lcm}{\operatorname{lcm}}

\begin{document}
\title[Bilipschitz classification
  of normal surface singularities]{The thick-thin decomposition 
and the bilipschitz classification
  of normal surface singularities}
\author{Lev Birbrair}\address{Departamento de Matem\'atica,
  Universidade Federal do Cear\'a (UFC), Campus do Picici, Bloco 914,
  Cep. 60455-760. Fortaleza-Ce, Brasil} \email{birb@ufc.br}
\author{Walter D Neumann}\address{Department of Mathematics, Barnard
  College, Columbia University, 2990 Broadway MC4429, New York, NY
  10027, USA} \email{neumann@math.columbia.edu} \author{Anne Pichon}
\address{Aix-Marseille Universit\'e, IML\\ FRE 3529 CNRS,
   Campus de Luminy - Case 907\\ 13288
  Marseille Cedex 9, France} \email{anne.pichon@univ-amu.fr}
\begin{abstract}We describe a natural decomposition of a normal
  complex surface singularity $(X,0)$ into its ``thick'' and ``thin''
  parts. The former is essentially metrically conical, while the
  latter shrinks rapidly in thickness as it approaches the origin. The
  thin part is empty if and only if the singularity is metrically
  conical; the link of the singularity is then Seifert fibered. In
  general the thin part will not be empty, in which case it always
  carries essential topology. Our decomposition has some analogy with
  the Margulis thick-thin decomposition for a negatively curved
  manifold. However, the geometric behavior is very different; for
  example, often most of the topology of a normal surface singularity
  is concentrated in the thin parts.

  By refining the thick-thin decomposition, we then give a complete
  description of the intrinsic bilipschitz geometry of $(X,0)$ in
  terms of its topology and a finite list of numerical bilipschitz
  invariants.
\end{abstract}
\maketitle
 
\section{Introduction}

Lipschitz geometry of complex singular spaces is an intensively
developing subject.  In \cite{SS}, L. Siebenmann and D. Sullivan
conjectured that the set of Lipschitz structures is tame, i.e., the
set of equivalence classes of complex algebraic sets in $\C^n$,
defined by polynomials of degree less than or equal to $k$ is
finite. One of the most important results on Lipschitz geometry of
complex algebraic sets is the proof of this conjecture by T. Mostowski
\cite{mostowski} (generalized to the real setting by Parusi\'nski
\cite{parusinski1,parusinski}).  But so far, there exists no explicit
description of the equivalence classes, except in the case of complex
plane curves which was studied by Pham and Teissier
\cite{pham-teissier} and Fernandes \cite{fernandes}. They show that
the embedded Lipschitz geometry of plane curves is determined by the
topology. The present paper is devoted to the case of complex
algebraic surfaces, which is much richer.

Let $(X,0)$ be the germ of a normal complex surface singularity. Given
an embedding $(X,0)\subset (\C^n,0)$, the standard hermitian metric on
$\C^n$ induces a metric on $X$ given by arc-length of curves in $X$
(the so-called ``inner metric''). Up to bilipschitz equivalence this
metric is independent of the choice of embedding in affine space.

It is well known that for all sufficiently small $\epsilon>0$ the
intersection of $X$ with the sphere $S_\epsilon\subset \C^n$ about $0$
of radius $\epsilon$ is transverse, and the germ $(X,0)$ is therefore
``topologically conical,'' i.e., homeomorphic to the cone on its link
$X\cap S_\epsilon$ (in fact, this is true for any semi-algebraic
germ).  However, $(X,0)$ need not be ``metrically conical''
(bilipschitz equivalent to a standard metric cone). The first example
of a non-metrically-conical $(X,0)$ was given in \cite{BF08}, and the
examples in \cite{BFN1, BFN2, BFN4} then suggested that failure
of metric conicalness is common. In \cite{BFN4} it is also shown that
bilipschitz geometry of a singularity may not be determined by its
topology.

In those papers the failure of metric conicalness and differences in
bilipschitz geometry were determined by
local obstructions: existence of topologically non-trivial subsets
of the link of the singularity (``fast loops'' and ``separating
sets'') whose size shrinks faster than linearly as one approaches the
origin. 

In this paper we first describe a natural decomposition of the germ
$(X,0)$ into two parts, the \emph{thick} and the \emph{thin} parts,
such that the thin part carries all the ``non-trivial'' bilipschitz
geometry, and we later refine this to give a classification of the
bilipschitz structure. Our thick-thin decomposition is somewhat
analogous to the Margulis thick-thin decomposition of a negatively
curved manifold, where the thin part consists of points $x$ which lie
on a closed essential (i.e., non-nullhomotopic) loop of length
$\le 2\eta$ for some small $\eta$. A (rough) version of our thin part
can be defined similarly using essential loops in $X\setminus\{0\}$,
and length bound of the form $\le |x|^{1+c}$ for some small $c$. We
return to this in section \ref{sec:uniqueness}.
\begin{definition}[Thin]\label{def:thin} A semi-algebraic germ
  $(Z,0)\subset (\R^N,0)$ of pure dimension $k$ is \emph{thin} if its
  tangent cone $T_0Z$ has dimension strictly less than $k$.
\end{definition}
This definition only depends on the inner metric of $Z$ and not on the
embedding in $\R^N$. Indeed, instead of $T_0Z$ one can use the
metric tangent cone $\mathcal T_0Z$ of Bernig and Lytchak in the
definition, since it is a bilipschitz invariant for the inner metric
and maps finite-to-one to $T_0Z$ (see 
\cite{BL}). The metric tangent cone is discussed further in Section
\ref{sec:tangent cone}, where we show that it can be recovered from
the thick-thin decomposition.

``Thick'' is a generalization of
  ``metrically conical.'' Roughly speaking, a thick
algebraic set is obtained by slightly inflating a metrically conical
set in order that it can interface along its boundary
with thin parts. The precise definition is as follows:

\begin{definition}[Thick]\label{def:thick}
Let $B_\epsilon\subset
  \R^N$ denote the ball of radius $\epsilon$ centered at the origin,
  and $S_\epsilon$ its boundary.  A semi-algebraic germ $(Y,0)\subset
  (\R^N,0)$ is \emph{thick} if there exists $\epsilon_0>0$ and $K\ge
  1$ such that $Y\cap B_{\epsilon_0}$ is the union of subsets
  $Y_\epsilon$, $\epsilon\le\epsilon_0$ which are metrically conical
  with bilipschitz constant $K$ and satisfy the following properties
  (see Fig.~\ref{fig:1}):
  \begin{enumerate}
\item $Y_\epsilon\subset B_\epsilon$, $Y_\epsilon\cap
  S_\epsilon=Y\cap S_\epsilon$ and $Y_\epsilon$ is metrically
  conical as a cone on its link $Y\cap S_\epsilon$.
\item For $\epsilon_1<\epsilon_2$ we have $Y_{\epsilon_2}\cap
  B_{\epsilon_1}\subset Y_{\epsilon_1}$ and this embedding respects
  the conical structures.  Moreover, the difference $(
  Y_{\epsilon_1}\cap S_{\epsilon_1})\setminus (Y_{\epsilon_2}\cap
  S_{\epsilon_1})$ of the links of these cones 
  homeomorphic to $\partial(Y_{\epsilon_1}\cap
  S_{\epsilon_1})\times[0,1)$.
  \end{enumerate}
  \end{definition}

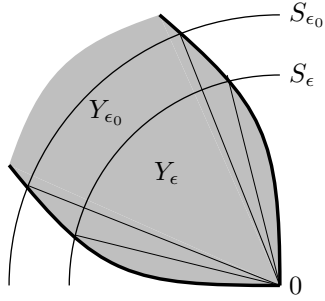
\begin{figure}[ht]
\centering
\begin{tikzpicture} 
 \node at(4,0)[right]{$0$};
 \node at(4,2.8)[right]{$ S_{\epsilon}$};
 \node at(4,3.6)[right]{$ S_{\epsilon_0}$};
 \path[ fill=lightgray] 
(4,0) .. controls (2,0) and (1.6,0.1)  .. (0.4,1.6) (0.4,1.6)--(2.4,3.6)
 (2.4,3.6)   .. controls (3.8,2.4)   and (4,2)  .. (4,0); 
 \path[ fill=lightgray]  (4,0)--(0.4,1.6).. controls (0.8,2.8) and 
(1.2,3.2)..(2.4,3.6)--cycle;
\draw[line width=1.3pt ]  (4,0) .. controls (2,0) and (1.6,0.1)  .. (0.4,1.6);
\draw[line width=1.3pt ]  (4,0) .. controls (4,2) and (3.8,2.4)  .. (2.4,3.6);
\draw[line width=0.5pt ]  (0.4,0) arc (180:90:3.6);
\draw[line width=0.5pt ]  (1.2,0) arc (180:90:2.8);
\draw[line width=0.25pt ] (4,0)--(0.65,1.33); 
\draw[line width=0.25pt ] (4,0)--(2.65,3.4); 
\draw[line width=0.25pt ] (4,0)--(1.28,0.67); 
\draw[line width=0.25pt ] (4,0)--(3.3,2.8); 
 \node at(2.5,1.5)[]{$ Y_{\epsilon}$};
 \node at(1.7,2.3)[]{$ Y_{\epsilon_0}$};
 \end{tikzpicture} 
  \caption{Thick germ}
  \label{fig:1}
\end{figure}

Clearly, a semi-algebraic germ cannot be both thick and thin. The
following proposition helps picture ``thinness''. Although it is well known,
we give a quick proof in Section \ref{sec:thin} for convenience.

Let $1< q\in \Q$. A \emph{$q$-horn neighborhood} of a semi-algebraic
germ $(A,0)\subset (\R^N,0)$ is  a set of the form $\{x\in \R^n\cap
B_\epsilon:d(x,A)\le c|x|^q\}$ for some $c>0$.   
\begin{proposition}\label{prop:thin nh} 
  Any semi-algebraic germ $(Z,0)\subset (\R^N,0)$ is contained in some $q$-horn
  neighborhood of its tangent cone $T_0 Z$. 
\end{proposition}
 
\noindent For example, the set $Z=\{(x,y,z)\in \R^3: x^2+y^2\le z^3\}$
gives a thin germ at $0$ since it is a $3$-dimensional germ whose
tangent cone is the $z$-axis.  The intersection $Z\cap B_\epsilon$
is contained in a closed $3/2$-horn neighborhood of the
$z$-axis.  The complement in $\R^3$ of this thin set is thick. 

\medskip 
For
any subgerm $(A,0)$ of $(\C^n,0)$ or $(\R^N,0)$ we write
$$A^{(\epsilon)}:=A\cap S_\epsilon\,\subset\, S_\epsilon\,.$$
In particular, when $A$ is semi-algebraic and $\epsilon$ is
sufficiently small, $A^{(\epsilon)}$ is the $\epsilon$-link of
$(A,0)$.
\begin{definition}[Thick-thin decomposition]\label{def:thick-thin}
  A \emph{thick-thin decomposition} of the normal complex surface germ
  $(X,0)$ is a decomposition of it as a union of germs of pure
  dimension $4$:
  \begin{equation}
    \label{eq:thick-thin}
    (X,0)=\bigcup_{i=1}^r(Y_i,0)\cup\bigcup_{j=1}^s(Z_j,0)\,,
  \end{equation}
  such
  that the  $Y_i\setminus\{0\}$ and $Z_j\setminus\{0\}$ are connected and:
  \begin{enumerate}
  \item Each $Y_i$ is thick and each $Z_j$ is thin.
  \item\label{it:1} The $Y_i\setminus\{0\}$ are pairwise disjoint and the
    $Z_j\setminus\{0\}$ are pairwise disjoint.
  \item\label{it:2} If $\epsilon_0$ is chosen small enough that
    $S_\epsilon$ is transverse to each of the germs $(Y_i,0)$
    and $(Z_j,0)$ for $\epsilon\le \epsilon_0$, then 
    $X^{(\epsilon)}=\bigcup_{i=1}^rY^{(\epsilon)}_i\cup\bigcup_{j=1}^sZ^{(\epsilon)}_j$
    decomposes the $3$-manifold $X^{(\epsilon)}\subset
    S_\epsilon$ into connected submanifolds with boundary,
    glued along their boundary components.
\end{enumerate}
We call the links $Y_i^{(\epsilon)}$ and $Z_j^{(\epsilon)}$ of the
thick and thin pieces 
\emph{thick and thin zones} of the link $X^{(\epsilon)}$.
\end{definition}
\begin{definition}\label{def:minimal}
  A thick-thin decomposition is \emph{minimal} if
  \begin{enumerate}
  \item\label{it:min1} the tangent cone of
  its thin part $\bigcup_{j=1}^sZ_j$ is contained in the
  tangent cone of the thin part 
  of any other thick-thin decomposition and
\item the number $s$ of its thin pieces is minimal among thick-thin
  decompositions satisfying \eqref{it:min1}.
  \end{enumerate}
\end{definition}
The following  theorem expresses the existence and uniqueness of
a minimal thick-thin decomposition for a normal complex surface germ $(X,0)$.
\begin{theorem}\label{th:uniqueness}
  A minimal thick-thin decomposition of $(X,0)$ exists. For any two
  minimal thick-thin decompositions of $(X,0)$ there exists $q>1$ and
  a homeomorphism of the germ $(X,0)$ to itself which takes the one
  decomposition to the other and moves each $x\in X$ distance at most
  $|x|^q$.
\end{theorem}

The homeomorphism in the above theorem is not necessarily bilipschitz,
but the bilipschitz classification which we describe later leads to a
``best'' minimal thick-thin decomposition, which is unique up to
bilipschitz homeomorphism.
 
\begin{theorem}[Properties]\label{th:main in intro}A
    minimal thick-thin decomposition of $(X,0)$ as in equation
    \eqref{eq:thick-thin} satisfies $r\ge 1$, $s\ge 0$ and has the
  following properties for $0<\epsilon\le \epsilon_0$:
  \begin{enumerate}
  \item\label{it:thick} Each thick zone $Y^{(\epsilon)}_i$ is a
    Seifert fibered manifold.
  \item\label{it:thin} Each thin zone $Z^{(\epsilon)}_j$ is a graph
    manifold (union of Seifert manifolds glued along boundary
    components) and not a solid torus.
\item\label{it:fibration} There exist constants $c_j>0$
    and $q_j>1$ and fibrations $\zeta_j^{(\epsilon)}\colon
    Z^{(\epsilon)}_j\to S^1$ depending smoothly on $\epsilon\le
    \epsilon_0$ such that the fibers $\zeta_j^{-1}(t)$ have diameter at
    most $c_j\epsilon^{q_j}$ (we call these fibers the \emph{Milnor
      fibers} of $Z_j^{(\epsilon)}$).
  \end{enumerate}
\end{theorem}
The minimal thick-thin decomposition is constructed in Section
\ref{sec:thick-thin}. Its minimality and uniqueness are proved in
Section \ref{sec:uniqueness}.

We will take a resolution approach to construct the
thick-thin decomposition, but another way of constructing it
is as follows. Recall (see \cite{Sn, LT}) that a line $L$ tangent
to $X$ at $0$ is \emph{exceptional} if the limit at $0$ of tangent
planes to $X$ along a curve in $X$ tangent to $L$ at $0$ depends on
the choice of this curve. Just finitely many tangent lines to $X$ at
$0$ are exceptional. To obtain the thin part one intersects
$X\setminus\{0\}$ with a $q$-horn disk-bundle neighborhood of each
exceptional tangent line $L$ for $q>1$ sufficiently small and then
discards any ``trivial'' components of these intersections (those
whose closures are locally just cones on solid tori; such trivial
components arise also in our resolution approach, and showing that
they can be absorbed into the thick part takes some effort, see
section \ref{sec:thick}).

\smallskip In \cite{BFN1} a \emph{fast loop} is defined as a family of
closed curves in the links $X^{(\epsilon)}:=X\cap S_{\epsilon}$,
$0<\epsilon\le \epsilon_0$, depending continuously on $\epsilon$,
which are not homotopically trivial in $X^{(\epsilon)}$ but whose
lengths are proportional to $\epsilon^k$ for some $k>1$, and it is
shown that fast loops are obstructions to metric
conicalness\footnote{We later call these \emph{fast loops of the first
    kind}, since in section \ref{fast loops} we define a related
  concept of \emph{fast loop of the second kind} and show these also
  obstruct metric conicalness.}

In Theorem \ref{th:thin fast loops} we show
\begin{theorem*}[\ref{th:thin fast loops}]
  Each thin piece $Z_j$ contains fast loops. 
  In fact, each boundary component of its Milnor fiber gives a fast
  loop.
\end{theorem*}
\begin{corollary}\label{cor:conical} The following are equivalent, and
  each implies that the link of $(X,0)$ is Seifert fibered:
  \begin{enumerate}
  \item   $(X,0)$ is metrically conical;
  \item $(X,0)$ has no fast loops;
\item $(X,0)$ has no thin piece (so it consists of a single thick piece).
  \end{enumerate}
\end{corollary}

\medskip\noindent{\bf Bilipschitz classification}. We will give a
complete classification of the geometry of $(X,0)$ up to bilipschitz
equivalence, based on a refinement of the thick-thin decomposition. We
will describe this refinement in terms of the decomposition of the
link $X^{(\epsilon)}$.

We first refine the decomposition
$X^{(\epsilon)}=\bigcup_{i=1}^rY^{(\epsilon)}_i\cup\bigcup_{j=1}^sZ^{(\epsilon)}_j$
by decomposing each thin zone $Z_j^{(\epsilon)}$ into its JSJ
decomposition (minimal decomposition into Seifert fibered manifolds
glued along their boundaries \cite{JSJ,NS}), while leaving the thick
zones $Y_i^{(\epsilon)}$ as they are.  We then thicken some of the
gluing tori of this refined decomposition to collars $T^2\times I$, to
add some extra ``annular'' pieces (the choice where to do this is
described in Section \ref{sec:decomp}).  At this point we have
$X^{(\epsilon)}$ glued together from various Seifert fibered manifolds
(in general not the minimal such decomposition).

Let $\Gamma_0$ be the decomposition graph for this, with a vertex for
each piece and edge for each gluing torus, so we can write this
decomposition as  
\begin{equation}
   \label{eq:link-decomposition}
   X^{(\epsilon)}=\bigcup_{\nu\in V(\Gamma_0)} M_\nu^{(\epsilon)} \,,
\end{equation}
where $V(\Gamma_0)$  is the vertex set of $\Gamma_0$.
\begin{theorem}[Classification Theorem]\label{th:classification}
  The bilipschitz geometry of $(X,0)$ determines and is uniquely
  determined by the following data:
  \begin{enumerate}
  \item\label{it:1.9.1} The decomposition of $X^{(\epsilon)}$ into
    Seifert fibered manifolds as described above, refining the
    thick-thin decomposition;
  \item\label{it:1.9.2} for each thin zone $Z_j^{(\epsilon)}$, the
    homotopy class of the foliation by fibers of the fibration
    $\zeta_j^{(\epsilon)}\colon Z_j^{(\epsilon)}\to S^1$ (see Theorem
    \ref{th:main in intro}\,\eqref{it:fibration});
  \item\label{it:1.9.3} for each vertex $\nu\in V(\Gamma_0)$, a
    rational weight $q_\nu\ge 1$ with $q_\nu=1$ if and only if
    $M_\nu^{(\epsilon)}$ is a $Y_i^{(\epsilon)}$ (i.e., a thick zone)
    and with $q_\nu\ne q_{\nu'}$ if $\nu$ and $\nu'$ are adjacent
    vertices.
  \end{enumerate}
\end{theorem}
In item \eqref{it:1.9.2} we ask for the foliation by fibers rather
than the fibration itself since we do not want to distinguish
fibrations $\zeta\colon Z\to S^1$ which become equivalent after
composing each with a covering maps $S^1\to S^1$. Note that item
\eqref{it:1.9.2} describes discrete data, since the foliation is
determined up to homotopy by a primitive element of
$H^1(Z_j^{(\epsilon)};\Z)$ up to sign.

The data of the above theorem can also be conveniently encoded by
adding the $q_\nu$'s as weights on a suitable decorated resolution
graph. We do this in Section \ref{sec:examples}, where we compute
various examples.

The proof of Theorem \ref{th:classification} is in terms of a
canonical ``bilipschitz model''
\begin{equation}
   \label{eq:model-decomposition}
   \widehat X=\bigcup_{\nu\in V(\Gamma_0)}  \widehat M_\nu\cup 
\bigcup _{\sigma\in  E(\Gamma_0)}  \widehat A_\sigma\,,
 \end{equation}
with $\widehat X\cong X\cap B_\epsilon$ (bilipschitz) and where each 
$\widehat A_\sigma$ is a collar (cone on a toral annulus $T^2\times I$)
while each $\widehat M_\nu$ is homeomorphic to the cone on
$M_\nu^{(\epsilon)}$. The  pieces carry Riemannian metrics
determined by the $q_\nu$'s and the foliation data of the theorem;
these metrics are global versions of the local metrics used by Hsiang
and Pati \cite{hsiang-pati} and Nagase \cite{nagase}. On a piece
$\widehat M_\nu$ the metric is what Hsiang and Pati call a ``Cheeger
type metric'' (locally of the form
$dr^2+r^2d\theta^2+r^{2q_\nu}(dx^2+dy^2)$; see Definitions
\ref{def:p1}, \ref{def:p3}).  On a piece $\widehat A_\sigma$ it has
a Nagase type metric as described in Nagase's correction to
\cite{hsiang-pati} (see Definition \ref{def:p2}).

Mostovski's work mentioned earlier is based on a construction of
Lipschitz trivial stratifications. Our approach is different in that
we decompose the germ $(X,0)$ using the carrousel theory of D. T. L\^e
(\cite{Le}, see also \cite{le-michel-weber}) applied to the
discriminant curve of a generic plane projection of the
surface. However, our work has some similarities with Mostovski's
(loc.~cit., see also \cite{M1}) in the sense that the geometry near
the polar curves also plays an important role, in particular the
subgerms where the family of polar curves accumulates while one varies
the direction of projections (Propositions \ref{le:very thin} and
\ref{prop:polar wedge}).

A thick-thin decomposition exists also for
  higher-dimensional germs, and we conjecture with Alberto Verjovsky
  that it can be made canonical.  It is the rigidity of topology in
  dimension 3, linked to the nontriviality of fundamental groups in
  this dimension, which enables us to get strong results for surfaces.
  The less rigid topology in higher dimensions makes it is harder to
  pin down the ``trivial'' parts mentioned earlier which can be
  absorbed into the thick zones, and there are similar issues in
  determining boundaries between the pieces in a full bilipschitz
  classification. 

\smallskip\noindent{\bf Acknowledgments.}  We are very
  grateful to the referee for insightful comments which corrected an
  error in the paper and improved it in other ways, and to
Adam Parusi\'nski, Jawad Snoussi, Don O'Shea,
Bernard Teissier, Guillaume Valette, and Alberto
Verjovsky for useful conversations. Neumann was supported by NSF grants
DMS-0905770  and DMS-1206760. Birbrair was supported by CNPq grants 201056/2010-0 and
300575/2010-6. Pichon was supported by the ANR project SUSI 12-JS01-0002-01.  We are also
grateful for the hospitality/support of the following institutions:
Jagiellonian University (B), Columbia University (B,P), Institut de
Math\'ematiques de Luminy, Universit\'e d'Aix-Marseille, Instituto do
Mil\'enio (N), IAS Princeton, CIRM petit groupe de travail (B,N,P),
Universidade Federal de Ceara, CRM Montr\'eal (N,P).

\section{Construction  of the thick-thin
  decomposition}\label{sec:thick-thin}

Let $(X,0)\subset (\C^n,0)$ be a normal surface germ.  In this
section, we explicitly describe the thick-thin decomposition for a
normal complex surface germ $(X,0)$ in terms of a suitably adapted
resolution of $(X,0)$.

Let $\pi\colon(\widetilde X,E)\to (X,0)$ be the minimal resolution with
the following properties:
\begin{enumerate}
\item\label{r1} It is a good resolution, i.e., the
exceptional divisors are smooth and meet transversely, at most two at
a time.
\item\label{r2} It has no basepoints for a general linear system of
  hyperplane sections, i.e., $\pi$ factors through the
  normalized blow-up of the origin.  An exceptional curve intersecting
  the strict transforms of the generic members of a general linear
  system will be called an \emph{\L-curve}.
\item\label{r3} No two  \L-curves intersect.  
\end{enumerate}
 This is achieved by starting with a minimal good resolution, then
blowing up to resolve any basepoints of a general system  of hyperplane
sections, and finally blowing up any intersection point between  \L-curves.

Let $\Gamma$ be the resolution graph of the above resolution. A vertex
of $\Gamma$ is called a \emph{node} if it has valency  $\ge 3$ or
represents a curve of genus $>0$ or represents an \L-curve. If a
node represents an \L-curve it is called an \L-node, otherwise a
\T-node. By the previous paragraph, \L-nodes cannot be adjacent to
each other.

The subgraphs of $\Gamma$ resulting by removing the \L-nodes and adjacent
edges from $\Gamma$ are called the \emph{Tjurina components} of
$\Gamma$ (following \cite[Definition III.3.1]{spivakovsky}), so \T-nodes are
precisely the nodes of $\Gamma$ that are in Tjurina components.

A \emph{string} is a
connected subgraph of $\Gamma$ containing no nodes.
A \emph{bamboo} is a string ending in a vertex of
valence $1$.

For each exceptional curve $E_\nu$ in $E$ let $N(E_\nu)$ be a small
closed tubular neighborhood.  For any subgraph $\Gamma'$ of $\Gamma$
define (see Fig.~\ref{fig:N}):
$$N(\Gamma'):= \bigcup_{\nu\in\Gamma'}N(E_\nu)\quad\text{and}\quad
\Nn(\Gamma'):= 
\overline{N(\Gamma)\setminus \bigcup_{\nu\notin \Gamma'}N(E_\nu)}\,.$$
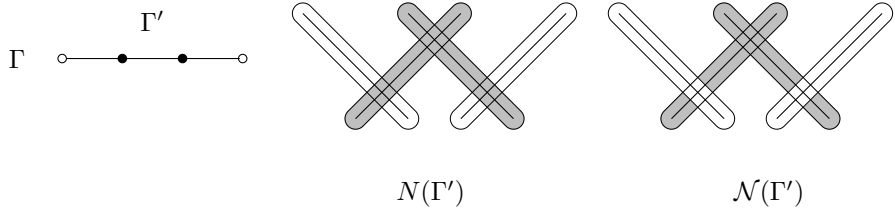
\begin{figure}[ht]\label{fig:N}
\centering
\begin{tikzpicture}
  \node[ ](a)at (0.8,1.3){$\Gamma'$}; \node[ ](b)at
  (-1,0.8){$\Gamma$}; \node[ ](c)at (4.5,-1){$N(\Gamma')$}; \node[
  ](d)at (9,-1){${\cal N}(\Gamma')$};\draw[
  scale=0.8,xshift=-0.5cm,yshift=1cm,thin] (0+2pt,0)--(2.95,0);\draw
  [ scale=0.8,xshift=-0.5cm,yshift=1cm ](0,0)circle(2pt);
 \draw[scale=0.8, xshift=-0.5cm,yshift=1cm,fill] (1,0)circle(2pt);
 \draw[ scale=0.8,xshift=-0.5cm,yshift=1cm, fill] (2,0)circle(2pt);
 \draw[scale=0.8,xshift=-0.5cm,yshift=1cm, thin](3,0)circle(2pt);
 \draw[scale=0.7,xshift=5cm,fill, lightgray](2,2)+(-45:0.2) arc
  (-45:135:0.2)--++(-2,-2) arc (135:315:0.2)--++(2,2);
 \draw[scale=0.7,xshift=7cm, fill, lightgray](1,0)+(-135:0.2) arc
  (-135:45:0.2)--++(-2,2) arc (45:225:0.2)--++(2,-2);
 \draw[scale=0.7,xshift=5cm,very thin](2,2)+(-45:0.2) arc
  (-45:135:0.2)--++(-2,-2) arc (135:315:0.2)--++(2,2);
 \draw[scale=0.7,xshift=7cm,very thin](2,2)+(-45:0.2) arc
  (-45:135:0.2)--++(-2,-2) arc (135:315:0.2)--++(2,2);
 \draw[scale=0.7,xshift=5cm,very thin](1,0)+(-135:0.2) arc
  (-135:45:0.2)--++(-2,2) arc (45:225:0.2)--++(2,-2);
 \draw[scale=0.7,xshift=7cm, very thin](1,0)+(-135:0.2) arc
  (-135:45:0.2)--++(-2,2) arc (45:225:0.2)--++(2,-2);
 \draw[scale=0.7,xshift=5cm](0,0)--(2,2);
 \draw[scale=0.7,xshift=5cm](-1,2)--(1,0);
 \draw[scale=0.7,xshift=5cm](1,2)--(3,0);
 \draw[scale=0.7,xshift=5cm](2,0)--(4,2);
 \draw[scale=0.7,xshift=11cm,fill, lightgray](2,2)+(-45:0.2) arc
  (-45:135:0.2)--++(-2,-2) arc (135:315:0.2)--++(2,2);
 \draw[scale=0.7,xshift=13cm, fill, lightgray](1,0)+(-135:0.2) arc
  (-135:45:0.2)--++(-2,2) arc (45:225:0.2)--++(2,-2);
 \draw[scale=0.7,xshift=11cm,fill,
  white](0.5,0.5)--++(-45:0.2)--++(45:0.2)
  --++(135:0.4)--++(225:0.4)--++(-45:0.4)--++(45:0.2);
 \draw[scale=0.7,xshift=13cm,fill,
  white](0.5,0.5)--++(-45:0.2)--++(45:0.2)
  --++(135:0.4)--++(225:0.4)--++(-45:0.4)--++(45:0.2);
 \draw[scale=0.7,xshift=11cm,very thin](2,2)+(-45:0.2) arc
  (-45:135:0.2)--++(-2,-2) arc (135:315:0.2)--++(2,2);
 \draw[scale=0.7,xshift=13cm,very thin](2,2)+(-45:0.2) arc
  (-45:135:0.2)--++(-2,-2) arc (135:315:0.2)--++(2,2);
 \draw[scale=0.7,xshift=11cm,very thin](1,0)+(-135:0.2) arc
  (-135:45:0.2)--++(-2,2) arc (45:225:0.2)--++(2,-2);
 \draw[scale=0.7,xshift=13cm, very thin](1,0)+(-135:0.2) arc
  (-135:45:0.2)--++(-2,2) arc (45:225:0.2)--++(2,-2);
 \draw[scale=0.7,xshift=11cm](0,0)--(2,2);
 \draw[scale=0.7,xshift=11cm](-1,2)--(1,0);
 \draw[scale=0.7,xshift=11cm](1,2)--(3,0);
 \draw[scale=0.7,xshift=11cm](2,0)--(4,2);
\end{tikzpicture} 
\caption{$N(\Gamma')$ and $\Nn(\Gamma')$ for the $A_4$ singularity}
\end{figure}

In the Introduction we used standard $\epsilon$-balls to state our
results, but in practice it is often more convenient to work with a
different family of Milnor balls.  For example, one can use, as in
Milnor \cite{milnor}, the ball of radius $\epsilon$ at the origin, or
the balls with corners introduced by K\"ahler \cite{kaehler}, Durfee
\cite{durfee} and others. In our proofs it will be convenient to use
balls with corners, but it is a technicality to deduce the results for
round Milnor balls. We will define the specific family of balls we use
in Section \ref{sec:balls}. We denote it again by $B_{\epsilon}$,
$0<\epsilon\le \epsilon_0$ and put $S_\epsilon:=\partial B_\epsilon$.

\begin{definition}[Thick-thin decomposition]
  Assume $\epsilon_0$ is sufficiently small that $\pi^{-1}(X\cap
  B_{\epsilon_0})$ is included in $N(\Gamma)$. Let $\Gamma_1 ,
  \ldots, \Gamma_s$ denote the  Tjurina components of $\Gamma$ which are not
  bamboos, and by $\Gamma'_1,\ldots, \Gamma'_r$ the maximal connected
  subgraphs in $\Gamma \setminus \bigcup_{j=1}^s \Gamma_j$.
 
  For each each $i=1,\ldots,r$, define
$$Y_i:=\pi(N(\Gamma'_j))\cap  B_{\epsilon_0},$$
and for each $j=1,\ldots,s$, define
$$Z_j:=\pi(\Nn(\Gamma_j))\cap  B_{\epsilon_0}.$$
 
 Notice that   each $\Gamma'_i$ consists
of an \L-node  and any attached bamboos.  So the $Y_i$ are in
one-one correspondence with the \L-nodes.

The $Y_i$ are the \emph{thick pieces} and the $Z_j$ are the \emph{thin
  pieces}.
\end{definition} 

By construction, the decomposition $(X,0)=\bigcup(
Z_j,0)\cup\bigcup (Y_i,0)$ satisfies items \eqref{it:1} and \eqref{it:2}
of Definition \ref{def:thick-thin} and items \eqref{it:thick}
and \eqref{it:thin} of Theorem \ref{th:main in intro}.  Item
\eqref{it:fibration} of Theorem \ref{th:main in intro} and the
thinness of the $Z_j$ are proved in Section \ref{sec:thin}. The
thickness of $Y_i$ is proved in Section \ref{sec:thick}.

\section{Polar curves}\label{sec:polars}
Let $(X,0)\subset (\C^n,0)$ be a normal surface germ.  In this
section, we prove two independent results on polar curves of linear
projections $X \to \C^2$ which will be used in the sequel. We first
need to introduce some classical material.

Let $\cal D$ be a $(n-2)$-plane in $\C^n$ and let $\ell_{\cal D}
\colon \C^n \to \C^2$ be the linear projection $\C^n \to \C^2$ with
kernel $\cal D$. We restrict ourselves to those $\cal D$ in the
Grassmanian $\grassman(n-2,\C^n)$ such that the restriction
$\ell_{\cal D}{\mid}\colon(X,0)\to(\C^2,0)$ is finite.
The \emph{polar curve}
$\Pi_{\cal D}$ of $(X,0)$ for the direction $\cal D$ is the closure in
$(X,0)$ of the singular locus of the restriction of $\ell_{\cal D} $
to $X \setminus \{0\}$. The \emph{discriminant curve} $\Delta_{\cal D}
\subset (\C^2,0)$ is the image $\ell_{\cal D}(\Pi_{\cal D})$ of the polar
curve $\Pi_{\cal D}$.
 
There exists an open dense subset $\Omega \subset \grassman(n-2,\C^n)$
such that the germs of curves $(\Pi_{\cal D},0), \cal D \in \Omega$
are equisingular in terms of strong simultaneous resolution and such
that the discriminant curves $\Delta_{\cal D} = \ell_{\cal
  D}(\Pi_{\cal D})$ are reduced and no tangent line to $\Pi_{\cal D}$
at $0$ is contained in $\cal D$ (\cite[(2.2.2)]{LT0} and
\cite[V. (1.2.2)]{teissier}).

The condition $\Delta_{\cal D}$ reduced means that any
$p\in\Delta_{\cal D}\setminus\{0\}$ has a neighborhood $U$ in $\C^2$
such that one component of $(\ell_{\cal D}|_X)^{-1}(U)$ maps by a
two-fold branched cover to $U$ and the other components map
bijectively.
 
\begin{definition} 
The projection $\ell_{\cal D} \colon \C^n \to \C^2$
  is \emph{generic} for $(X,0)$ if $\cal D \in \Omega$.
\end{definition}

Let $\lambda \colon X\setminus\{0\} \to \grassman(2,\C^n)$ be the
  map which maps $x \in X\setminus\{0\}$ to the tangent plane
  $T_xX$. The closure $\check X$ of the graph of $\lambda$ in $X
  \times \grassman(2,\C^n)$ is a reduced analytic surface. By
  definition, the \emph{Nash modification} of $(X,0)$ is the induced
  morphism $\nu\colon \check X \to X$.

\begin{lemma}[{\cite[Part III,  Theorem
    1.2]{spivakovsky}}, {\cite[Section 2]{GS}}]\label{le:nash} 
A resolution of $(X,0)$ factors through Nash modification if and
  only if it has no base points for the family of polar curves.\qed
\end{lemma}

Let us fix $\cal D \in \Omega$. We suppress the subscript
$\cal D$ and note simply $\ell$ for $\ell_{\cal D}$ and $\Pi$ and
$\Delta$ for the polar and discriminant curves of $\ell|_X$.  The
\emph{local bilipschitz constant of $\ell|_X$} is the map
$K\colon X\setminus \{0\}\to \R\cup\{\infty\}$ defined as follows. It is
infinite on the polar curve and at a point $p\in X\setminus \Pi$ it is
the reciprocal of the shortest length among images of unit vectors in
$T_{p} X$ under the projection $d\ell\colon T_{p} X\to \C^2$.

\begin{proposition}\label{le:very thin} 
  Let ${\pi'} \colon {\widetilde X}' \to X$ be a resolution of $X$
  which factors through Nash modification.  Let  $\Pi^*$ denote  the
  strict transform of the polar curve $\Pi$ by $\pi'$.
 Given any neighborhood $U$ of\/ ${\Pi}^*\cap
  ({\pi'})^{-1}(B_\epsilon\cap X)$ in ${\widetilde X}'\cap
  ({\pi'})^{-1}(B_\epsilon\cap X)$, the local bilipschitz constant $K$
  is bounded on $(B_\epsilon \cap X)\setminus \pi'(U)$.
\end{proposition}

\begin{proof}
  Let $\sigma \colon {\widetilde X}' \to \grassman(2,\C^n)$ be the map
  induced by the projection $p_2 \colon \check X \subset X \times
  \grassman(2,\C^n) \to \grassman(2,\C^n)$ and let $\alpha \colon
  \grassman(2,\C^n) \to \R \cup \{\infty\}$ be the map sending $H \in
  \grassman(2,\C^n)$ to the bilipschitz constant of the restriction
  $\ell|_{H}\colon H \to \C^2$.  The map $\alpha \circ \sigma$
  coincides with $K \circ {\pi'}$ on ${\widetilde X}' \setminus
  {\pi'}^{-1}(0)$ and takes finite values outside ${\Pi}^*$. The map
  $\alpha \circ \sigma$ is continuous and therefore bounded on the
  compact set ${\pi'}^{-1}(B_{\epsilon}) \setminus U$.
\end{proof}
 In the rest of the section, we consider a branch $\Delta_0$
  of the discriminant curve $\Delta$ and the component $\Pi_0$ of the
  polar of $\ell$ such that $\ell(\Pi_0)=\Delta_0$. We will study the
  behavior of $\ell$ on a suitable zone $A$ in $X$ containing
  $\Pi_0$, outside of which $\ell$ is a local bilipschitz homeomorphism. 

We choose coordinates in $\C^2$ so that $\Delta_0$ is
not tangent to the $y$-axis. Then
$\Delta_0$ admits a Puiseux series expansion 
$$y=\sum_{j\ge 1} a_jx^{{p_j}}\in \C\{x^{\frac1N}\}, \quad \text{ with
  }p_j\in \Q,\quad
1\le p_1<p_2<\cdots\,.$$ 
Here $N=\lcm_{j\ge 1}\denom(p_j)$, where ``denom'' means denominator.

For $K_0 \geq 1$, set $B_{K_0}:=\bigl\{ p\in X\cap
(B_\epsilon\setminus \{0\}):K(p)\ge K_0\bigr\}\,,$ and let
$B_{K_0}(\Pi_0)$ denote the closure of the connected component of $
B_{K_0}\setminus \{0\}$ which contains $\Pi_0 \setminus\{0\}$. We set
$N_{K_0}(\Delta_0) = \ell(B_{K_0}(\Pi_0))$.

\begin{proposition}[Polar Wedge Lemma]\label{prop:polar wedge}~
\begin{enumerate}
\item\label{it:LZ1} There exists $k\ge1$ such that if $s:=p_k$ then
  for any $\alpha>0$ there is $K_0\ge 1$ such that $N_{K_0}(\Delta_0)$
  is contained in the set
$$B=\bigl\{(x,y):\Bigl|y-\sum_{j\ge1}a_jx^{{p_j}}
\Bigr|\le\alpha|x|^{s}\bigr\}\,.$$
We call the largest such $s$ the \emph{contact exponent} of $\Delta_0$.
\item \label{it:LZ2}
Let $A_0$ be the closure of the component of $\ell^{-1}(B)
  \setminus \{0\}$ which contains $\Pi_0$. Then up to bilipschitz
  equivalence $A_0$ is a topological cone on a solid torus, $([0,\epsilon]\times S^1\times
  D^2)/(\{0\}\times S^1\times D^2)$, equipped with the metric
  $dr^2+r^2d\theta^2+r^{2s}g$, where $g$ is the standard metric on the
  unit disk. We call such an $A_0$ a \emph{polar wedge}.
 \end{enumerate}
\end{proposition}

\begin{remark*}
 Note that in part \eqref{it:LZ1} $B$ could be replaced by the set
 $$B'=\bigl\{(x,y):\Bigl|y-\sum_{j\ge1,p_j\le s}a_jx^{{p_j}}
 \Bigr|\le\alpha|x|^{s}\bigr\}\,,$$ 
 since truncating higher order terms
 does not change the bilipschitz geometry. Up to bilipschitz
 equivalence this does not change $A_0$ in part \eqref{it:LZ2} either.
%
\end{remark*}
\begin{proof}[Proof of Proposition \ref{prop:polar wedge}]
  We are considering the germ $(X,0)$, so in this proof all subsets of
  $X$ or $\widetilde X'$ are implicitly intersected with $B_\epsilon$
  or $(\pi')^{-1}(B_\epsilon)$ for some sufficiently small $\epsilon$.

  According to Proposition \ref{le:very thin}, for each neighborhood
  ${A}^*$ of $\Pi_0^*$ in $\widetilde X'$ there exists $K_0$ such
  that $B_{K_0}(\Pi_0) \subset \pi'(A^*)$.
  We first construct such an $A^*$ as the union of a family of
  disjoint strict transforms of components  $\Pi^*_{0,\cal
    D_t}$ of polars  $\Pi^*_{\cal
    D_t}$ parametrized by $t$ in a neighborhood of $0$ in $\C$, and
  with $\cal D_0=\cal D$. So $\Pi_0=\Pi_{0,\cal D_0}$.  Let $E\subset
  \pi'^{-1}(0)$ be the exceptional curve with $E\cap \Pi_0^*\ne
  \emptyset$.

Let $\sigma\colon \widetilde X'\to \grassman(2,\C^n)$ be as in the
  proof of Proposition \ref{le:very thin} and let $U$ denote a small
  neighborhood of $T:=\sigma(E\cap \Pi_0^*)$ in
  $\grassman(2,\C^n)$. We first assume $n=3$ so
  $\grassman(2,\C^n)=\grassman(2,\C^3)\cong P^2\C$. Choose any $T'\in
  \grassman(2,\C^3)\setminus U$ so that $T'\subset \C^3$ contains
  $\cal D$.  The line $L$ in $\grassman(2,\C^3)$ through $T$ and $ T'$
  is the set of $2$-planes in $\C^3$ containing the line $\cal D$, so
  its inverse image under $\sigma$ is exactly $\Pi^*$. Now consider
  the pencil of lines $L_t$ through $T'$, parametrized so
  $L_0=L$. Each $L_t$ is the set of $2$-planes containing some line
  $\cal D_t$. The family of inverse images of the $L_t$ which
  intersect $U$ is a family $\{\Pi^*_{\cal D_t}\}$ of disjoint strict
  transforms of polar components foliating an open neighborhood of
  $\Pi^*$.

  If $n\ge 3$ we choose an $(n-3)$-dimensional subspace $W\subset
  \C^n$ transverse to $T$. Shrinking $U$ if necessary, we can assume
  that $W$ is transverse to every $T''\in U$. Let $\grassman(2,\C^n;W)$ denote the set  of $2$-planes in $\C^n$ transverse to
  $W$, so the projection $p\colon \C^n\to \C^n/W$ induces a map
  $p'\colon \grassman(2,\C^n;W) \to \grassman(2,\C^n/W)\cong
  P^2\C$. We again consider the pencil of lines $L_t$ in
  $\grassman(2,\C^n/W)\cong P^2\C$ through some point outside
  $p'(U)$. The family of inverse images by $p'\circ\sigma$ of those of
  these lines which intersect $p'(U)$ is again a family of disjoint
  strict transforms of polar components foliating an open neighborhood
  of $\Pi^*$.  The polar corresponding to $L_t$ is the polar for the
  projection with kernel $\cal D_t$, where $\cal D_t/W\subset \C^n/W$
  is again the common line in the family of $2$-planes $L_t\subset
  \grassman(2,\C^n/W)$.  Indeed, for any $2$-plane $T'$ in $\C^n$
  transverse to $W$, the image $p(T')$ contains $\cal D_t/W$ if and
  only if $T'$ intersects $\cal D_t$ nontrivially.

  Consider now the neighborhood $\bigcup_{t\in V}\Pi_{\cal D_t}^*$ of
  $\Pi^*$ where $V$ is a small closed disk in $\C$ centered at $0$. We
  denote by $A^*$ the connected component of $\bigcup_{t\in
    V}\Pi_{\cal D_t}^*$ which contains $\Pi_0^*$. Then $A^* =
  \bigcup_{t\in V}\Pi_{0,\cal D_t}^*$, where $\Pi_{0,\cal D_t}$ is a
  branch of $\Pi_{\cal D_t}$, and $\Pi_{0,\cal D_0}=\Pi_0$.  We write
  $A:=\pi'(A^*)=\bigcup_{t\in V}\Pi_{0,\cal D_t}$.

The curves $\ell(\Pi_{0,\cal D_t})$ for $t\in V$  have Puiseux
  expansions 
 $$y=\sum_{j\ge 1} a_j(t)x^{{p_j}}\in \C\{x^{\frac1N}\}$$
 where $a_j(t)\in\C\{\!\{t\}\!\}$. The contact exponent $s$ is the
 first $p_j$ for which the coefficient $a_j(t)$ is non-constant.  Part
 \eqref{it:LZ1} of the proposition then follows.
 \begin{remark*}
In fact, according to the proof of Lemme 1.2.2~ii) in Teissier
  \cite[p.~462]{teissier}, the family of plane curves $\ell_{\cal
    D}(\Pi_{\cal D'})$ parametrized by $(\cal D,\cal D') \in \Omega
  \times \Omega$ is equisingular on a Zariski open neighborhood of the
  diagonal (a more explicit proof for hypersurfaces is found in
  Brian\c con-Henry \cite[Theorem 3.7]{BH}).  It follows that we can
  choose $V$ so that the curves  $\ell(\Pi_{0,\cal D_t})$ for $t\in V$
  form an equisingular family of plane curves.
 \end{remark*}
To prove part \eqref{it:LZ2} we first choose coordinates
$(z_1,\dots,z_n)$ in $\C^n$ and $(x,y)$ in $\C^2$ so that $\ell$ is
the projection $(x,y)=(z_1,z_2)$. We may assume $z_1$ and $z_2$ are
generic linear forms for $X$.  The multiplicity of $z_1$ along the
exceptional curve $E$ is $N$.  Let $(u,v)$ be local coordinates
centered at $\Pi_{0,\cal D_0}^*\cap E$ such that $v=t$ is the local
equation for $\Pi_{0,\cal D_t}^*$ and $z_1=u^N$. Then $z_2$ has the
form
  $$z_2 = u^Nf_0(u) + u^{Ns} \sum_{i\geq 1}v^i f_i(u),$$
  where $f_k(u)\in \C\{\!\{u\}\!\}$ for $k\ge1$ (and
  $u^Nf_0(u)=\sum_{j}a_j(0)u^{Np_j}$ in our earlier notation).

  Now, $\ell\circ\pi$ has $\Pi_{0,\cal D_0}^*\cup \{u=0\}$ as
  critical locus. The jacobian of $\ell\circ\pi$ is
$$J(\ell\circ \pi)(u,v)=
\begin{pmatrix}
  Nu^{N-1}&0\\\star&u^{Ns}(f_1(u)+2vf_2(u)+\cdots)
\end{pmatrix}\,,
$$
so $\Pi_{0,\cal D_0}^*\cup \{u=0\}$ has equation $Nu^{N+Ns-1}g(u,v)=0$
where $g(u,v)=f_1(u)+2vf_2(u)+\cdots$. Since $v=0$ is the equation of
$\Pi_{0,\cal D_0}^*$ this implies $f_1(u)=0$ and $f_2(0)\ne 0$. So
$g(u,v)=vh_2(u,v)$ with $h_2(u,v)=2f_2(u)+3vf_3(u)+\cdots$ a unit in
$\C\{\!\{u,v\}\!\}$. Summarizing,  
\begin{align*}
  z_1&=u^N\\
z_2&=u^Nf_{2,0}(u)+v^2u^{Ns}h_2(u,v)\\
z_j &=
 u^Nf_{j,0}(u)+vu^{Ns}h_j(u,v)\,,\quad j\ge 3
\end{align*}
with $h_2(u,v)$ a unit. Moreover, at least one $h_j(u,v)$ with $j\ge
3$ is a unit by a small adaptation of the argument of \cite[p.~464,
lines 7--11 ff.]{teissier}.

 The strict transform of
$A_0$ by the resolution $\pi'$ is the set expressed in local
coordinates by  
$$ A_0^*=\{(u,v): 
|z_2(u,v)-\sum_{j\ge1}a_j(0)u^{p_jN}|\le \alpha|z_1(u,v)|^s\}\,.
$$
Using the equations for $z_1$ and $z_2$ above we then obtain that
$$A_0^*=\{(u,v):|v^2h_2(u,v)|\le \alpha\}\,.$$
Since $h_2$ is a unit in $\C\{\!\{u,v\}\!\}$, the germ $(A_0,0)$
agrees up to order $>s$ with the germ $(A_0',0)$, where
$A_0'=\pi'(\{(u,v): |v|^2\le \beta\})$ where
$\beta=\alpha/|h_2(0,0)|$. Therefore the germs $(A_0,0)$ and
$(A_0',0)$ are bilipschitz equivalent, so it suffices to prove part
\eqref{it:LZ2} for the germ $(A_0',0)$. The cone structure of part
\eqref{it:LZ2} of the proposition is given by the foliation by solid
tori $T_r:=\{|z_1|=r\}\cap A_0'$. Fixing $c$ such that $|c|=r$, the
intersection $\{z_1=c\}\cap A_0'$ consists of $N$ meridianal disks.
Each is to high order of the form
$\{|c|^s(0,v^2h_2(0,0),vh_3(0,0),\dots,vh_n(0,0)):|v|\le
\sqrt\beta\}$, and is therefore bilipschitz equivalent to a flat disk
of radius proportional to $|c|^s$.

The tangent cone of $(A_0',0)$ is the line $L$ spanned by
$(1,f_{2,0}(0),f_{3,0}(0), \dots,f_{n,0}(0))$, which is transverse to the hyperplanes
$z_1=c$, so the angle between this line $L$ and the meridianal disk
sections is bounded away from $0$.  Let $D_\epsilon$ be the disk of
radius $\epsilon$ in $L$. Then up to bilipschitz equivalence the
transverse disks can be considered to be orthogonal to $D_\epsilon$,
giving a metric on $(A_0',0)$ outside the origin as a disk bundle over
$D_\epsilon\setminus \{0\}$ with fibers orthogonal to this disk and of
radius proportional to $r^s$ at distance $r$ from the origin.
\end{proof}
\begin{remark*}[VTZ]
  Recall that a Puiseux exponent $p_j$ of a plane curve given by
  $y=\sum_{i}a_ix^{{p_i}}$ is
  \emph{characteristic} if the embedded topology of the plane curves
  $y=\sum_{i=1}^{j-1}a_ix^{{p_i}}$ and $y=\sum_{i=1}^ja_ix^{{p_i}}$
  differ; equivalently $\denom(p_j)$ does not divide
  $\lcm_{i<j}\denom(p_i)$. We denote the characteristic exponents by
  ${p_{j_k}}$ for $k=1,\dots,r$, and write $p_{max}=p_{j_r}$ for the largest
  characteristic exponent.

  We had believed that the contact exponent of a component of the
  discriminant curve satisfies $s\ge p_{max}$ in general (we called
  this the ``Very Thin Zone Lemma'' or ``VTZ'' for short), but the
  referee pointed out a gap in the proof. And indeed, VTZ is
  false. For example, for the hypersurface given by
  $(x^2+y^2+z^2)^2+x^5+y^5+z^5=0$ the components of the polar curve
  are cusps with exponent $\frac32$ but the contact exponent is
  $1$.
VTZ is true if the multiplicity of $X$ is $\le 3$,
  and then $s$ is often significantly larger than $p_{max}$. In
  Examples \ref{ex:E8} and \ref{ex:two pair} below
we have respectively $p_{max}={\frac53}$ and $s=\frac{10}3$, and
$p_{max}=\frac{17}{9}$ and $s=\frac{124}{9}$. 
\end{remark*}

\begin{example}\label{ex:E8}
  Let $(X,0)$ be the $E_8$ singularity with equation
  $x^2+y^3+z^5=0$. Its resolution
  graph, with all Euler weights $-2$ and decorated with arrows corresponding to the strict transforms
  of the coordinate functions $x, y$ and $z$, is: 
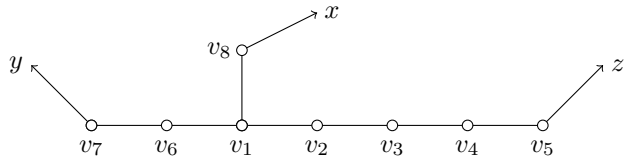
\begin{center}
\begin{tikzpicture}
  \draw[fill=white ] (-1,0)circle(2pt);
   \draw[fill=white ] (1,0)circle(2pt);
   \draw[thin ](-1,0)--(5,0);
      \draw[thin ](1,0)--(1,1);
      \draw[thin,>-stealth,->](-1,0)--+(-0.8,0.8);
       \draw[thin,>-stealth,->](5,0)--+(0.8,0.8);
        
        \draw[thin,>-stealth,->](1,1)--+(1,0.5);
  \draw[fill=white   ] (1,0)circle(2pt);
  \draw[fill=white  ] (3,0)circle(2pt);
     \draw[fill =white ] (1,1)circle(2pt);      
        \draw[fill =white ] (4,0)circle(2pt);
  \draw[fill=white   ] (0,0)circle(2pt);
   \draw[fill=white   ] (2,0)circle(2pt);
    \draw[fill=white ] (-1,0)circle(2pt);
   \draw[fill=white ] (1,0)circle(2pt); 
      \draw[fill=white ] (5,0)circle(2pt); 
 
\node(a)at(1,-0.3){   $v_ 1$};
\node(a)at(2,-0.3){   $v_ 2$};
\node(a)at(3,-0.3){   $v_ 3$};
\node(a)at(4,-0.3){   $v_ 4$};
\node(a)at(5,-0.3){   $v_ 5$};
\node(a)at(0,-0.3){   $ v_ 6$};
\node(a)at(-1,-0.3){   $ v_ 7$};
\node(a)at(0.7,1){   $v_ 8$};

\node(a)at(-2,0.8){   $y$};
\node(a)at(2.2,1.5){   $x$};
\node(a)at(6,0.8){   $z$};
 
  \end{tikzpicture} 
  \end{center}
  
  We denote by $C_j$ the exceptional curve corresponding to the vertex
  $v_ j$.  Then the total transform by $\pi$ of the coordinate
  functions $x, y$ and $z$ are:
  \begin{align*}
    (x \circ \pi) &= 15C_1 +12C_2+9C_3+6C_4+3C_5+10C_6+5C_7 +8C_8+   x^* \\
(y \circ \pi) &= 10C_1 +8C_2+6C_3+4C_4+2C_5+7C_6+4C_7 +5C_8+   y^* \\
(z \circ \pi) &= 6C_1 +5C_2+4C_3+3C_4+2C_5+4C_6+2C_7 +3C_8+   z^* 
  \end{align*}

  Set $f(x,y,z) = x^2+y^3+z^5$. The polar curve $\Pi$ of a generic
  linear projection $\ell\colon (X,0) \to (\C^2,0)$ has equation $g=0$
  where $g$ is a generic linear combination of the partial derivatives
  $f_x = 2x$, $f_y=3y^2$ and $f_z=5z^4$. The multiplicities of $g$ are
  given by the minimum of the compact part of the three divisors
  \begin{align*}
(f_x \circ \pi) &= 15C_1 +12C_2+9C_3+6C_4+3C_5+10C_6+5C_7 +8C_8+  f_x^* \\
(f_y \circ \pi) &=  20C_1 +16C_2+12C_3+8C_4+4C_5+14C_6+8C_7 +10C_8+  f_y^*  \\
(f_z \circ \pi) &= 24C_1 +20C_2+16C_3+12C_4+8C_5+16C_6+8C_7 +12C_8+ f_z^*
  \end{align*}
We then obtain that the total transform of $g$ is equal to:
$$(g \circ \pi) = 15C_1 +12C_2+9C_3+6C_4+3C_5+10C_6+5C_7 +8C_8+ \Pi^*\,.$$
In particular, $\Pi$ is resolved by $\pi$ and its strict transform
$\Pi^*$ has just one component, which intersects $C_8$.  Since the
multiplicities $m_8(f_x)=8$, $m_8(f_y)=10$ and $m_8(z)=12$ along $C_8$
are distinct, the family of polar curves, i.e., the linear
  system generated by $f_x, f_y$ and $f_z$, has a base point on
$C_8$. One must blow up twice to get an exceptional curve $C_{10}$
along which $m_{10}(f_x)=m_{10}(f_y)$, which resolves the linear
system. Then $N_{K_0}(\Delta_0)$ is included in the image by $\pi$
  of a neighborhood of $\Pi^*= \Pi_{\cal D}^*$ in ${\cal
    N}(C_{10})$ foliated by strict transforms $\Pi_{\cal D'}^*$,
  $\cal D'$ in a small disk around $\cal D$ in $\grassman(2,\C^3)$ as
  in the proof of Proposition \ref{prop:polar wedge}.

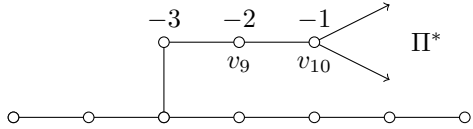
 \begin{center}
\begin{tikzpicture}
  \draw[fill=white ] (-1,0)circle(2pt);
   \draw[fill=white ] (1,0)circle(2pt);
   \draw[thin ](-1,0)--(5,0);
    \draw[thin ](1,1)--(3,1);
      \draw[thin ](1,0)--(1,1);

        \draw[thin,>-stealth,->](3,1)--+(1,0.5);
          \draw[thin,>-stealth,->](3,1)--+(1,-0.5);
  \draw[fill=white   ] (1,0)circle(2pt);
  \draw[fill=white  ] (3,0)circle(2pt);
     \draw[fill =white ] (1,1)circle(2pt);      
        \draw[fill =white ] (4,0)circle(2pt);
  \draw[fill=white   ] (0,0)circle(2pt);
   \draw[fill=white   ] (2,0)circle(2pt);
    \draw[fill=white ] (-1,0)circle(2pt);
       \draw[fill=white ] (5,0)circle(2pt); 
      
      \draw[fill=white ] (2,1)circle(2pt); 
\draw[fill=white ] (3,1)circle(2pt); 
    
  \node(a)at(1,1.3){   $-3$};
 \node(a)at(2,1.3){   $-2$};
\node(a)at(3,1.3){   $-1$};

\node(a)at(2,0.7){   $ v_ 9$};
\node(a)at(3,0.7){   $v_ {10}$};

\node(a)at(4.5,1){   $\Pi^*$};
  \end{tikzpicture} 
  \end{center}
  
  We now compute the contact exponent $s$ in Proposition
  \ref{prop:polar wedge}. For $(a,b) \in \C^2$ generic,
  $x+ay^2+bz^4=0$ is the equation of the polar curve $\Pi_{a,b}$ of a
  generic projection. The image $\ell(\Pi_{a,b}) \subset \C^2$ under
  the projection $\ell=(y,z)$ has equation
$$y^3 +a^2y^4 + 2aby^2z^4+z^5+b^2z^8=0$$
The discriminant curve $\Delta= \ell(\Pi_{0,0})$ has Puiseux expansion $y=(-z)^{5/3}$,
while for $(a,b) \neq (0,0)$, we get for   $\ell(\Pi_{a,b})$ a Puiseux expansion
$y=(-z)^{5/3} - \frac{a^2}3z^{10/3}+\cdots$. So the discriminant curve $\Delta$ has highest characteristic
exponent $5/3$ and its contact exponent is $10/3$.
\end{example}

\begin{example}\label{ex:two pair} Consider $(X,0)$ with equation
  $z^2+xy^{14}+(x^3+y^5)^3=0$. The dual graph of the minimal
  resolution $\pi$ has two nodes, one of them with Euler class
  $-3$. All other vertices have Euler class $-2$. Similar
  computations show that $\pi$ also resolves the polar $\Pi$ and we
  get:
 
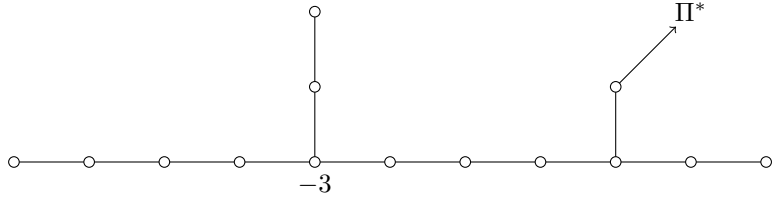
\begin{center}
\begin{tikzpicture}
  
   \draw[thin ](0,0)--(10,0);
     \draw[thin ](4,0)--(4,2);
       \draw[thin ](8,0)--(8,1);
    \draw[thin,>-stealth,->](8,1)--+(0.8,0.8);
   
   \draw[fill=white ] (0,0)circle(2pt);
     \draw[fill=white ] (1,0)circle(2pt);
       \draw[fill=white ] (2,0)circle(2pt);
         \draw[fill=white ] (3,0)circle(2pt);
           \draw[fill=white ] (4,0)circle(2pt);
             \draw[fill=white ] (5,0)circle(2pt);
               \draw[fill=white ] (6,0)circle(2pt);
                 \draw[fill=white ] (7,0)circle(2pt);
                   \draw[fill=white ] (8,0)circle(2pt);
                     \draw[fill=white ] (9,0)circle(2pt);
                       \draw[fill=white ] (10,0)circle(2pt);
                     
                       \draw[fill=white ] (4,1)circle(2pt);
                       \draw[fill=white ] (4,2)circle(2pt);
                       
                         \draw[fill=white ] (8,1)circle(2pt);
                         
                         \node(a)at(9,2){   $\Pi^*$};
                          \node(a)at(4,-0.3){   $-3$};
    \end{tikzpicture} 
  \end{center}
  
  Denoting by $C_1$ the exceptional curve such that $C_1 \cap \Pi^*
  \neq \emptyset$, we get $m_1(f_x)=124$, 
  $m_1(f_y)=130$ and $m_1(f_z)=71$. Then the linear system of polar
  curves admits a base point on $C_1$ and one has to perform
  $124-71=53$ blow-ups to resolve it. In this case one computes that
  the discriminant curve has two characteristic exponents, $5/3$ and
  $17/9$, and the Lipschitz exponent is $s= 124/9$.
\end{example}

Now we consider again the resolution $\pi\colon\widetilde X\to X$
defined in Section \ref{sec:thick-thin}, which is obtained from a
minimal good resolution by first blowing up base points of the linear
system of generic hyperplane sections and then blowing up intersection
points between \L-curves.  The following results help locate the polar
components relative to the Tjurina components of $\pi$.
\begin{proposition}\label{prop:intersecting L}
  If there are intersecting \L-curves before the final step then
  for any generic plane projection the strict transform of the polar
  has exactly one component through that common point and it
  intersects the two \L-curves transversely.
\end{proposition}
\begin{proof}
   Let $E_\mu$ and $E_\nu$ denote the two intersecting \L-curves and  choose
  coordinates $(u,v)$ centered at the intersection such that $E_\mu$
  and $E_\nu$ are locally given by $u=0$ and $v=0$ respectively. We
  assume our plane projection is given by $\ell=(x,y)\colon\C^n\to
  \C^2$, so $x$ and $y$ are generic linear forms. Then without loss of
  generality $x=u^mv^n$ in our local coordinates, and
  $y=u^mv^n(a+bu+cv+g(u,v))$ with $a\ne 0$ and $g$ of order
  $\ge2$. The fact that $E_\mu$ and $E_\nu$ are \L-curves means that
  $c\ne0$ and $b\ne0$ respectively. The polar component is given by
  vanishing of the Jacobian determinant
  $\det\frac{\partial(x,y)}{\partial(u,v)}=
  u^{2m-1}v^{2n-1}(mcv-nbu+mvg_v-nug_u)$.  Modulo terms of order $\ge
  2$ this vanishing is the equation $v=\frac{nb}{mc}u$, proving the
  lemma.
\end{proof}
\begin{lemma}[{Snoussi \cite[6.9]{Sn}}]\label{le:snoussi}
If\/ $\Gamma_j$ is a Tjurina component of $\Gamma$ and $E^{(j)}$ the union
of the $E_\nu$ with $\nu\in\Gamma_j$, then the strict transform of the
polar curve of any general linear projection to $\C^2$ intersects
$E^{(j)}$.\qed
\end{lemma}

\section{Milnor balls}\label{sec:balls}

From now on we assume our coordinates $(z_1\dots,z_n)$ in $\C^n$ are
chosen so that $z_1$ and $z_2$ are generic linear forms and
$\ell:=(z_1,z_2)\colon X\to \C^2$ is a generic linear projection. In
this section we denote by $B_\epsilon^{2n}$ the standard round ball in
$\C^n$ of radius $\epsilon$ and $S_\epsilon^{2n-1} $ its boundary.

The family of Milnor balls we use in the sequel consists of standard
``Milnor tubes'' associated with the Milnor-L\^e fibration for the map
$\zeta:=z_1|_X\colon X\to \C$. Namely, for some sufficiently small
$\epsilon_0$ and some $R>0$ we define for $\epsilon\le\epsilon_0$:
$$B_\epsilon:=\{(z_1,\dots,z_n):|z_1|\le \epsilon, |(z_1,\dots,z_n)|\le
R\epsilon\}\quad\text{and}\quad S_\epsilon=\partial B_\epsilon\,,$$ 
where $\epsilon_0$ and $R$ are chosen so that for $\epsilon\le \epsilon_0$:
\begin{enumerate}
\item\label{it:mb1} $\zeta^{-1}(t)$ intersects
$S_{R\epsilon}^{2n-1}$ transversely for $|t|\le \epsilon$;
\item\label{it:mb2} the polar curves for the projection
$\ell=(z_1,z_2)$ meet $S_\epsilon$ in the part
$|z_1|=\epsilon$.
\end{enumerate}
\begin{proposition}
$\epsilon_0$ and $R$ as above exist.
\end{proposition}
\begin{proof}
  We can clearly achieve \eqref{it:mb2} by choosing $R$ sufficiently
  large, since the tangent lines to the polar curve are transverse to
  the hyperplane $z_1=0$ by genericity of $z_1$.

To see that we can achieve \eqref{it:mb1} note that for $p\in X$
and $t=\zeta(p)$ the sphere $S_{|p|}^{2n-1}$ is not transverse to
$\zeta^{-1}(t)$ at the point $p$ if and only if the intersection
$T_{p} X\cap \{z_1=0\}$ is orthogonal to the direction
$\frac{\Vec{p}}{|p|}$ (considering $T_{p} X$ as a subspace of $\C^n$). We
will say briefly that condition $T(p)$ holds.

We must show there exists $r>0$ and $R>0$ so that
$T(p)$ fails for all 
$p\in X$ with $R|\zeta(p)|\le |p|\le r$, 
since then $R$ and $\epsilon_0:=\frac rR$ do what is
required. Suppose the contrary. Then the set
$$S:=\{(p,R)\in X\times \R_+ :  R|\zeta(p)|\le |p|\text{ and }T(p)\text{
  holds}\}$$ contains points with $p$ arbitrarily close to $0$ and $R$
arbitrarily large. By the arc selection lemma for semi-algebraic sets there is an
analytic arc $\gamma\colon [0,1]\to X\times \R_+$,
$\gamma(t)=(p(t),R(t))$, with $\gamma((0,1])\subset S$ and $\lim_{t\to
  0} p(t)=0$ and $\lim_{t\to 0}R(t)=\infty$. This arc is then tangent
at $0$ to a component $C$ of the curve $\zeta^{-1}(0)$.  Let $T$
denote the limit of tangent planes $T_{p(t)}X$ as $t\to 0$. Then,
since the tangent cone to $\zeta^{-1}(0)$ includes no exceptional
directions, $L':=T\cap \{z_1=0\}=\lim_{t\to 0} (T_{p(t)}X\cap
\{z_1=0\})$ is the tangent line to the curve $C$. Let $L$ be the real
line tangent to the curve $p(t)$ at $0$. Then $L$ is in $T$, since its
direction is a limit of directions $p(t)/|p(t)|$, and $L$ is in
$z_1=0$ by the definition of $\gamma$. Therefore $L\subset L'$, which
contradicts that $T(p)$ holds along the curve $p(t)$.
\end{proof}

\section{The thin pieces}\label{sec:thin}
In this section we prove the thinness of the pieces $Z_j$ defined
in section \ref{sec:thick-thin}.
We start with a proof of  Proposition \ref{prop:thin nh}, which states that a semi-algebraic germ $(Z,0)$ is contained in a horn neighborhood of its tangent cone.

\begin{proof}
  Without loss of generality $Z$ is closed. Consider the function $f\colon \epsilon \mapsto \max\{d(x, T Z \cap
  S_{\epsilon} ):x\in Z \cap S_{\epsilon}\}$. Since $f$
  is semi-algebraic, there exists $c>0$ and a Lojasiewicz exponent $q \geq 1$
  such that $f(\epsilon) \leq c \epsilon^q$ for all $\epsilon$ sufficiently small.  The tangent cone $T Z$ of
  $Z$ is the cone over the Hausdorff limit $\lim_{\epsilon\to
    0}(\frac{1}{\epsilon} Z \cap S_1)$. Thus for any $C>0$ the
  function $f$ satisfies $f(\epsilon) < C \epsilon$ for $\epsilon$
  sufficiently small ($d$ denotes the Hausdorff distance). This implies $q>1$ and proves  the proposition.
\end{proof}

\begin{proposition}\label{prop:thin1}~\hbox{}  
   \begin{enumerate}
   \item\label{it:EX1} For each $j$, the tangent cone of $(Z_j,0)$ at
     $0$ is an exceptional tangent line $L_j$ of $(X,0)$.
  
   Let  $q_j >1$ such that $Z_j\cap
  B_\epsilon$ is contained in a $q_j$-horn neighborhood
  of the line $L_j$. 
    \item\label{it:EX2}
The  restriction $\zeta_j\colon Z_j\setminus\{0\}\to
  D^2_\epsilon\setminus\{0\}$ of $z_1$ is a locally trivial fibration and there exists $c_j >0$ such that 
  such that each fiber $\zeta_j^{-1}(t)$ lies in a ball with radius 
  $(c_jt^{q_j})$ centered at the point $L_j\cap\{z_1=t\}$
  (we call these fibers the \emph{Milnor fibers of $Z_j$}).
\item\label{it:EX3} There is a vector field $v_j$ on
  $Z_j\setminus\{0\}$ which lifts by $\zeta_j$ the inward radial unit
  vector field on $\C\setminus \{0\}$ and has the property that any
  two integral curves for $v_j$ starting at points of $Z_j$ with the
  same $z_1$ coordinate approach each other faster than linearly. In
  particular, the flow along this vector field takes Milnor fibers to
  Milnor fibers while shrinking them faster than linearly.
\end{enumerate}
\end{proposition}

Note that it follows from part \eqref{it:EX1} of this proposition
that, with our choice of Milnor balls as in section \ref{sec:balls},
the link $Z_j^{(\epsilon)}=Z_j\cap S_\epsilon$ is included in the
$|z_1|=\epsilon$ part of $S_\epsilon$.

We need the following lemma. For any $h$ in the maximal ideal $\mathfrak m_{X,0}$ 
set  $\widetilde h:= h\circ\pi\colon\widetilde X\to \C$ and denote by 
$m_\nu(h)$ the multiplicity of $\widetilde h$ along the
exceptional curve $E_\nu$.

\begin{lemma}\label{le:1} Let $h_1=\zeta=z_1|_X$.
  For any Tjurina component there exist functions $h_2,\dots, h_m\in
\mathfrak m_{X,0}$ such that $h_1,\dots, h_m$ generate
  $\mathfrak m_{X,0}$, and such that for any vertex $\nu$ of the Tjurina
  component we have $m_{\nu}(h_i)>m_{\nu}(h_ 1)$ for $i>1$.
\end{lemma}

\begin{proof}
  Take any functions $g_2,\dots,g_m\in \mathfrak m_{X,0}$ such that $h_1$ and
  $g_2,\dots, g_m$ generate $\mathfrak m_{X,0}$.  Choose a vertex $\nu$ of the
  Tjurina component. By choice of $z_1$ we know $m_{\nu} (h_1)\le
  m_{\nu}(g_ i)$ for all $i>1$.  Choose a point $p$ on $E_\nu$
  distinct from the intersections with other exceptional curves and
  the strict transform of $g_i^{-1}(0)$.  Then $\widetilde h_1$ and
  $\widetilde g_i$ are given in local coordinates $(u,v)$ centered at
  $p$ by $\widetilde h_1=u^{m_{\nu} (h_1)}(a_1+\alpha_1(u,v))$ and
  $\widetilde g_i=u^{m_{\nu}(h_1)}(a_i+\alpha_i(u,v))$ where $a_1\ne
  0$ and $\alpha_1(0,0)=\alpha_i(0,0)= 0$. Let
  $h_i:=g_i-\frac{a_i}{a_1}h_1$ for $i=2,\dots,m$. Then
  $h_1,\dots,h_m$ generate $\mathfrak m_{X,0}$. 
  
  Assume that $m_{\nu}(h_i) = m_{\nu}(h_1)$.  Then the strict
  transform of $h_i^{-1}(0)$ passes through $p$. Let
  $E_{\nu_1},\ldots,E_{\nu_r}$ be the exceptional curves representing
  the vertices of our Tjurina component $\Gamma' \subset \Gamma$, with
  $\nu_1=\nu$. Let $I = \big( E_{\nu_k}. E_{\nu_l} \big)_{1\leq
    k,l\leq r}$ be the intersection matrix associated with $\Gamma'$
  and consider the $r$-vectors
  $$V={}^t(v_1,\dots, v_r),\quad 
  B={}^t(b_1,\ldots,b_r)\quad\text{defined by}$$
  $$v_k=m_{\nu_k}(h_i)- m_{\nu_k}(h_1),\quad 
  b_k=h_i^*.E_{\nu_k} - h_1^*.E_{\nu_k} + \sum_{\mu \in{ \cal
      L_{\nu_k}}} (m_{\mu}(h_i)- m_{\mu}(h_1))\,,
  $$ 
  where $\cal L_{\nu_k}$ denotes the set of \L-nodes of $\Gamma$
  adjacent to $\nu_k$.  Since $h_1^* . E_{\nu_k} = 0$ for all $k$ and
  $h_i^* . E_{\nu_1} \neq 0$, we have $b_k \geq 0$ for all $k$ and
  $b_1 >0$. Now $I.V+ B =0$, so $V=-I^{-1}B$.  All entries of $I^{-1}$
  are strictly negative, so all entries of $V$ are strictly positive,
  contradicting $m_{\nu}(h_i)= m_{\nu}(h_1)$. So in fact
  $m_{\nu}(h_i)>m_{\nu}(h_1)$.

  We now claim that $m_{\mu}(h_i)>m_{\mu}(h_1)$ for $i>1$ for any
  vertex $\mu$ of the Tjurina component adjacent to $\nu$ (it then
  follows inductively for every vertex of the Tjurina component).  So
  let $\mu$ be such a vertex and assume that $m_{\mu}(h_i)
  =m_{\mu}(h_1)$. Consider the meromorphic function $\widetilde
  h_1/\widetilde h_i$ on $E_\mu$. It takes finite values almost
  everywhere and has a pole at $E_\mu\cap E_\nu$ so it must have a
  zero at some point of $E_\mu$. This cannot happen since
  $m_{\nu'}(h_1)\le m_{\nu'}(h_i)$ for any $\nu'$ and the strict
  transform of the zero set of $h_1$ only intersects the \L-nodes
  ($h_1$ is the generic linear form).
\end{proof}
\begin{proof}[Proof of Proposition \ref{prop:thin1}]  Let
  $h_1,\dots,h_m$ be as in Lemma
  \ref{le:1} and let 
  $$q_j:=\min\left\{2,\frac{m_{\nu}(h_i)}{m_{\nu}(h_1)}:\nu\in\Gamma_j,
    i>1\right\}\,.$$ 
For each $k=2,\dots,n$ one has $$z_k|_X=\lambda_k h_1+\beta_k$$ with
$\lambda_k\in\C$ and $\beta_k\in(h_1^2, h_2,\dots,h_m)$. 

We will prove that the complex line $L_j \subset \C^n$ parametrized by
$(t,\lambda_2t,\dots, \lambda_nt)$, $t\in\C$ is the tangent cone to
$Z_j$. Since $Z_j$ contains complex curves (for example the projection
$\pi(\gamma)$ of any curvette $\gamma$ of the exceptional divisor
inside $\cal N(\Gamma_j)$), the tangent cone $TZ_j$ contains a complex
line. Therefore it suffices to prove that $Z_j \cap B_{\epsilon_0}$
is contained in a horn neighborhood of $L_j$.


Consider local coordinates $(u,v)$ in a compact neighborhood of a
point of $E_\nu\setminus \bigcup_{\mu\ne\nu}E_\mu$ in which $\widetilde
\beta_i=u^{m_{\nu}(\beta_i)}(a_i+\alpha_i(u,v))$ with $a_i\ne0$ and
$\alpha_i$ holomorphic. In this neighborhood $|\widetilde
z_i-\lambda_i\widetilde z_1|=O(|\widetilde
z_1|^{m_{\nu}(\beta_i)/m_{\nu}(h_ 1)})=O(|\widetilde z_1|^{q_j})$.

In a neighborhood of a point $E_\nu\cap E_\mu$ of intersection of two
exceptional curves of the Tjurina component we have local coordinates
$u,v$ such that $\widetilde h_1=u^{m_{\mu}(h_1)}v^{m_{\nu}(h_1)}$ and
$\widetilde
\beta_i=u^{m_{\mu}(\beta_i)}v^{m_{\nu}(\beta_i)}(a_i+\alpha_i(u,v))$
with $a_i\ne0$ and $\alpha_i$ holomorphic.  In this neighborhood we
again have $|\widetilde z_i-\lambda_i\widetilde z_1|=O(|\widetilde z_1|^{q_j})$.

By compactness of $\Nn(\Gamma_j)$ the estimate $|\widetilde z_i-\lambda_i
\widetilde z_1|=O(|\widetilde z_1|^{q_j})$ holds on all of $\Nn(\Gamma_j)$.
Thus, if we define $f_j\colon
Z_j=\pi(\Nn(\Gamma_j))\to \C^n$ by $f_j(p):=
z_1(p)(1,\lambda_2,\dots,\lambda_n)$, we have shown that
$|p-f_j(p)|=O(z_1(p)^{q_j})$.  In particular, for $p\in \Nn(\Gamma_j)$
the distance of $\pi(p)$ from $0$ is $O(|z_1(p)|)$, and therefore
$Z_j\cap B_\epsilon$ is contained in an
$(c_j\epsilon^{q_j})$-neighborhood of the line $L_j$ for some
$c_j>0$.

In particular, we see that $Z_j$ is thin, and tangent to the line
$L_j$. 
Thus each limit of tangent hyperplanes at a sequence of points
converging to $0$ in $Z_j\setminus\{0\}$ contains $L_j$.  We will
show there exist infinitely many such limits, which, by
definition, means that $L_j$ is an exceptional tangent line (this
is one of the two implications of Proposition 6.3
of Snoussi \cite{Sn}, see also Proposition 2.2.1 of \cite{LT},
but we did not find a clear statement in the literature).
Indeed, for any generic $(n-2)$-plane $H$ let $\ell_H\colon
\C^n\to\C^2$ be the projection with kernel $H$. By Lemma
\ref{le:snoussi} the strict transform $\Pi^*$ of the
polar $\Pi$ of $\ell_H|_X$ intersects
$\bigcup_{\nu\in\Gamma_j}E_\nu$. Let $C$ be a branch of $\Pi^*$
which intersects $\bigcup_{\nu\in\Gamma_j}E_\nu$. Then at each
$p\in C$ the tangent plane to $X$ at $p$ intersects $H$
non-trivially, so the limit of these planes as $p\to 0$ in $C$ is
a plane containing $L_j$ which intersects $H$ non-trivially. By
varying $H$ we see that there are infinitely many such limits
along sequences of points approaching $0$ along curves tangent to
$L_j$, as desired.

\eqref{it:EX2} of Proposition \ref{prop:thin1} is just the observation that
$\zeta_j$ is the restriction to $Z_j\setminus\{0\}$ of the Milnor-L\^e
fibration for $z_1\colon X\to \C$.
 
For \eqref{it:EX3} we will use local coordinates as above to construct the
desired vector field locally in $\pi^{-1}(Z_j\setminus\{0\})$; it can
then be glued together by a standard partition of unity
argument. Specifically, in a neighborhood of a point $p$ of
$E_\nu\setminus \bigcup_{\mu\ne\nu}E_\mu$, using coordinates
$(u=r_ue^{i\theta_u}, v=r_ve^{i\theta_v})$ with $\widetilde
h_1=u^{m_\nu(h_1)}$, the vector field
$$(\frac1{mr_u^{m-1}}\frac\partial{\partial r_u}\,,\,0)$$ with
$m=m_\nu(h_1)$ works, while in a neighborhood of a point $E_\nu\cap
E_\mu$, using coordinates with $\widetilde
h_1=u^{m_\nu(h_1)}v^{m_\mu(h_1)}$, the vector field
$$(\frac1{
\hbox{$mr_u^{m-1}r_v^{m'}$}}\frac\partial{\strut\partial r_u}\,,\,
\frac1{\hbox{$m'r_u^{m}r_v^{m'-1}$}}\frac\partial{\strut\partial r_v})$$ 
with
$m=m_\nu(h_1)$ and $m'=m_\mu(h_1)$ works.
\end{proof}

\section{The thick pieces}\label{sec:thick}

 In this section we prove the thickness of the pieces
  $Y_i$, $i=1\ldots,r$ defined in Section \ref{sec:thick-thin}.
  Recall that any such $Y_i$ has the form
    $Y=\pi(N(\Gamma_\nu))$, where $\Gamma$ is the dual graph of
    the resolution defined at the beginning of Section
    \ref{sec:thick-thin} and $\Gamma_{\nu}$ is a subgraph consisting
    of an \L-node $\nu$ of $\Gamma$ and any attached bamboos.

\begin{proposition}\label{prop:thick}  
    $Y=\pi(N(\Gamma_\nu))$ is thick.
\end{proposition}
 
\begin{proof}
  We use the minimal resolution $\pi' \colon \tilde X' \to X$ which
  factors both through $\pi$ and through Nash modification.  We denote
  by $\Gamma'$ its resolution graph and by $\sigma \colon \tilde X'
  \to \tilde X$ the map such that $\pi'=\sigma \circ \pi$. We then
  have $Y=\pi'(N(\Gamma'_\nu))$ where $\Gamma'_\nu$ is the subgraph of
  $\Gamma'$ which projects on $\Gamma_\nu$ when blowing-down through
  $\sigma$.  In particular, notice that if $G$ is a maximal connected
  subgraph in $\Gamma'_\nu \setminus \nu$, the link of $\pi'(\Nn(G))$
  is a solid torus.

  The thickness of $Y$ will follow from Lemma \ref{le:2} below,
  of which part \eqref{lem2} is the most difficult. The proof of
  \eqref{lem2} uses the techniques introduced in Section
  \ref{sec:carrousel}, and will be completed in section
  \ref{sec:carrousel1} (see the beginning of the proof of Lemma
  \ref{le:well reduced}).

\begin{lemma}\label{le:2}For any \L-node $\nu$ of $\Gamma'$ we have
    that
  \begin{enumerate}
  \item\label{lem1} $ {\pi'}(\Nn(\nu))$ is metrically conical;
  \item\label{lem2} $\pi(N(G))$ is conical for any maximal connected
    subgraph $G$ of $\Gamma'_{\nu} \setminus \nu$;
  \item\label{lem3}  $ {\pi}(N(E_\nu))$ is thick. 
\end{enumerate}
\end{lemma}

Assume we have the lemma.  The conicalness of the union of the conical
piece $ {\pi'}(\Nn(\nu))$ with the conical pieces ${\pi'}(N(G))$
coming from the maximal subgraphs $G$ of $\Gamma'_{\nu} \setminus \nu$
follows from \cite[Corollary 0.2]{valette} and part \eqref{lem3} of
the lemma then completes the proof that $Y$ is thick.
 \end{proof}

We now prove parts \eqref{lem1} and \eqref{lem3} of Lemma
\ref{le:2}. As mentioned above, part \eqref{lem2} will be proved later.

\begin{proof}[Proof of part \eqref{lem1} of Lemma \ref{le:2}]
  As before, $\ell=(z_1,z_2)\colon B_{\epsilon_0} \to (\C^2,0)$ is a
  generic linear projection to $\C^2$ and $\Pi$ the polar curve for
  this projection.  From now on we work only inside of a Milnor ball
  $B_{\epsilon_0}$ as defined in Section \ref{sec:balls} and $X$ now
  means $X\cap B_{\epsilon_0}$.

Denote by
$\Delta=\Delta_1\cup\dots\cup\Delta_k\subset \C^2$ the decomposition
of the discriminant curve $\Delta$ for $\ell$ into its irreducible
components. By genericity, we can assume the tangent lines in $\C^2$
to the $\Delta_i$ are of the form $z_2=b_iz_1$. Set
$$V_i:=\{(z_1,z_2)\in \C^2: |z_1|\le \epsilon_0,
|z_2-b_iz_1|\le\eta |z_1|\}\,,$$ 
where we choose $\eta$ small enough that
$V_i\cap V_j=\{0\}$ if $b_i\ne b_j$ and $\epsilon_0$ small enough that
$\Delta_i\cap\{(z_1,z_2):|z_1|\le \epsilon_0\}\subseteq V_i$.  Notice
that if two $\Delta_i$'s are tangent then the corresponding $V_i$'s
are the same.

Let $W_{i1},\dots,W_{ij_i}$ be the closure of the connected components
of $X\cap \ell^{-1}(V_i\setminus\{0\})$ which contain components of
the polar curve. Then the restriction of $\ell$ to
$\overline{X\setminus \bigcup W_{ij}}$ is a bilipschitz local
homeomorphism by Proposition \ref{prop:polar wedge}. In
particular, this set is metrically conical.

Assume that the strict transform by ${\pi}'$ of a component $\Pi_0$ of
$\Pi$ intersects the \L-node $E_{\nu}$.  By definition of ${\pi}'$, $p
= \Pi^*_0 \cap E_{\nu}$ is a smooth point of the exceptional divisor
$E= \pi'^{-1}(0)$ and $p$ is not a base point of the family of polar
curves of generic plane projections.  So varying the plane projection
varies locally the intersection point of $E_\nu$ with the strict
transform of the polar component, and hence also the tangent line of
the $\Pi_{0}$. We call $\Pi_0$ a \emph{moving polar component}. The
contact exponent of this moving polar component is $s=1$, so we can
take the corresponding $W_{ij}$ containing $\Pi_0$ to be the set $A_0$
of Proposition \ref{prop:polar wedge}.\eqref{it:LZ2} and it follows
that $W_{ij}$ is metrically conical.  Hence, if $W$ denotes the union
of $W_{ij}$ such that $\Pi \cap W_{ij}$ is not a moving polar
component, the semi-algebraic set $\overline{X \setminus W}$ is
metrically conical.  Since the coordinates at the double points of $E$
on an \L-curve $E_\nu$ can be chosen so that $\Nn(\nu)$ is a connected
component of $\pi^{-1}(\overline{X \setminus W})$, part \eqref{lem1}
of the lemma is proved.
\end{proof}
 
\begin{proof}[Proof of part \eqref{lem3}]
  Let $\mu$ an adjacent vertex to our \L-node $\nu$. We must show that
  the conical structure we have just proved can be extended to a thick
  structure over $N(E_\nu)\cap N(E_\mu)$. Take local coordinates $u,v$
  at the intersection of $E_\nu$ and $E_\mu$ such that $u=0$ resp.\
  $v=0$ is a local equation for $E_\nu$ resp.\ $E_\mu$.

  By Lemma \ref{le:1} we can choose $c \in \C$ so that the linear form
  $h_2 = z_1 -c z_2$ satisfies $m_{\mu}(z_1) < m_{\mu}(h_2)$. We 
  change coordinates to replace $z_2$ by $h_2$, which does not change
  the linear projection $\ell$, so we still have that
  $m_{\nu}(z_1)=m_{\nu}(h_2)$. So we may assume that in our local coordinates
  $z_1=u^{m_{\nu}(z_1)}v^{m_{\mu}(z_1)}$ and
  $z_2=u^{m_{\nu}(z_1)}v^{m_{\mu}(h_2)}(a+g(u,v))$ where $a\in \C^*$ and
  $g(0,0)=0$. To higher order, the lines $v=c$ for $c\in \C$ project
  onto radial lines of the form $z_2=c'z_1$ and the sets $|u|=d$ with
  $d\in \R^+$ project to horn-shaped real hypersurfaces of the form
  $|z_2|=d'|z_1|^{m_{\mu}(h_2)/m_{\mu}(z_1)}$. 
  \begin{figure}[ht]
 \centering
\begin{tikzpicture} 
\begin{scope}[scale=0.7]
(-1.2,-2.5)--(-1.2,-1.2)--(1.2,-1.2)--(1.2,-2.5)--cycle;
(-1.2,1.2)--(-1.2,2.5)--(1.2,2.5)--(1.2,1.2)--cycle;

\path[fill, lightgray](-1.2,-1.2)--(-1.2,1.2)--(1.2,1.2)--(1.2,-1.2)--(-1.2,-1.2);
\draw[thick,>=stealth,->](-2.5,0)--(2.5,0);
\draw[thick,>=stealth,->](0,-2.5)--(0,2.5);
\draw[thin](1.2,2.5)--(1.2,-2.5);
\draw[thin](-1.2,2.5)--(-1.2,-2.5);
\draw[thin](-2.5,1.2)--(2.5,1.2);
\draw[thin](-2.5,-1.2)--(2.5,-1.2);
\draw[thin,gray](-1,-1.2)--(-1,1.2);
\draw[thin,gray](-0.8,-1.2)--(-0.8,1.2);
\draw[thin,gray](-0.6,-1.2)--(-0.6,1.2);
\draw[thin,gray](-0.4,-1.2)--(-0.4,1.2);
\draw[thin,gray](-0.2,-1.2)--(-0.2,1.2);

\draw[thin,gray](1,-1.2)--(1,1.2);
\draw[thin,gray](0.8,-1.2)--(0.8,1.2);
\draw[thin,gray](0.6,-1.2)--(0.6,1.2);
\draw[thin,gray](0.4,-1.2)--(0.4,1.2);
\draw[thin,gray](0.2,-1.2)--(0.2,1.2);

\draw[thin,gray](-1.4,-1.2)--(-1.4,1.2);
\draw[thin,gray](-1.6,-1.2)--(-1.6,1.2);
\draw[thin,gray](-1.8,-1.2)--(-1.8,1.2);
\draw[thin,gray](-2,-1.2)--(-2,1.2);
\draw[thin,gray](-2.2,-1.2)--(-2.2,1.2);

\draw[thin,gray](1.4,-1.2)--(1.4,1.2);
\draw[thin,gray](1.6,-1.2)--(1.6,1.2);
\draw[thin,gray](1.8,-1.2)--(1.8,1.2);
\draw[thin,gray](2,-1.2)--(2,1.2);
\draw[thin,gray](2.2,-1.2)--(2.2,1.2);

\node(a)at(0.3,2.4){u};
\node(a)at(2.5,0.25){v};
\node(a)at(-0.5,-2.8){${\cal N}(E_\mu)$};
\node(a)at(-3,-0.5){${\cal N}(E_\nu)$};

\draw[very thick,>=stealth,->](3.5,0)--(5.5,0);
\node(a)at(4.5,0.4){$\pi$};

\begin{scope}[xshift=8cm]

\node(a)at(90:2.8){$\pi({\cal N}(E_\mu))$};
\node(a)at(0:2.5){$\pi({\cal N}(E_\nu))$};

\path[fill, lightgray] (0,0)..controls (0,0.5) and (0,1.5)..(-1.5,2.5)--cycle;
\path[fill, lightgray] (0,0)..controls (0,0.5) and (0,1.5)..(1.5,2.5);
\path[fill, lightgray] (0,0)..controls (0,-0.5) and (0,-1.5)..(1.5,-2.5);
\path[fill, lightgray] (0,0)..controls (0,-0.5) and (0,-1.5)..(-1.5,-2.5);

\draw[thin,gray](0:0)--(56:3);
\draw[thin,gray](0:0)--(53:3);
\draw[thin,gray](0:0)--(50:3);
\draw[thin,gray](0:0)--(47:3);
\draw[thin,gray](0:0)--(44:3);
\draw[thin,gray](0:0)--(41:3);

\draw[thin,gray](0:0)--(-56:3);
\draw[thin,gray](0:0)--(-53:3);
\draw[thin,gray](0:0)--(-50:3);
\draw[thin,gray](0:0)--(-47:3);
\draw[thin,gray](0:0)--(-44:3);
\draw[thin,gray](0:0)--(-41:3);

\draw[thin,gray](0:0)--(180-56:3);
\draw[thin,gray](0:0)--(180-53:3);
\draw[thin,gray](0:0)--(180-50:3);
\draw[thin,gray](0:0)--(180-47:3);
\draw[thin,gray](0:0)--(180-44:3);
\draw[thin,gray](0:0)--(180-41:3);

\draw[thin,gray](0:0)--(56-180:3);
\draw[thin,gray](0:0)--(53-180:3);
\draw[thin,gray](0:0)--(50-180:3);
\draw[thin,gray](0:0)--(47-180:3);
\draw[thin,gray](0:0)--(44-180:3);
\draw[thin,gray](0:0)--(41-180:3);

\draw[thin,gray](0:0)--(65:2.3);
\draw[thin,gray](0:0)--(68:2.05);
\draw[thin,gray](0:0)--(71:1.8);
\draw[thin,gray](0:0)--(74:1.5);
\draw[thin,gray](0:0)--(77:1.3);
\draw[thin,gray](0:0)--(80:1.05);
    
\draw[thin,gray](0:0)--(-62:2.6);
\draw[thin,gray](0:0)--(-65:2.3);
\draw[thin,gray](0:0)--(-68:2.05);
\draw[thin,gray](0:0)--(-71:1.8);
\draw[thin,gray](0:0)--(-74:1.5);
\draw[thin,gray](0:0)--(-77:1.3);
\draw[thin,gray](0:0)--(-80:1.05);
    
\draw[thin,gray](0:0)--(180-62:2.6);
\draw[thin,gray](0:0)--(180-65:2.3);
\draw[thin,gray](0:0)--(180-68:2.05);
\draw[thin,gray](0:0)--(180-71:1.8);
\draw[thin,gray](0:0)--(180-74:1.5);
\draw[thin,gray](0:0)--(180-77:1.3);
\draw[thin,gray](0:0)--(180-80:1.05);
       
\draw[thin,gray](0:0)--(62-180:2.6);
\draw[thin,gray](0:0)--(65-180:2.3);
\draw[thin,gray](0:0)--(68-180:2.05);
\draw[thin,gray](0:0)--(71-180:1.8);
\draw[thin,gray](0:0)--(74-180:1.5);
\draw[thin,gray](0:0)--(77-180:1.3);
\draw[thin,gray](0:0)--(80-180:1.05);
       
\draw(0,0)..controls (0,0.5) and (0,1.5)..(-1.5,2.5);
\draw(0,0)..controls (0,0.5) and (0,1.5)..(1.5,2.5);
\draw(0,0)..controls (0,-0.5) and (0,-1.5)..(1.5,-2.5);
\draw(0,0)..controls (0,-0.5) and (0,-1.5)..(-1.5,-2.5);
\draw[thin](1.5,-2.5)--(-1.5,2.5);
\draw[thin](1.5,2.5)--(-1.5,-2.5);
 
\draw[thin,gray](0:0)--(62:2.6);
\draw[thin,gray](0:0)--(65:2.3);
\draw[thin,gray](0:0)--(68:2.05);
\draw[thin,gray](0:0)--(71:1.8);
\draw[thin,gray](0:0)--(74:1.5);
\draw[thin,gray](0:0)--(77:1.3);
\draw[thin,gray](0:0)--(80:1.1);
\end{scope}
\end{scope}
\end{tikzpicture}
    \caption{Thickness near the boundary of a thick piece}
    \label{fig:6.1(2)}
\end{figure}
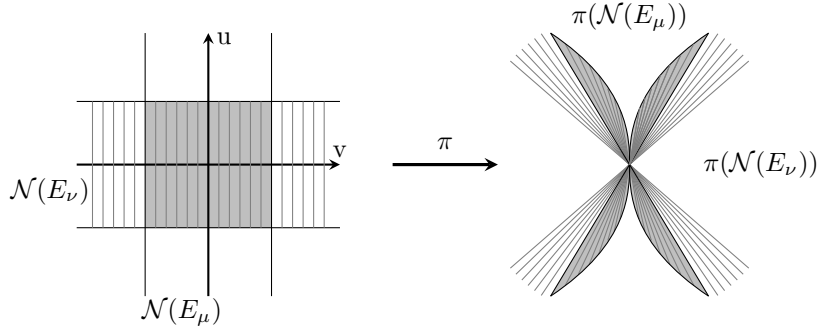
  In particular the image of a small domain of the form
  $\{(u,v):|u|\le d_1, |v|\le d_2\}$ is thick.  See Figure
  \ref{fig:6.1(2)} for a schematic real picture.  By Proposition
  \ref{le:very thin} the   local bilipschitz constant remains bounded
  in this added region, so the desired result again follows by pulling
  back the standard conical structure from $\C^2$.
\end{proof}

\section{Fast loops}\label{fast loops}

We recall the definition of a ``fast loop'' in the sense of
\cite{BFN1}, which we will here call sometimes ``fast loop of the
first kind,'' since we also define a closely related concept of ``fast
loop of the second kind.''  We first need another definition.

\begin{definition}
  If $M$ is a compact Riemannian manifold and $\gamma$ a closed
  rectifiable null-homotopic curve in $M$, the \emph{isoperimetric
 ratio} for $\gamma$ is the infimum of areas of singular disks in $M$
  which $\gamma$ bounds, divided by the square of the length of
  $\gamma$.
\end{definition}

\begin{definition} \label{def:fast loop} Let $\gamma$ be a closed
  curve in the link $X^{(\epsilon_0)}=X\cap S_{\epsilon_0}$. Suppose
  there exists a continuous family of loops $\gamma_\epsilon\colon
  S^1\to X^{(\epsilon)}$, $\epsilon\le \epsilon_0$, whose lengths
  shrink faster than linearly in $\epsilon$ and with
  $\gamma_{\epsilon_0}=\gamma$. If $\gamma$ is homotopically
  nontrivial in $X^{(\epsilon_0)}$ we call the family
  $\{\gamma_\epsilon\}_{0<\epsilon\le \epsilon_0}$ a \emph{fast loop
    of the first kind} or simply a \emph{fast loop}. If $\gamma$ is
  homotopically trivial but the isoperimetric ratio of
  $\gamma_\epsilon$ tends to $\infty$ as $\epsilon\to 0$ we call the
  family $\{\gamma_\epsilon\}_{0<\epsilon\le \epsilon_0}$ a \emph{fast
    loop of the second kind}\footnote{Fast loops of the
    second kind were needed in an early version of this paper but not
    in the current version. We have retained them since their analogs
    are useful in higher dimension.}.
\end{definition}

\begin{proposition}\label{prop:fast loops}
  The existence of a fast loop of the first or second kind is an
  obstruction to the metric conicalness of $(X,0)$.
\end{proposition}

\begin{proof}
  It is shown \cite{BFN1} that a fast loop of the first kind cannot
  exist in a metric cone, so its existence is an obstruction to
  metric conicalness.

  For a fixed Riemannian manifold $M$ the isoperimetric ratio for a
  given nullhomotopic curve is invariant under scaling of the metric
  and is changed by a factor of at most $K^4$ by a $K$-bilipschitz
  homeomorphism of $M$. It follows that if $X$ is metrically conical
  then for any $C>0$ there is a overall bound on the isoperimetric
  ratio of nullhomotopic curves in $X^{(\epsilon)}$ of length $\le
  C\epsilon$ as $\epsilon\to 0$.
\end{proof}

\begin{theorem}\label{th:fast loops}
  Any curve in a Milnor fiber of a thin zone $Z_j^{(\epsilon_0)}$ of
  $X^{(\epsilon_0)}$ which is homotopically nontrivial in
  $Z_j^{(\epsilon_0)}$ gives a fast loop of the first or second kind.
\end{theorem}

\begin{proof}
  Using the vector field of Proposition \ref{prop:thin1}, any closed
  curve $\gamma$ in the Milnor fiber of the link $Z_j^{(\epsilon_0)}$
  of a thin piece $Z_j$ of $X^{(\epsilon_0)}$ gives rise to a
  continuous family of closed curves $\gamma_\epsilon\colon[0,1]\to
  Z_j^{(\epsilon)}$, $\epsilon\le\epsilon_0$, whose lengths shrink
  faster than linearly with respect to $\epsilon$.

  If $\gamma$ is homotopically non-trivial in $X^{(\epsilon_0)}$ then
  $\{\gamma_\epsilon \}_{0<\epsilon\le\epsilon_0}$ is a fast loop of
  the first kind. Otherwise, let $f\colon D_\epsilon\to
  X^{(\epsilon)}$ be a map of a disk with boundary $\gamma_\epsilon$.

  Let $Y=\bigcup_iY_i$ denote the thick part of $X$. Let
  $Y_\epsilon\subset Y$ for $\epsilon\le \epsilon_0$ be a collection
  of conical subsets as in Definition \ref{def:thick}. The link of
  $Y_\epsilon$ is $Y^{(\epsilon)}$. We can approximate $f$ by a smooth
  map transverse to $\partial Y^{(\epsilon)}$ while only increasing
  area by an arbitrarily small factor. Then $f(D)\cap\partial
  Y^{(\epsilon)}$ consists of smooth immersed closed curves. A
  standard innermost disk argument shows that at least one of them is
  homotopically nontrivial in $\partial Y^{(\epsilon)}$ and bounds a
  disk $D'$ in $Y^{(\epsilon)}$ obtained by restricting $f$ to a
  subdisk of $D$.  Since there is a lower bound proportional to
  $\epsilon$ on the length of essential closed curves in $\partial
  Y^{(\epsilon)}$ and $Y$ is thick, the area of $D'$ is bounded below
  proportional to $\epsilon^2$. It follows that the isoperimetric
  ratio for $\gamma_\epsilon$ tends to $\infty$ as $\epsilon\to 0$.
\end{proof}

\begin{theorem}\label{th:thin fast loops}~
\begin{enumerate}
\item  Every thin piece $Z_j$ contains a fast loop (of the first
  kind). In fact every boundary component of a Milnor fiber $F_j$ of
  $Z_j$ is a  fast loop.
\item Every point $p\in Z_j$ lies on some fast loop
  $\{\gamma_\epsilon\}_{0<\epsilon\le \epsilon_0}$ and every real
  tangent line to $Z_j$ is tangent to 
  $\bigcup_\epsilon\gamma_\epsilon\cup\{0\}$ for some such fast loop.
\end{enumerate}
\end{theorem}
 
\begin{proof} Part 2 follows from Part 1 because the link
  $Z_j^{(\epsilon_0)}$ is foliated by Milnor fibers and a boundary
  component of a Milnor fiber $\zeta_j^{-1}(t)$ can be isotoped into a
  loop $\gamma$ through any point of the same Milnor fiber. The family
  $\{\gamma_\epsilon\}_{0<\epsilon\le \epsilon_0}$ obtained from
  $\gamma$ using the vector field of Proposition \ref{prop:thin1} is a
  fast loop whose tangent line is the real line $L_j\cap \{\arg z_1=
  \arg t\}$. 

  We will actually prove a stronger result than part 1 of the theorem,
  since it is needed in Section \ref{sec:invariance}. First we need a
  remark.
 \begin{remark}\label{rem:nielsen}
   The monodromy map $\phi_j\colon F_j\to F_j$ for the fibration
   $\zeta_j \mid_{ Z_j^{(\epsilon_0)}}$ is a quasi-periodic map, so
   after an isotopy it has a decomposition into subsurfaces $F_\nu$ on
   which $\phi_j$ has finite order acting with connected quotient,
   connected by families of annuli which $\phi_j$ cyclically permutes
   by a generalized Dehn twist (i.e., some power is a Dehn twist on
   each annulus of the family). The minimal such decomposition is the
   Thurston-Nielsen decomposition, which is unique up to isotopy. By
   \cite{pichon}[Lemme 4.4] each $F_{\nu}$ is associated with a node
   $\nu$ of $\Gamma_j$ while each string joining two nodes $\nu$ and
   $\nu'$ of $\Gamma_j$ corresponds to a $\phi_j$-orbit of annuli
   connecting $F_{\nu}$ to $F_{\nu'}$. This decomposition of the fiber
   $F_j$ corresponds to the minimal decomposition of
   $Z_j^{(\epsilon_0)}$ into Seifert fibered manifolds
   $Z_{\nu}^{(\epsilon_o)}$, with thickened tori between them, i.e.,
   the JSJ decomposition.
\end{remark}
The following proposition is more general than part 1 of Theorem
\ref{th:thin fast loops} and therefore completes its proof.
\end{proof}
\begin{proposition}\label{prop:fast boundaries}
  Each boundary component of a $F_{\nu}$ gives a fast loop (of the
  first kind).
 \end{proposition}
 \begin{proof}
   We assume the contrary, that some boundary component $\gamma$ of
   $F_{\nu}$ is homotopically trivial in $X^{(\epsilon_0)}$. We will
   derive a contradiction.

   Let $T$ be the component of $\partial Z_{\nu}^{(\epsilon_0)}$ which
   contains $\gamma$. Then it is compressible, so it contains an
   essential closed curve which bounds a disk to one side of $T$ or
   the other.  Cutting $X^{(\epsilon_0)}$ along $T$ gives a manifold
   with a compressible boundary torus. But a plumbed
   manifold-with-boundary given by negative definite plumbing has a
   compressible boundary component if and only if it is a solid torus,
   see \cite{neumann81, MPW}.  So $T$ separates $X^{(\epsilon_0)}$
   into two pieces, one of which is a solid torus.  Call it $A$.
  \begin{lemma}\label{le:gamma is positive multiple of core}
 $\gamma$ represents a nontrivial element of $\pi_1(A)$.
\end{lemma}
\begin{proof}
  Let $\Gamma_0$ be the subgraph of $\Gamma$ representing $A$, i.e., a
  component of the subgraph of $\Gamma$ obtained by removing the edge
  corresponding to $T$. We will attach an arrow to $\Gamma_0$ where
  that edge was. Then the marked graph $\Gamma_0$ is the plumbing
  graph for the solid torus $A$. We note that $\Gamma_0$ must contain
  at least one \L-node, since otherwise $A$ would have to be
  $Z_j^{(\epsilon_0)}$ or a union of pieces of its JSJ decomposition,
  but this cannot be a solid torus.

  We now blow down $\Gamma_0$ to eliminate all $(-1)$-curves. We then
  obtain a bamboo $\Gamma'$ with negative definite matrix with an
  arrow at one extremity. It is the resolution graph of some cyclic
  quotient singularity $Q=\C^2/(\Z/p)$ and from this point of view the
  arrow represents the strict transform of the zero set of the
  function $x^p$ on $Q$, where $x,y$ are the coordinates of
  $\C^2$. The Milnor fibers of this function give the meridian discs
  of the solid torus $A$. Let $D$ be such a meridian disc. In the
  resolution, the intersection of $D$ with any curvette (transverse
  complex disk to an exceptional curve) is positive (namely an entry
  of the first row of $pS_{\Gamma'}^{-1}$, where $S_{\Gamma'}$ is the
  intersection matrix associated with $\Gamma'$). Since $\Gamma_0$
  contains an \L-node, then in particular $D$ intersects transversely
  the strict transform $z_1^*$ of $z_1$ and we set $D.z_1^*=s>0$.
  
  Let us return back to our initial $X^{(\epsilon_0)}$. Recall that in
  this paper, we use Milnor ball $B_{\epsilon}$ which are standard
  Milnor tubes for the Milnor-L\^e fibration $\zeta = z_1|_{X}$ (see
  Section \ref{sec:balls}). Denote again by $\zeta \colon T \to
  S^1_{\epsilon_0} $ the restriction of $z_1$ to $T =
  z_1^{-1}(S^1_{\epsilon_0})\cap X^{(\epsilon_0)}$, and let
  $\zeta_*\colon H_1(T;\Z)\to\Z$ be the induced map.
    
  The meridian curve $c =\partial D$ of $A$ is homologically
  equivalent to the sum of the boundary curves $c_1,\ldots c_s$ of
  small disk neighborhoods of the intersection points of $D$ with $
  z_1^*$.  The Milnor-L\^e fibration $\zeta$ is equivalent to the
  Milnor fibration $\frac{z_1}{|z_1|} \colon X^{(\epsilon_0)}
  \setminus L \to \mathbb S^1$ outside a neighborhood of the link
  $L=\{z_1=0\} \cap X^{(\epsilon_0)}$, and the latter is an open-book
  fibration with binding $L$. Therefore $\zeta_*(c_i)=1$ for each
  $i=1,\ldots,s$ and $\zeta_*(c)=s>0$.  Since $\zeta_*(\gamma)=0$,
  this implies the lemma.
  \end{proof}
 
  We now know that $\gamma$ is represented by a non-zero multiple of
  the core of the solid torus $A$. This core curve is the curvette
  boundary for the curvette transverse to the end curve of the bamboo
  obtained by blowing down $\Gamma_0$ in the above proof.

\begin{lemma}
 \label{le:meridian curves are nontrivial in pi_1} Given any
 good resolution graph (not necessarily minimal) for a normal
 complex surface singularity whose link $\Sigma$ has infinite
 fundamental group, the boundary of a curvette  always represents an
 element of infinite order in $\pi_1(\Sigma)$.
  \end{lemma}
    
\begin{proof}
  After blowing down to obtain a minimal good resolution graph the
  only case to check is when the blown down curvette $c$ intersects a
  bamboo, since otherwise its boundary represents a nontrivial element
  of the fundamental group of some piece of the JSJ decomposition of
  $\Sigma$ and each such piece embeds $\pi_1$-injectively in
  $\Sigma$. If $c$ intersects a single exceptional curve $E_1$ of the
  bamboo, then its boundary is homotopic to a positive multiple of the
  curvette boundary for $E_1$, so by the argument in the proof of the
  previous lemma, the boundary of $c$ is homotopic to a positive
  multiple of the core curve of the corresponding solid torus. As a
  fiber of the Seifert fibered structure on a JSJ component of
  $\Sigma$, this element has infinite order in $\pi_1$ of the
  component and therefore of $\Sigma$. Finally suppose $c$ intersects
  the intersection point of two exceptional divisors $E_1$ and $E_2$
  of the bamboo. Then in local coordinates $(u,v)$ with $E_1=\{u=0\}$
  and $E_2=\{v=0\}$ the curvette can be given by a Puiseux expansion
  $u=\sum_iv^{\frac{q_i}{p_i}}$, where, without loss of generality,
  $q_i>p_i$. Then the boundary of $c$ is an iterated torus knot in
  this coordinate system which homologically is a positive multiple of
  $p_1\mu_1+q_1\mu_2$, where $\mu_i$ is a curvette boundary of $E_i$
  for $i=1,2$. We conclude by applying the previous argument to
  $\mu_1$ and $\mu_2$ that $c$ is a positive multiple of the core
  curve of the solid torus, completing the proof.
\end{proof}
Returning to the proof of Proposition \ref{prop:fast boundaries}, we see
that if $\pi_1(X^{(\epsilon_0)})$ is infinite then the curve $\gamma$
of Lemma \ref{le:gamma is positive multiple of core} is nontrivial in
$\pi_1(X^{(\epsilon_0)})$ by Lemma \ref{le:meridian curves are
  nontrivial in pi_1}, proving the proposition in this case.

It remains to prove the proposition when $\pi_1(X^{(\epsilon_0)})$ is
finite. Then $(X,0)$ is a rational singularity (\cite{boutot,
  hochster}), and the link of $X^{(\epsilon_0)}$ is either a lens
space or a Seifert manifold with three exceptional fibers with
multiplicities $(\alpha_1, \alpha_2, \alpha_3)$ equal to $(2,3,3)$,
$(2,3,4)$, $(2,3,5)$ or $(2,2,k)$ with $k \geq 2$.

We will use the following  lemma

\begin{lemma}
  \label{le: core criterion} Let $A$ and $\gamma$ as in Lemma
  \ref{le:gamma is positive multiple of core}. Suppose the subgraph
  $\Gamma_0$ of\/ $\Gamma$ representing the solid torus $A$ is a bamboo
  $$\splicediag{6}{20}{
 \undertag{\overtag{\circ}{-b_1}{8pt}}{{\nu_1}}{0pt}\lineto[r]&
 \undertag{\overtag{\circ}{-b_2}{8pt}}{{\nu_2}}{0pt}\dashto[r]&&
 \dashto[r]&\undertag{\overtag{\circ}{-b_r}{8pt}}{{\nu_r}}{0pt}\\~}$$
  attached at vertex $\nu_1$ to the vertex $\nu$ of\/ $\Gamma$. Denote
  by $m_{\nu}$ and $m_{\nu_1}$ the multiplicities of the function
  $\widetilde z_1 = z_1 \circ \pi$ along $E_{\nu} $ and $E_{\nu_1}$, and
  let $(\alpha,\beta)$ be the Seifert invariant of the
  core of $A$ viewed as a singular fiber of the $S^1$-fibration
  over $E_{\nu}$, i.e.,  $\frac{\alpha}{\beta} = [b_{\nu_1},
  \ldots,b_{\nu_n}]$. Let $C$ be the
  core of $A$ oriented as the boundary of a curvette of
  $E_{\nu_r}$. Then, with $d=\gcd(m_{\nu}, m_{\nu_1})$, we have
  $$ \gamma = \left(\frac{m_{\nu_1}}d \alpha -
  \frac{m_{\nu}}d \beta\right)C  
  \quad\text{in}\quad H_1(A;\Z).$$
\end{lemma}

\begin{proof}
  Orient the torus $T$ as a boundary component of $A$. Let $C_{\nu}$
  and $C_{\nu_1}$ in $T$ be boundaries of curvettes of $E_{\nu}$ and
  $E_{\nu_1}$. Then $\gamma = \frac{m_{\nu}}{d}
  C_{\nu_1} - \frac{m_{\nu_1}}{d} C_{\nu}$, and
  the meridian of $A$ on $T$ is given by $M = \alpha C_{\nu_1} + \beta
  C_{\nu}$. Then $\gamma = \lambda C$ where $\lambda =M. \gamma $. As
  $C_{\nu_1}.C_{\nu} = +1$ on $T$, we then obtain the stated formula.
\end{proof}

Let us return to the proof of Proposition \ref{prop:fast
  boundaries}. When $X^{(\epsilon_0)}$ is a lens space, then the
minimal resolution graph is a bamboo:
$$\splicediag{6}{20}{
  \overtag{\circ}{-b_1}{8pt}\lineto[r]&
  \overtag{\circ}{-b_2}{8pt}\dashto[r]&&
  \dashto[r]&\overtag{\circ}{-b_n}{8pt}}$$ The function $\widetilde z_1$
has multiplicity $1$ along each $E_{i}$, and the strict transform of
$z_1$ has $b_{1}-1$ components intersecting $E_1$, $b_n-1$ components
on $E_n$, and $b_i-2$ components on any other curve $E_i$. In
particular the \L-nodes are the two extremal vertices of the bamboo
and any vertex with $b_i\ge 3$.  To get the adapted resolution graph
of section \ref{sec:thick-thin} we should blow up once between any two
adjacent \L-nodes. Then the subgraph $\Gamma_j$ associated with our
thin piece $Z_j$ is either a $(-1)$-weighted vertex or a maximal
string $\nu_i,\nu_{i+1},\dots,\nu_k$ of vertices excluding $\nu_1,
\nu_n$ carrying self intersections $b_j=-2$.  

We first consider the
second case. Let $\gamma$ be the intersection of the Milnor fiber of
$Z_j$ with the plumbing torus at the intersection of $E_k$ and
$E_{k+1}$. According to Lemma \ref{le: core criterion}, we have
$\gamma = ( \alpha - \beta )C$ where $C$ is the boundary of a curvette
of $E_n$ and $\frac{\alpha}{\beta} = [b_{k+1},\ldots,b_n]$ with $E_i^2
= -b_i$. But $C$ is a generator of the cyclic group
$\pi_1(X^{(\epsilon_0)})$, which has order $p = [b_{1}, \ldots ,
b_{n}]$. Since $0 <\alpha - \beta \leq \alpha < p$, we obtain that
$\gamma$ is nontrivial in $\pi_1(X^{(\epsilon_0)})$.

For the case of a $(-1)$-vertex we can work in the
minimal resolution with $E_k$ and $E_{k+1}$ the two adjacent \L-curves
and $T$ the plumbing torus at their intersection, since blowing down
the $(-1)$-curve does not change $\gamma$. So it is the same
calculation as before. This
completes the lens space case.
 
We now assume that $X^{(\epsilon_0)}$ has three exceptional fibers
whose multiplicities $(\alpha_1,\alpha_2, \alpha_3)$ are $(2,3,3)$,
$(2,3,4)$, $(2,3,5)$ or $(2,2,k)$ with $k \geq 2$. Then the graph
$\Gamma$ is star-shaped with three branches and with a central node
whose Euler number is $e \leq -2$. If $e \leq -3$, then $\widetilde z_1$
has multiplicity $1$ on each exceptional curve and we conclude using
Lemma \ref{le: core criterion} as in the lens space case. 

When $e=-2$, then the multiplicities of $\widetilde{z_1}$ are not all
equal to $1$, and we have to examine all the cases one by one. For
$(\alpha_1,\alpha_2, \alpha_3) = (2,3,5)$, there are eight possible
values for the Seifert pairs, which are $(2,1)$, $(3,\beta_2)$,
$(5,\beta_3)$ where $\beta_2 \in \{1,2\}$ and $\beta_3 \in
\{1,2,3,4\}$.  For example, in the case $\beta_2 = 1$ and $\beta_3 =
3$ the resolution graph $\Gamma$ is represented by
$$\splicediag{10}{24}{
  \undertag{\overtag{\circ}{-2}{8pt}}{(1)}{-3pt}\lineto[r]&\undertag{\overtag{\circ}{-2}{8pt}}{(2)\quad}{-3pt}\lineto[r]\lineto[dd]&
  \undertag{\overtag{\circ}{-3}{8pt}}{(1)}{-3pt}\ar@{->}[dr]\\&&&
  \\
  &\righttag{\lefttag{\circ}{-3}{2pt}}{(2)}{6pt}\lineto[dd]\ar@{->}[dr]\\&&&
  \\
  &\righttag{\lefttag{\circ}{-1}{2pt}}{(3)}{6pt}\lineto[dd]\\
  \\
  &\righttag{\lefttag{\circ}{-4}{2pt}}{(1)}{6pt}\ar@{->}[dr]\\&&& }$$
The arrows represent the strict transform of the generic linear
function $z_1$.  There are two thin zones obtained by deleting the
vertices with arrows and their adjacent edges, leading to four curves
$\gamma$ to be checked. Lemma \ref{le: core criterion} computes them
as $C_3$, $2C_5$, $C_5$ and $C_5$ respectively, where $C_p$ is the
exceptional fiber of degree $p$.  The other cases are easily checked in
the same way. The case $(2,2,k)$ gives an infinite family similar to
the lens space case.
\end{proof}

Proposition \ref{prop:fast loops} and either one of Theorem
\ref{th:fast loops} or Theorem \ref{th:thin fast loops} show:

\begin{corollary}\label{cor:conical1}
  $(X,0)$ is metrically conical if and only if there are no thin
  pieces. Equivalently, $\Gamma$ has only one
  node, and it is the unique \L-node.\qed
\end{corollary}

\begin{example} \label{ex:conical} Let $(X,0) \subset (\C^3,0)$,
  defined by $x^a+y^b+z^b = 0$ with $1\leq a<b$.  Then the graph
  $\Gamma$ is star-shaped with $b$ bamboos and the single \L-node is
  the central vertex. Therefore, $(X,0)$ is metrically conical. This
  metric conicalness was first proved in \cite{BFN2}. The rational
  singularities which are metrically conical have been determined by
  Pedersen \cite{pedersen}; they form an interesting three discrete
  parameter family.
\end{example}

\section{Existence and uniqueness of the minimal thick-thin
  decomposition}
\label{sec:uniqueness}

We first prove the existence:
\begin{lemma}\label{le:minimal}
  The thick-thin decomposition constructed in Section
  \ref{sec:thick-thin} is minimal, as defined in Definition
  \ref{def:minimal}.
\end{lemma}

\begin{proof} 
  Any real tangent line to $Z_j$ is a tangent line to some fast loop
  $\{\gamma_\epsilon\}_{0<\epsilon\le \epsilon_0}$ by Theorem
  \ref{th:thin fast loops}. For any other thick-thin decomposition
  this fast loop is outside any conical part of the thick part for all
  sufficiently small $\epsilon$ so its tangent line is a tangent line
  to a thin part. It follows that this thick-thin decomposition
  contains a thin piece which is tangent to the tangent cone of $Z_j$.
  Thus the first condition of minimality is satisfied. For the
    second condition, note that, according to the proof of Proposition \ref{prop:thick},   $\bigcup_{j=1}^sN(\Gamma_j)$  has
    conical complement and that the $N(\Gamma_j)$'s are pairwise
    disjoint except at the origin and each contains its respective
    $Z_j$. Consider some thick-thin decomposition of $(X,0)$ with thin
    pieces $Z'_i$, $i=1,\dots,s'$. Each $N(\Gamma_j)$ must have a
    $Z'_i$ in it since each $N(\Gamma_j)$ contains fast loops, and
    this $Z'_i$ is completely inside $N(\Gamma_j)$ (as a germ at 0)
    since its tangent space is contained in the tangent line to
    $Z_j$. Thus $s\le s'$.
\end{proof}

We restate the Uniqueness Theorem of the Introduction: 

\begin{theorem*}[\ref{th:uniqueness}] For any
    two minimal thick-thin decompositions of $(X,0)$ there exists $q>1$ and a
    homeomorphism of the germ $(X,0)$ to itself which takes one
    decomposition to the other and moves each $x\in X$ distance at
    most $|x|^q$.
\end{theorem*}

\begin{proof}It follows from the proof of Lemma \ref{le:minimal} that
  any two minimal thick-thin decompositions have the same numbers of
  thick and thin pieces. Let us consider the thick-thin decomposition
  $$(X,0)=\bigcup_{i=1}^r(Y_i,0)\cup\bigcup_{j=1}^s(Z_j,0)$$ constructed
  in Section \ref{sec:thick-thin} and another minimal thick-thin
  decomposition
  $$(X,0)=\bigcup_{i=1}^r(Y'_i,0)\cup\bigcup_{j=1}^s(Z'_j,0)\,.$$   
  They can be indexed so that for each $j$ the intersection $Z_j\cap
  Z_j'\cap (B_\epsilon\setminus\{0\})$ is non-empty for all small
  $\epsilon$ (this is not hard to see, but in fact we only need that
  $Z_j$ and $Z'_j$ are very close to each other in the sense that
  the distance between $Z_j\cap S_\epsilon$ and $Z'_j\cap S_\epsilon$
  is bounded by $c\epsilon^{q'}$ for some $c>0$ and $q'>1$, which is
  immediate from the proof of Lemma \ref{le:minimal}).

  Definition \ref{def:thick} says $Y=\bigcup_iY_i$ is the union of
  metrically conical subsets $Y_\epsilon\subset B_\epsilon$,
  $0<\epsilon\le \epsilon_0$. We will denote by
  $\partial_0Y_\epsilon:=\overline{\partial Y_\epsilon\setminus(Y_\epsilon\cap
  S_\epsilon)}$, the ``sides'' of these cones, which form a disjoint
  union of cones on tori. We do the same for  $Y'=\bigcup_iY'_i$.

  The limit as $\epsilon\to 0$ of the tangent cones of any component
  of $\partial_0Y_\epsilon$ is the tangent line to a $Z_j$. It follows
  that for all $\epsilon_1$ sufficiently small
  $\partial_0Y'_{\epsilon_1}$ will be ``outside'' $\partial_0
  Y_{\epsilon_0}$ in the sense that is disjoint from $Y_{\epsilon_0}$
  except at $0$. So choose $\epsilon_1\le\epsilon_0$ so this is
  so. Similarly, choose $0<\epsilon_3\le\epsilon_2\le\epsilon_1$ so
  that $\partial_0 Y_{\epsilon_2}$ is outside $\partial_0 Y'_{\epsilon_1}$
  and $\partial_0 Y'_{\epsilon_3}$ is outside $\partial_0 Y_{\epsilon_2}$.

  Now consider the $3$-manifold $M\subset S_{\epsilon_3}$ consisting
  of what is between $\partial_0Y_{\epsilon_0}$ and
  $\partial_0Y'_{\epsilon_3}$ in $X\cap S_{\epsilon_3}$. We can write
  $M$ as $M_1\cup M_2\cup M_3$ where $M_1$ is between
  $\partial_0Y_{\epsilon_0}$ and $\partial_0Y'_{\epsilon_1}$, $M_2$
  between $\partial_0Y'_{\epsilon_1}$ and $\partial_0Y_{\epsilon_2}$
  and $M_3$ between $\partial_0Y_{\epsilon_2}$ and
  $\partial_0Y'_{\epsilon_3}$.  Each component of $M_1\cup M_2$ is
  homeomorphic to $\partial_1M_2\times I$ and each component of
  $M_2\cup M_3$ is homeomorphic to $\partial_0M_2\times I$, so it
  follows by a standard argument that each component of $M_2$ is
  homeomorphic to $\partial_0M_2\times I$ ($M_2$ is an invertible bordism
  between its boundaries and apply Stalling's h-cobordism theorem
  \cite{stallings}; alternatively, use  Waldhausen's classification of
  incompressible surfaces in {\it surface}$\times I$
  in \cite{waldhausen}).

  Set $Z=\bigcup_{j=1}^s Z_j$ and $Z'=\bigcup_{j=1}^s Z'_j$.  The
  complement of the $\epsilon_3$-link of $Y_{\epsilon_2}$ in
  $X^{(\epsilon_3)}$ is the union of $Z^{(\epsilon_3)}$ and a collar
  neighborhood of its boundary and is hence isotopic to
  $Z^{(\epsilon_3)}$. It is also isotopic to its union with $M_2$
  which is the complement of the $\epsilon_3$-link of
  $Y'_{\epsilon_1}$. This latter is $(Z')^{(\epsilon_3)}$ union a
  collar, and is hence isotopic to $(Z')^{(\epsilon_3)}$. Thus the
  links of $Z$ and $Z'$ are homeomorphic, so $Z$ and $Z'$ are
  homeomorphic, and we can choose this homeomorphism to move any point
  $x$ distance at most $|x|^q$ for any $q<q'$, where $q'$ is defined
  above.
\end{proof}

We can also characterize, as in the following theorem, the unique
minimal thick-thin decomposition in terms of the analogy with the
Margulis thick-thin decomposition mentioned in the Introduction. The
proof is immediate from the proof of Lemma \ref{le:minimal}.

\begin{theorem}\label{th:M uniqueness}  
  The canonical thick-thin decomposition can be characterized among
  thick-thin decompositions by the following condition: For
  any sufficiently small $q> 1$ there exists $\epsilon_0>0$ such
  that for all $\epsilon\le \epsilon_0$ any point $x$ of the thin part
  with $|x|<\epsilon$ is on an essential loop in $X\setminus \{0\}$ of
  length $\le |x|^q$.\qed
\end{theorem}

In fact one can prove more (but we will not do so here): the set of
points which are on essential loops as in the above theorem gives
the thin part of a minimal thick-thin decomposition when intersected
with a sufficiently small $B_\epsilon$. 

We have already pointed out that the thick-thin decomposition is an
invariant of semi-algebraic bilipschitz homeomorphism. But in fact, if
one makes an arbitrary $K$-bilipschitz change to the metric on $X$,
the thin pieces can still be recovered up to homeomorphism using
the construction of the previous paragraph plus some $3$-manifold
topology to tidy up the result. Again, we omit details.

\section{Metric tangent cone}\label{sec:tangent cone}

The thick-thin decomposition gives a description of the metric tangent
cone of Bernig and Lytchak \cite{BL} for a normal complex surface
germ. The \emph{metric tangent cone} $\mathcal T_0A$ is defined for
any real semi-algebraic germ $(A,0)\subset (\R^N,0)$ as the
Gromov-Hausdorff limit  $$\mathcal T_0A:=\ghlim_{\epsilon\to
  0}\,(\frac1\epsilon A,0)\,,$$ where $\frac1\epsilon A$ means $A$ with
its inner metric scaled by a factor of $\frac1\epsilon$.

As a metric germ, $\mathcal T_0A$ is the strict cone on its link
$\mathcal T_0A^{(1)}:=\mathcal T_0A\cap S_1$, and 
$$\mathcal T_0A^{(1)}=\ghlim_{\epsilon\to 0}\frac1\epsilon
A^{(\epsilon)}\,,$$ where $\frac 1\epsilon A^{(\epsilon)}$ is the link
of radius $\epsilon$ of $(A,0)$ scaled to lie in the unit sphere
$S_1\subset \R^N$.

Applying this to $(X,0)\subset \C^n$, we see that the thin zones
collapse to circles as $\epsilon\to 0$. In particular, the boundary
tori of each Seifert manifold link $\frac1\epsilon Y_i^{(\epsilon)}$
of a thick zone collapse to the same circles, so that in the limit we
have a ``branched Seifert manifold'' (where branching of $k>1$ sheets
of the manifold meeting along a circle occurs if the collapsing map of
a boundary torus to $S^1$ has fibers consisting of $k>1$ circles).
Therefore, the link $\mathcal T_0X^{(1)}$ of the metric tangent cone is
the union of the branched Seifert manifolds glued along the circles to
which the thin zones have collapsed. 

The ordinary tangent cone $T_0X$ and its link $T_0X^{(1)}$ can be
similarly constructed as the Hausdorff limit of $\frac1\epsilon X$
resp.\ $\frac1\epsilon X^{(\epsilon)}$ (as embedded metric spaces in
$\C^n$ resp.\ the unit sphere $S_1$). In particular, there is a
canonical finite-to-one projection $\mathcal T_0X\to T_0X$ (described
in the more general semi-algebraic setting in \cite{BL}), whose degree
over a general point $p\in T_0X$ is the multiplicity of $T_0X$ at that
point. This is a branched cover, branched over the exceptional tangent
lines in $T_0X$.  

The circles to which thin zones collapse map to the links of some of
these exceptional tangent lines, but there can also be branching in
the part of $\mathcal T_0X$ corresponding to the thick part of $X$;
such branching corresponds to bamboos on \L-nodes of the resolution of
$X$ of Section \ref{sec:thick-thin} or to basepoints of
  the family of polars which are smooth points of $\pi^{-1}(0)$ on
  \L-curves. Summarizing:

\begin{theorem}
  The metric tangent cone $\mathcal T_0X$ is a cover of the
  tangent cone $T_0X$ branched over some of the exceptional lines in
  $T_0X$. It has a natural complex structure (lifted from
  $T_0X$) as a non-normal complex surface. Removing a finite set of
  complex lines (corresponding to thin zones) results in a complex
  surface which is homeomorphic to the interior of the thick part of
  $X$ by a homeomorphism which is bilipschitz outside an arbitrarily
  small cone neighborhood of the removed lines.\qed
\end{theorem}

\vglue 10pt plus 6pt minus 6pt\goodbreak\vglue 0pt plus 6pt minus 6pt
\centerline{\bf Part II: The bilipschitz classification}

\bigskip The remainder of the paper builds on the thick-thin
decomposition to complete the bilipschitz classification, 
as described in Theorem
\ref{th:classification}.

\section{ The refined decomposition of $(X,0)$ }\label{sec:decomp}
The decomposition of $(X,0)$ for the Classification Theorem
\ref{th:classification} was described there in terms of the link
$X^{(\epsilon)}$ as follows: first refine the thick-thin decomposition
$X^{(\epsilon)}=\bigcup_{i=1}^rY^{(\epsilon)}_i\cup\bigcup_{j=1}^sZ^{(\epsilon)}_j$
by decomposing each thin zone $Z_j^{(\epsilon)}$ into its JSJ
decomposition (minimal decomposition into Seifert fibered manifolds
glued along their boundaries) while leaving the thick zones
$Y_i^{(\epsilon)}$ as they are; then thicken some of the gluing
tori of this refined decomposition to collars $T^2\times I$. We
will call the latter \emph{special annular} pieces.

In this section we describe where these special annular pieces are
added, to complete the description of the data \eqref{it:1.9.1} of
Theorem \ref{th:classification}.

We consider a generic linear projection $\ell\colon (X,0) \to
(\C^2,0)$ and we denote again by $\Pi$ its polar curve and by $\Pi^*$
the strict transform with respect to the resolution defined in Section
\ref{sec:thick-thin}.  Before adding the special annular pieces, each
piece of the above refined decomposition is either a thick zone,
corresponding to an \L-node of our resolution graph $\Gamma$, or a
Seifert fibered component of the minimal decomposition of a thin zone,
which corresponds to a \T-node of $\Gamma$.  The incidence graph for
this decomposition is the graph whose vertices are the nodes of
$\Gamma$ and whose edges are the maximal strings $\sigma$ between
nodes. We add a special annular piece corresponding to a string
$\sigma$ if and only if $\Pi^*$ meets an exceptional curve belonging
to this string.

This refines the incidence graph of the decomposition by adding a
vertex on each edge which gives a special annular piece.  As in the
Introduction, we call this graph $\Gamma_0$. We then have a
decomposition
\begin{equation}\label{eq:pre}
   X^{(\epsilon)}=\bigcup_{\nu\in V(\Gamma_0)} M_\nu^{(\epsilon)} \,,
\end{equation}
into Seifert fibered components, some of which are special annular.

In the next section we will use this decomposition, together with the
additional data described in Theorem \ref{th:classification}, to
construct a bilipschitz model for $(X,0)$. To do so we will need to
modify the decomposition \eqref{eq:pre} by replacing each separating torus of
the decomposition by a toral annulus $T^2\times I$:  
\begin{equation}
  \label{eq:fine_decomp}
    X^{(\epsilon_0)} = \bigcup_{\nu \in V(\Gamma_0)}M_\nu^{(\epsilon_0)} 
\cup \bigcup_{\sigma \in E(\Gamma_0)}A_{\sigma}^{(\epsilon_0)}\,.
\end{equation}

At this point the only data of Theorem \ref{th:classification} which
have not been described are the weights $q_\nu$ of part
\eqref{it:1.9.3}.  These are certain Puiseux exponents of the
discriminant curve $\Delta$ of the generic plane projection of $X$, and
will be revealed in sections \ref{sec:carrousel} and
\ref{sec:carrousel1}, where we show that $X$ is bilipschitz
homeomorphic to its bilipschitz model. Finally, we prove that the data
are bilipschitz invariants of $(X,0)$ in Section \ref{sec:invariance}.

\section{The bilipschitz model of $(X,0)$}\label{sec:model}

We describe how to build a bilipschitz model for a normal surface
germ $(X,0)$ by gluing individual pieces using the data of Theorem
\ref{th:classification}.  Each piece will be topologically the cone on
some manifold $N$ and we call the subset which is the cone on
$\partial N$ the \emph{cone-boundary} of the piece. The pieces will be
glued to each other along their cone-boundaries using isometries. We
first define the pieces.

\begin{definition}[$A(q,q')$]\label{def:p2}
Here $1\le q< q'$ are rational numbers. 

Let $A$ be the euclidean annulus $\{(\rho,\psi):1\le \rho\le 2,\, 0\le
\psi\le 2\pi\}$ in polar coordinates and for $0<r\le 1$  let
$g^{(r)}_{q,q'}$ be the metric on $A$:
$$g^{(r)}_{q,q'}:=(r^q-r^{q'})^2d\rho^2+((\rho-1)r^q+(2-\rho)r^{q'})^2d\psi^2\,.
$$ 
So $A$ with this metric is isometric to the euclidean annulus with
inner and outer radii $r^{q'}$ and $r^q$. The metric completion
of $(0,1]\times S^1\times A$ with the metric
$$dr^2+r^2d\theta^2+g^{(r)}_{q,q'}$$ compactifies it by adding a single point at
$r=0$. We call the result $A(q,q')$.

(To make the comparison with the local metric of Nagase \cite{nagase}
clearer we note that this metric is bilipschitz equivalent to
$dr^2+r^2d\theta^2+r^{2q}\bigl(ds^2+(r^{q'-q}+s)^2)
d\psi^2\bigr)$
with $s=\rho-1$.)
\end{definition}
\begin{definition}[$B(F,\phi,q)$]\label{def:p1}
  Let $F$ be a compact oriented $2$-manifold, $\phi\colon F\to F$ an
  orientation preserving diffeomorphism, and $F_\phi$ the mapping
  torus of $\phi$, defined as:
$$F_\phi:=([0,2\pi]\times F)/((2\pi,x)\sim(0,\phi(x)))\,.$$
Given a rational number $q>1$ we will define a metric space
$B(F,\phi,q)$, which is topologically the cone on the mapping torus
$F_\phi$.
 
For each $0\le \theta\le 2\pi$ choose a Riemannian metric $g_\theta$
on $F$, varying smoothly with $\theta$, such that for some small
$\delta>0$:
$$
g_\theta=
\begin{cases}
g_0&\text{ for } \theta\in[0,\delta]\,,\\
\phi^*g_{0}&\text{ for }\theta\in[2\pi-\delta,2\pi]\,.
\end{cases}
$$
Then for any $r\in(0,1]$ the metric $r^2d\theta^2+r^{2q}g_\theta$ on
$[0,2\pi]\times F$ induces a smooth metric on $F_\phi$. Thus
$$dr^2+r^2d\theta^2+r^{2q}g_\theta$$ defines a smooth metric on
$(0,1]\times F_\phi$. The metric completion of $(0,1]\times F_\phi$
adds a single point at $r=0$. Denote this completion by
$B(F,\phi,q)$. (It can be thought of as a ``globalization'' of the local metric
of Hsiang and Pati \cite{hsiang-pati}). 

Note that changing $\phi$ by an isotopy or changing the initial choice
of the family of metrics on $F$ does not change the bilipschitz class
of $B(F,\phi,q)$. It will be convenient to make some
additional choices. For a boundary component $\partial_iF$ of $F$ let
$m_i(F)$ be the smallest $m>0$ for which
$\phi^{m}(\partial_iF)=\partial_iF$. By changing $\phi$ by an isotopy
if necessary and choosing the $g_\theta$ suitably we may assume:
\begin{itemize}
\item $\phi^{m_i(F)}$ is the
  identity on $\partial_iF$ for each $i$;
\item in a neighborhood of $\partial F$ we have $g_\theta=g_0$ for
  all $\theta$;
\item the lengths of the boundary components of $F$ are $2\pi$.
\end{itemize}
Then the metric $r^2d\theta^2+r^{2q}g_\theta$ on the boundary
component of $F_\phi$ corresponding to $\partial_iF$ is the product of
a circle of circumference $2\pi m_i(F) r$ and one of circumference $2\pi
r^q$.
\end{definition}
\begin{definition}[$CM$]\label{def:p3}
 Given a compact smooth $3$-manifold $M$,
  choose a Riemannian metric $g$ on $M$ and consider the metric
  $dr^2+r^2g$ on $(0,1]\times M$. The completion of this adds a point
  at $r=0$, giving a \emph{metric cone on $M$}. The bilipschitz class
  of this metric is independent  of choice of $g$, and we will choose
  it later to give the boundary components of $M$ specific desired shapes.
\end{definition}

A piece bilipschitz equivalent to $A(q,q')$ will also be said to be
``of type $A(q,q')$'', of briefly ``of type $A$'', and similarly for
types $B(F,\phi, q)$ and $CM$. We now describe how to glue together
pieces of these three types to obtain our bilipschitz model for
$(X,0)$. Note that, although we are only interested in metrics up to
bilipschitz equivalence, we must glue pieces by strict isometries of
their cone-boundaries. In order that the cone-boundaries are strictly
isometric, we may need to change the metric in Definition \ref{def:p1}
or \ref{def:p3} by replacing the term $r^2d\theta^2$ by
$m^2r^2d\theta^2$ for some positive integer $m$. This gives the same
metric up to bilipschitz equivalence.

For example, given $F$, $\phi$, $q$ and $q'$, each component $C$ of
the cone-boundary of $B(F,\phi,q)$ is isometric to the left boundary
component of $A(q,q')$ after altering the metrics on $B(F,\phi,q)$ and
$A(q,q')$ as just described, using $m$ equal to the number of
components of $F$ and the number of components of $F\cap C$
respectively.  So we can glue $B(F,\phi,q)$ to $A(q,q')$ along $C$,
giving a manifold with piecewise smooth metric. Similarly, a piece
$B(F,\phi,q')$ can be glued to the right boundary of
$A(q,q')$. Finally, since the left boundary of a piece $A(1,q')$ is
strictly conical, it can be glued to a boundary component of a conical
piece $CM$ (again, after suitably adjusting the metric to be correct
on the appropriate boundary component of $M$).

Consider now the decomposition
\eqref{eq:fine_decomp} of the previous section,
which is the decomposition of part \eqref{it:1.9.1} of
Theorem \ref{th:classification} but with the gluing tori thickened to
annular pieces.
We consider also the weights $q_\nu$ of part \eqref{it:1.9.3}
of Theorem \ref{th:classification}. 

Recall that, except for adding annular pieces, this decomposition is
obtained from the thick-thin decomposition by JSJ-decomposing the thin
zones.  

\begin{remark}\label{rem:transversality}
  Any JSJ decomposition can be positioned uniquely up to isotopy
to be transverse to the foliation by fibers of a fibration over $S^1$
(as long as the leaves are not tori).  Indeed, by the
  Transversality Lemma of Roussarie and Thurston
  (\cite{roussarie,thurston}), we can modify the JSJ decomposition by
  an isotopy in such a way that each separating torus and boundary
  torus is transversal to all leaves of the foliation. Then, using
  Waldhausen (\cite[Satz 2.8]{waldhausen}), we can assume that, up to
  bilipschitz equivalence, the leaves are transversal to the
  Seifert fibers in each Seifert fibered component. 
\end{remark}

In particular, this applies to the foliation of part
\eqref{it:1.9.2} of Theorem \ref{th:classification}, so we have a
foliation with compact leaves on each non-thick piece of this
decomposition. For a piece $A_{\sigma}^{(\epsilon_0)}$ the leaves are
annuli and we define
$$\widehat A_\sigma:=A(q_\nu,q_{\nu'})\,,$$
where $\nu$ and $\nu'$  are  nodes at
the ends of the string $\sigma$, ordered so that $q_\nu<q_{\nu'}$.
For a non-thick
piece $M_\nu^{(\epsilon)}$ the leaves are fibers of a fibration
$M_\nu^{(\epsilon)} \to S^1$.  In this case let $\phi_\nu\colon
F_\nu\to F_\nu$  be the monodromy map of this fibration and define
$$\widehat M_\nu := B(F_\nu,
\phi_\nu, q_\nu)\,.$$ Finally, for each node $\nu$ of $\Gamma$ such
that $q_\nu=1$ (i.e., $M_\nu^{(\epsilon)}$ is thick) we 
define $$\widehat M_\nu := CM_\nu^{(\epsilon)}\,.$$

We can glue together the pieces $\widehat M_\nu$ and $\widehat A_\sigma$
according to the topology of the decomposition
$X^{(\epsilon)}=\bigcup_{\nu} M_\nu^{(\epsilon)} \cup \bigcup _{\sigma
} A_\sigma^{(\epsilon)}$, arranging, as described above, that the
gluing is by isometries of the cone-boundaries. We can also make sure
that the foliations match to give the foliations of the thin zones
(item \eqref{it:1.9.2} of Theorem \ref{th:classification}).  We obtain
a semi-algebraic set
$$(\widehat{X},0) = \bigcup_{\nu }  (\widehat M_\nu,0)  
\cup \bigcup_{\sigma } ( \widehat A_\sigma,0)\,.$$

The next two sections prove that  $(\widehat{X},0)$ is bilipschitz equivalent to  $(X,0)$ while determining the $q_\nu$'s of item
\eqref{it:1.9.3} of Theorem \ref{th:classification} in the
process. The proof of the theorem will then be completed in the
following Section \ref{sec:invariance} by showing that the data of
Theorem \ref{th:classification} are bilipschitz invariants of $(X,0)$.

\section{Carrousel decomposition and lifting} 
\label{sec:carrousel}

In this section, we will define a decomposition of $(X,0)$ into $A$-,
$B$- and $CM$-pieces such that the polar wedges $A_0$ (see Proposition
\ref{prop:polar wedge}) are $B(D^2,id,s)$ pieces. In order to do this,
we define such a decomposition of $\C^2$ with respect to the
discriminant curve $\Delta$ of a generic projection $\ell \colon (X,0)
\to (\C^2,0)$ and we then lift this decomposition by $\ell$.

  We assume as usual that $\ell=(z_1,z_2)$ is a generic linear
  projection. We will work again inside the Milnor balls with corners 
  $B_\epsilon$  defined in Section \ref{sec:balls}. We denote by 
  $(x,y)$ the coordinates in $\C^2$ (so $(x,y)=(z_1,z_2)$).

  We first define the decomposition inside a conical neighborhood of
  each tangent line to $\Delta$.  Let $L_1,\ldots,L_m$ be the tangent
lines to the discriminant curve $\Delta$ of $\ell$. For each
$j=1,\ldots,m$, let $\Delta_j$ the union of branches of $\Delta$ which
are tangent to $L_j$ and assume that $L_j$ is the line $y=\lambda_j
x$. We consider a cone $V_j=\{(x,y):|x|\le \epsilon, |y-\lambda_j
x|\le \eta |x|\}\subset \C^2$ centered at $L_j$ such that $\Delta_j
\cap B_{\epsilon}$ lies inside $V_j$.
  




Following the ideas of L\^e \cite{Le} (see also
\cite{le-michel-weber}), we first construct a ``carrousel
decomposition'' of each $V_j, j=1,\ldots, m$ with respect to the branches of $\Delta_j$.
It will consist of closures of regions between successively smaller
neighborhoods of the successive 
Puiseux approximations of the branches of $\Delta_j$. As such, it is
finer than the one of \cite{Le}, which only
needed the first Puiseux pairs of the branches of $\Delta_j$.

  \begin{figure}[ht]
    \centering
\begin{tikzpicture} 
\begin{scope}[scale=0.7]
\draw[](8,0)ellipse (1 and 2.99) ;
\draw(0,0)--(8,3);
\draw(0,0)--(8,-3);
\draw[thick](0,0) .. controls (5,0) and (8,1)..  (10,3) ;
\draw[thick](0,0) .. controls (5,0) and (8,-1).. (10,-2.5) ;
\draw[thick](0,0) .. controls (7,0) and (8.5,-0.7)..  (11,-1.3) ;

\draw[fill](8,1.5)ellipse(1.3pt and 2pt);
\draw[fill](7.6,-1.2)ellipse(1.3pt and 2pt);
\draw[fill](8.4,-0.68)ellipse(1.3pt and 2pt);
    

\draw[thin,>=stealth,->](8,0)--(11,0);
\draw[thin,>=stealth,->](0,0)--(0,3);
\draw[thin,>=stealth,->](0,0)--(1.5,0.65);
    
\draw[thin](4,0)ellipse (0.5 and 1.48) ;

\draw[fill](4,0.27)ellipse(0.5pt and 1pt);
\draw[fill](3.93,-0.25)ellipse(0.5pt and 1pt);
\draw[fill](4.08,-0.11)ellipse(0.5pt and 1pt); 
    
    \node(a)at(11.3,0){$x$};
        \node(b)at(8,3.5){$V_j(x_0)$};
           \node(b)at(2,-1.3){$V_j $};
              \node(c)at(0.5,2){$y$};
                \node(c)at(10.3,3){$\Delta_j$};
                 \node(c)at(10.3,-2.5){$\Delta_j$};
                  \node(c)at(11.3,-1.3){$\Delta_j$};
                  
     \end{scope}             
\end{tikzpicture}
  \caption{The cone $V_j$ and the curve $\Delta_j$}
    \label{fig:6.1(3)}
\end{figure}
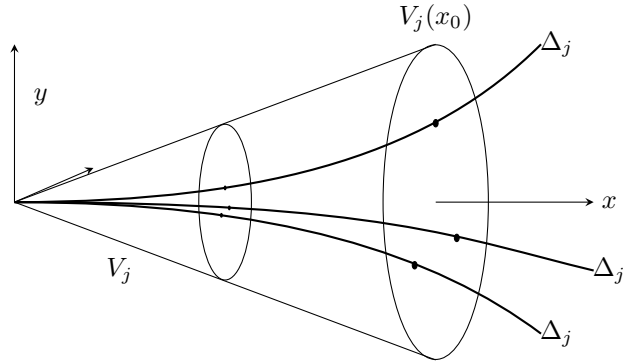

For simplicity we assume first that $\Delta_j$ has just one
branch. The curve $\Delta_j$ admits a Puiseux series expansion of the
form 
$$y=\sum_{i\ge 1} a_ix^{p_i}\in \C\{x^{\frac1N}\}\,.$$ 
Let $p_n=s$  be the contact exponent (see Proposition
  \ref{prop:polar wedge}).

For each  $k\le n$
choose $\alpha_k,\beta_k,\gamma_k>0$ with
$\alpha_k<|a_{k}|-\gamma_k<|a_{k}|+\gamma_k<\beta_k$  and 
consider the region 
$$B_k:=\Bigl\{(x,y): \alpha_k|x^{p_k}|\le |y-\sum_{i=1}^{k-1}a_ix^{p_i}|\le
\beta_k|x^{p_k}|,~ |y-\sum_{i=1}^{k}a_ix^{p_i}|\ge \gamma_k|x^{p_k}|\Bigr\}\,.$$

If the $\epsilon$ of our Milnor ball $B_{\epsilon}$ is small enough the $B_j$'s
will be pairwise disjoint. Denote by $A_1$ the closure of the region between
$\partial V_j$ and $B_1$ and $A_i$ the closure of the region between
$B_{i-1}$ and $B_{i}$ for $i=1,\dots n$. Finally let $D$ be
$$D:=\overline{V_j\setminus(\bigcup_{i=1}^{n}A_i\cup\bigcup_{i=1}^{n}B_i)}\,,$$
which is the union of connected pieces $D_1,\dots,D_n,D_{n+1}$, disjoint
except at $0$, and indexed so that $D_k$ is adjacent to $B_k$ and
$D_{n+1} \setminus \{0\}$ intersects $\Delta$.

Then, for $k=1,\dots,n$, $A_k$ 
is bilipschitz equivalent to $A(p_{k-1},p_{k})$ (Definition
\ref{def:p2}; we put $p_0=1$), and $B_k$ 
is bilipschitz equivalent to  $B(F_k,\phi_k, p_k)$  (Definition
\ref{def:p1}), where $F_k$ is a
planar surface with 
$$2+\frac{\lcm_{i\le
  k}\{\operatorname{denom}(p_i)\}}{\lcm_{i<k}\{\operatorname{denom}(p_i)\}}$$
boundary components and $\phi_k$ is a finite order
diffeomorphism. Finally, each $D_k$ is bilipschitz to $B(D^2,id,
p_k)$.

More generally, if $\Delta_j$ has several components, we first
truncate the Puiseux series for each component of $\Delta_j$ at its
contact exponent. Then for
each pair $\kappa=(f, p_k)$ consisting of a Puiseux polynomial
$f=\sum_{i=1}^{k-1}a_ix^{p_i}$ and an exponent $p_k$ for which
$f=\sum_{i=1}^{k}a_ix^{p_i}$ is a partial sum of the truncated Puiseux
series of some component of $\Delta_j$, we consider all components of
$\Delta_j$ which fit this data. If $a_{k1},\dots,a_{kt}$ are the
coefficients of $x^{p_k}$ which occur in these Puiseux polynomials we
define
\begin{align*}
  B_\kappa:=\Bigl\{(x,y):~&\alpha_\kappa|x^{p_k}|\le |y-\sum_{i=1}^{k-1}a_ix^{p_i}|\le
\beta_\kappa|x^{p_k}|\\ 
&
|y-(\sum_{i=1}^{k-1}a_ix^{p_i}+a_{kj}x^{p_k})|\ge
\gamma_\kappa|x^{p_k}|\text{ for }j=1,\dots,t\Bigr\}\,.
\end{align*}
Here $\alpha_\kappa,\beta_\kappa,\gamma_\kappa$ are chosen so that
$\alpha_\kappa<|a_{kj}|-\gamma_\kappa<|a_{kj}|+\gamma_\kappa<\beta_\kappa$ for each
$j=1,\dots,t$.  Again, if the $\epsilon$ of our Milnor ball is small
enough the sets $ B_\kappa$ will be disjoint for different $\kappa$,
and each is again bilipschitz equivalent to some $B(F,\phi,p_k)$ with
$\phi$ of finite order. The closure of the complement in $V_j$ of the union
of the $B_\kappa$'s is again a union of pieces bilipschitz equivalent to some
$A(q,q')$ or some $B(D^2,id,q)$.

The carrousel picture is the picture of the intersection of this
decomposition of the cone $V_j$ with the plane $x=\epsilon$.  Figure
\ref{fig:carrousel} shows this picture for a discriminant $\Delta_j$
having two branches with Puiseux  series expansions up to their
  contact exponents respectively $y=ax^{4/3}+bx^{13/6}$ and
$y=cx^{7/4}$. Note that the intersection of a piece of the
decomposition of $V_j$ with the disk $V_j\cap \{x=\epsilon\}$ will
usually have several components.
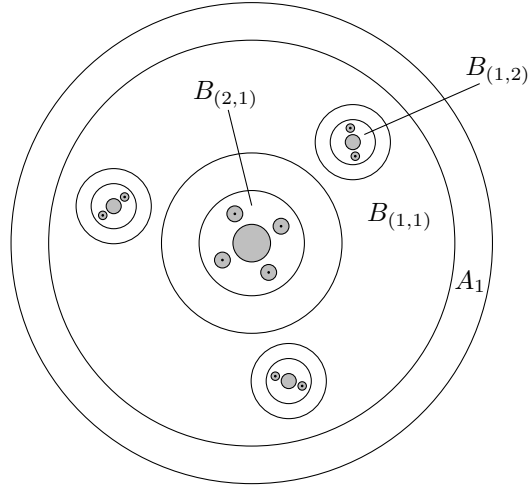
\begin{figure}[ht]
  \centering
\begin{tikzpicture}
\draw[ ] (0,0)circle(3.2cm);

\draw[ ] (0,0)circle(2.7cm);
\draw[ ] (0,0)circle(1.2cm);   

\draw[ ][ fill=lightgray]  (0,0)circle(0.25cm);
\draw[ ] (0,0)circle(0.7cm);
  \draw[ fill=lightgray] (0:0)+(30:0.45)circle(3pt);
    \draw[ fill] (0:0)+(30:0.45)circle(0.3pt);
  \draw[ fill=lightgray] (0:0)+(120:0.45)circle(3pt);
    \draw[ fill] (0:0)+(120:0.45)circle(0.3pt);
  \draw[ fill=lightgray] (0:0)+(210:0.45)circle(3pt);
    \draw[ fill ] (0:0)+(210:0.45)circle(0.3pt);
  \draw[ fill=lightgray] (0:0)+(300:0.45)circle(3pt);
    \draw[ fill] (0:0)+(300:0.45)circle(0.3pt);

  \draw[ ]  (45:1.9)circle(0.5);
  \draw[ ]  (45:1.9)circle(0.3);
  \draw[ ][ fill=lightgray]   (45:1.9)circle(0.1);
    \draw[ fill=lightgray] (45:1.9)+(100:0.19)circle(1.6pt);
      \draw[ fill ] (45:1.9)+(100:0.19)circle(0.3pt);
    \draw[ fill=lightgray] (45:1.9)+(280:0.19)circle(1.6pt);
      \draw[ fill ] (45:1.9)+(280:0.19)circle(0.3pt);

  \draw[ ]  (165:1.9)circle(0.5);
  \draw[ ]  (165:1.9)circle(0.3);
  \draw[ ][ fill=lightgray]   (165:1.9)circle(0.1);
     \draw[ fill=lightgray] (165:1.9)+(40:0.19)circle(1.6pt);
       \draw[ fill] (165:1.9)+(40:0.19)circle(0.3pt);
     \draw[ fill=lightgray] (165:1.9)+(220:0.19)circle(1.6pt);
       \draw[ fill] (165:1.9)+(220:0.19)circle(0.3pt);
          
  \draw[ ]  (285:1.9)circle(0.5);
  \draw[ ]  (285:1.9)circle(0.3);
  \draw[ ][ fill=lightgray]   (285:1.9)circle(0.1);
    \draw[ fill=lightgray] (285:1.9)+(-20:0.19)circle(1.6pt);
      \draw[ fill] (285:1.9)+(-20:0.19)circle(0.3pt);
    \draw[ fill=lightgray] (285:1.9)+(160:0.19)circle(1.6pt);
      \draw[ fill] (285:1.9)+(160:0.19)circle(0.3pt);

\draw[thin](35:3.7) --(44:2.08);
\draw[thin](100:1.8) --(90:0.5);


\node(a)at(-10:2.95){   $A_1$};
\node(a)at(10:2){   $B_{(1,1)}$};
\node(a)at(35:4){   $B_{(1,2)}$};
\node(a)at(100:2){   $B_{(2,1)}$};
  \end{tikzpicture} 
  \caption{Carrousel section for
    $\Delta=\{y=ax^{4/3}+bx^{13/6}+\ldots\}\cup\{y=cx^{7/4} +
    \ldots\}$. The region $D$ is gray.}
  \label{fig:carrousel}
\end{figure}

\smallskip We now take a decomposition as above in each of
  the cones $V_j$ and add the additional piece
  $Y:=(\overline{\C^2\setminus\bigcup V_j}, 0)$ to get a decomposition
  of all of $(\C^2,0)$. This piece $Y$ is conical, i.e., of type $CM$.

We next describe a lifting by $\ell$ of this decomposition of
$(\C^2, 0)$ to $(X,0)$.  

For each $j$ denote by $\Delta_{j1},\dots,\Delta_{jk_j}$  the
components of the discriminant which are in $V_j$ and
$B'_{j1},\dots,B'_{jk_j}$ the corresponding sets  described in the
remark following Proposition \ref{prop:polar wedge}. Let
$A_{j1},\dots ,A_{jk_j}$ be the corresponding polar wedges,
described as components of the inverse images of the $B'_{ji}$ (as in
Proposition \ref{prop:polar wedge}.\eqref{it:LZ2} and the remark),
with contact exponents $s_{j1}, \dots ,s_{jk_j}$.

Notice that each $B'_{ji}$ is a union of carrousel pieces. In
particular, if $k_j=1$, so the part of the discriminant in $V_j$ is
irreducible, then it is a piece of type $B(D^2,id,s_{j0})$.

Set $$(X^L,0):= \bigg(\overline{X\setminus\bigcup_j^t\bigcup_{i=1}^{k_j}A_{ji}},0\bigg),$$  and let $\ell^L\colon (X^L,0)\to (\C^2, 0)$ be
the restriction of $\ell\colon (X,0)\to (\C^2,0)$.
\begin{lemma}\label{lifting}
 \begin{enumerate}
 \item\label{it:12.1.1} Each polar wedge $A_{ji}$ is a piece of
   type $B(D^2, id, s_{ji})$;
\item\label{it:12.1.2} The inverse image by $\ell^L$ of the $CM$ piece
  $Y$ of the carrousel is a union of $CM$ pieces. The inverse image by
  $\ell^L$ of any $A(q,q')$ piece of the carrousel is a disjoint union
  of $A(q,q')$ pieces. The inverse image by $\ell^L$ of any
  $B(F,\phi,q)$ piece is a disjoint union of pieces of type
  $B(F,\phi,q)$ (with the same $q$ but usually different $F$'s and
  $\phi$'s), and the $\phi$'s for these pieces have finite order, so
  their links are Seifert fibered.
\end{enumerate}
\end{lemma}
\begin{proof}Part \eqref{it:12.1.1} is proved in Proposition
  \ref{prop:polar wedge}. In part \eqref{it:12.1.2}
 each inverse image piece is an
  unbranched cover of a piece of the carrousel, with metric
  bilipschitz equivalent to the lifted metric, so this part of the lemma follows.\end{proof}

We have now constructed a decomposition of $(X,0)$ into
``model pieces'' but the decomposition is not minimal. We will next
describe how to
simplify it step by step, leading to a decomposition equivalent to the
the one described in Section \ref{sec:model}.

\section{Bilipschitz equivalence with the model}\label{sec:carrousel1}

We will start with the above non-minimal decomposition of
$(X,0)$ into model pieces and then simplify it until it is
minimal. This will use the following trivial lemma, which gives rules
for simplifying a decomposition into model pieces. The reader can
easily add more rules; we intentionally list only the ones we use. 
 \begin{lemma}\label{le:rules} 
In this lemma $\cong$ means bilipschitz equivalence
and $\cup$ represents gluing along appropriate boundary components
by an isometry. 
\begin{enumerate}
\item\label{rule:D2} $B(D^2,\phi,q)\cong B(D^2,id,q)$; $B(S^1\times
  I,\phi,q)\cong B(S^1\times I,id,q)$.
\item\label{rule:AA} $A(q,q')\cup A(q',q'')\cong A(q,q'')$.
\item\label{rule:FF} If $F$ is the result of gluing a surface
  $F'$ to a disk $D^2$ along boundary components then
  $B(F',\phi|_{F'},q)\cup B(D^2,\phi|_{D^2},q)\cong B(F,\phi,q)$.
\item\label{rule:AD} $A(q,q')\cup B(D^2,id,q')\cong B(D^2,id,q)$.
\item\label{rule:CM}Each $B(D^2,id,1)$, $B(S^1\times I, id,
    1)$ or $B(F, \phi, 1)$ piece is a $CM$ piece and a union of $CM$
    pieces glued along boundary components is a $CM$ piece.\qed
\end{enumerate}
\end{lemma}

We will apply these rules repeatedly to amalgamate pieces in
$(X,0)$. But first we will introduce a restriction on the use of rule
\ref{rule:AA} which will avoid the amalgamation of some specific
annular pieces.
  
  Notice that in rule \ref{rule:AA} (as well as in rule
  \ref{rule:AD}), there is no restriction on $q,q'$ and $q''$. In
  particular, $q=q'$ or $q'=q''$ {\it i.e.,} amalgamation of pieces
  $A(q',q')$ pieces, is a priori allowed. In the carrousel
  decomposition of $(\C^2,0)$ annular pieces of type $A(q,q')$ only
  occur with $q<q'$, but in $(X,0)$ a $B(S^1\times I, id , q')$ piece
  (which can be thought of as a piece of type $A(q',q')$) can arise
  when amalgamating lifted pieces. The following lemma clarifies how
  such $A(q',q')$ pieces can arise in the decomposition of $(X,0)$.
 \begin{lemma}\label{le:13.2}
   If an annular piece of type $A(q', q')$ arises in $(X,0)$ it
   is either:
   \begin{enumerate}
   \item an unbranched cover of an annular region of the carrousel
   (arising as a union
   of a $B$-piece along with the pieces which ``fill the holes'' of
   this $B$-piece), or 
\item\label{it:sa}  a branched cover of a region of type $B(D^2,id,q)$ in the carrousel.
   \end{enumerate}
In case (1)  the two adjacent pieces to the $A(q',q')$ piece are of type $A(q,q')$ and
$A(q',q'')$ with $q<q'<q''$, while in case (2) the two adjacent pieces
are of type $A(q,q')$ with $q<q'$.
 \end{lemma}
 \begin{proof} For the piece to be annular its intersection with a
   Milnor fiber (i.e., intersection with $\{z_1=x_0\}$) must be a
   union of annuli $S^1\times I$. Euler characteristic shows that an
   annulus can branched cover only an annulus (with no branching) or a
   disk. The two possibilities give the two cases in the lemma. The
   statement about adjacent pieces is seen by considering the adjacent
   pieces to the image of the $A(q',q')$ piece in $(\C^2,0)$.
 \end{proof}

 A piece as in item \eqref{it:sa} of Lemma \ref{le:13.2} is called a
 \emph{special annular piece}.  We will now apply the
 rules of Lemma $\ref{le:rules}$ to amalgamate pieces of the
 decomposition, with the following additional rule:
\begin{enumerate}
\item[(6)] A special annular piece is never amalgamated with another piece;
\end{enumerate}

The amalgamation process stops when there is no
$B(D^2,id,q)$ piece left, no piece of type $A(q,q)$ which is not
special annular, no pairs of adjacent annular pieces which are not
special annular, and no pairs of adjacent $CM$ pieces.
At each step of the process there is at least one piece of type
$A(q,q')$ between any two pieces of type
$CM$, $B(F,\phi,q)$ with $F\ne D^2$, or special annular. So at the end
of the process we just have pieces of these three types with $A(q,q')$
pieces separating them.

We call this decomposition
  the \emph{classifying decomposition of $(X,0)$} and 
we create an incident graph $\Gamma_0$ for it with a vertex
for each connected piece which is of type $CM$, $B$ or special annular, and with
an edge for each annular piece connecting two such pieces.  

 \begin{lemma}\label{le:well reduced}
   $\Gamma_0$ is  the graph $\Gamma_0$ of
   Theorem \ref{th:classification} and Section \ref{sec:decomp}.
 \end{lemma}
\begin{proof}Note first that rule \eqref{rule:CM} of
    Lemma \ref{le:rules} absorbs into the initial $CM$ pieces
    (components $\ell^{-1}(Y)$, where
    $Y:=(\overline{\C^2\setminus\bigcup V_j}, 0)$) any adjacent
    $D$-type piece. These are any component of an $\ell^{-1}(V_i)$ which
    maps by a covering map to $V_i$, any piece coming from a moving
    polar (see proof of part \eqref{lem1} of Lemma \ref{le:2}) and any
    piece coming from a bamboo on an \L-node. In particular, part
    \eqref{lem3} of Lemma \ref{le:2} is proved, and we also see that
    the $CM$ pieces of the classifying decomposition are exactly the
    ``conical cores'' of the thick pieces of the thick-thin
    decomposition (i.e., when amalgamated with the adjacent annular
    pieces of type $A(1,q)$, they are the thick pieces).

     After removing the thick pieces we are left with the thin pieces,
     and  for each thin piece $Z_j$ and
     its link $Z_j^{(\epsilon)}$ we need to show the following:
   \begin{enumerate}
   \item[(i)]If one ignores the annular and special annular pieces the
     decomposition of  $Z_j^{(\epsilon)}$ is its minimal decomposition
     into Seifert fibered components;
   \item [(ii)]there is a special annular piece between two Seifert
     fibered components if and only if the strict transform of the polar
     meets the corresponding string of rational curves in the
     resolution graph.
   \end{enumerate}
   For (i), note that the JSJ decomposition of a thin zone is
   characterized by the fact that the pieces are not annular and the
   Seifert fibrations of adjacent pieces do not match up to isotopy on
   the common torus. 

   For pieces which do not have a special annular piece between them
   this is clear: The Seifert fibrations run parallel to the curves of
   Puiseux approximations to the branches of the $\Delta_j$ and are
   determined therefore by the weights $q$; two adjacent pieces with
   no special annular piece between them are separated by an $A(q,q')$
   and, since $q\ne q'$, the fibrations do not match.

   For pieces with a special annular piece between them it is also not
   hard to see. In this case there is an annular region between them
   which is composed of pieces $A(q_1,q')$ and $A(q_2,q')$ glued to a
   $B(S^1\times I, id,q')$ between them. We claim the mismatch of
   Seifert fibrations given by $A(q_1,q')$ and $A(q_2,q')$ accumulate
   rather than cancel.  A geometric approach is to consider an arc in
   the Milnor fiber $W_j\cap\{z_1=\epsilon e^{i\theta}\}$ from an
   intersection of a Seifert fiber of the one piece to an intersection
   of the Seifert fiber of the other, and watch what happens to this
   arc as one swings the Milnor fiber through increasing $\theta$,
   keeping the arc on the Milnor fiber and its ends on their
   respective Seifert fibers. If the Seifert fibrations were parallel
   the arc would have constant length. It passes through the special
   annulus, and we can assume it does this efficiently, so that its
   image in the image disk $V_j\cap\{z_1=x_0\}$ enters near a branch
   point, circles the branch point, and exits. The branch points in
   this disk are rotating faster than the endpoints of the arc since
   $q'>q_1$ and $q'>q_2$, so the curve must stretch (see Figure
   \ref{fig:A}).
   
   \begin{figure}[ht]
  \centering
\begin{tikzpicture}

   \begin{scope}[scale=0.7]
  \draw[ ] (-3,0)circle(1.6cm);
  \draw[ ] (-3,0)circle(1.8cm);
    \draw[ ] (-3,0)circle(0.6cm);
     \draw[ line width=0.4pt]  (-3.6,0)--(-3.3,0) ;
      \draw[ line width=0.4pt]  (180:3)+(0:0.6)--+(0:0.3) ;
     \draw[fill=lightgray] (180:3)+(180:0.25)circle(2pt);
       \draw[fill=lightgray] (180:3)+(0:0.25)circle(2pt);
    \draw[ line width=0.7pt]  (-4.6,0.1)--(-3.6,0.1) ;
\draw[line width=0.7pt ]  (-3.6,0.1) .. controls  (-3.4,0.1) and (-3.3,0.13)  .. (-3.25,0.13);
\draw[line width=0.7pt ]  (-3.6,-0.1) .. controls  (-3.4,-0.1) and (-3.3,-0.13)  .. (-3.25,-0.13);
\draw[line width=0.7pt ]  (-3.25,0.13) .. controls  (-3,0.13) and (-3,-0.13)  .. (-3.25,-0.13);
 \draw[ line width=0.7pt]  (-4.8,-0.1)--(-3.6,-0.1) ;
 
   \draw[>-stealth,->](-4,-0.1)--(-3.9,-0.1);
     \draw[>-stealth,->](-4,0.1)--(-4.1,0.1);

   \node(a)at(220:3){   $C$};
\node(a)at(205:3){   $C'$};
\node(a)at(-4,-0.4){   $\gamma$};

 \draw[ thin]    (0,1.7) .. controls  (1,2.2) and (6,2.2)  .. (7,1.7);
  \draw[ thin]    (0,-1.7) .. controls  (1,-1.2) and (6,-1.2)  .. (7,-1.7);

\node(a)at(7.9,1.7){   $\ell^{-1}(C)$};
\node(a)at(7.9,-1.7){   $\ell^{-1}(C')$};
\node(a)at(1.8,-1.8){   $\ell^{-1}(\gamma)$};
 \draw[ line width=0.4pt]  (2.5,-1.6)--(3.05,-1) ;

  \draw[ line width=0.4pt]  (3,0.75) .. controls  ( 3.5,0.95) and (4 ,0.95)  .. (4.5,0.75) ;
    \draw[ line width=0.4pt]  (3,-0.25) .. controls  ( 3.5,-0.55) and (4 ,-0.55)  .. (4.5,-0.25) ;
     \draw[ line width=0.4pt]  (3,0.75)--(3,-0.25) ;
       \draw[ line width=0.4pt]  (4.5,0.75)--(4.5,-0.25) ;
      \draw[fill=lightgray] (3,0.25)circle(4pt);
        \draw[fill=lightgray] (4.5,0.25)circle(4pt);

        \draw[ line width=0.4pt]  (1.5,0.6) .. controls  ( 2,0.9) and (2.5 ,0.9)  .. (3,0.75) ;
    \draw[ line width=0.4pt]  (1.5,-0.4) .. controls  ( 2,-0.5) and (2.5 ,-0.5)  .. (3,-0.25) ;
     \draw[ line width=0.4pt]  (1.5,0.6)--(1.5,-0.4) ;
         \draw[fill=lightgray] (1.5,0.1)circle(4pt);
 
 \draw[ line width=0.4pt ]  (1.5,0.6) .. controls  ( 1.3,0.7) and (1 ,0.75)  .. (0.7,0.75) ;
    \draw[ line width=0.4pt ]  (1.5,-0.4) .. controls  ( 1.3,-0.5) and (1 ,-0.6)  .. (0.7,-0.6) ;
     \draw[ line width=0.4pt, dotted]  (0,0.5) .. controls  ( 0.3,0.7) and (0.4 ,0.7)  .. (0.7,0.75) ;
    \draw[ line width=0.4pt, dotted ]  (0,-0.5) .. controls  ( 0.3,-0.7) and (0.4 ,-0.6)  .. (0.7,-0.6) ;

     \draw[ line width=0.4pt ]  (6,0.6) .. controls  ( 6.3,0.7) and (6.5 ,0.7)  .. (6.7,0.7) ;
    \draw[ line width=0.4pt, dotted ]  (6.7,0.7) .. controls  ( 7,0.7) and (7.3 ,0.6)  .. (7.4,0.5) ;
     \draw[ line width=0.4pt]  (6,-0.4) .. controls  ( 6.3,-0.5) and (6.5 ,-0.5)  .. (6.7,-0.5) ;
    \draw[ line width=0.4pt, dotted ]  (6.7,-0.5) .. controls  ( 7,-0.5) and (7.3 ,-0.4)  .. (7.4,-0.4) ;

        \draw[ line width=0.4pt]  (4.5,0.75) .. controls  ( 5,0.9) and (5.5 ,0.9)  .. (6,0.6) ;
    \draw[ line width=0.4pt]  (4.5,-0.25) .. controls  ( 5,-0.5) and (5.5 ,-0.5)  .. (6,-0.4) ;
             \draw[ line width=0.4pt]  (6,0.6)--(6,-0.4) ;
  \draw[fill=lightgray] (6,0.1)circle(4pt);

  \draw[ line width=0.7pt]  (3.1,-1.355)--(3.1,-0.1) ;
     \draw[ line width=0.7pt]  (3.1,-0.1) .. controls  (3.1,0)  and (3.25 ,0)  .. (3.25,0.2) ;
         \draw[ line width=0.7pt]  (3.1,0.55) .. controls  (3.1,0.4)  and (3.25 ,0.4)  .. (3.25,0.2) ;
    \draw[ line width=0.7pt]  (3.1,2.08)--(3.1,0.55) ;

   \draw[>-stealth,->](3.1,-1)--(3.1,-0.8);
     \draw[>-stealth,->](3.1,0.8)--(3.1,1);

     \begin{scope}[  yshift=-4.5cm]   
        \draw[ thin]    (0,1.7) .. controls  (1,2.2) and (6,2.2)  .. (7,1.7);
  \draw[ thin]    (0,-1.7) .. controls  (1,-1.2) and (6,-1.2)  .. (7,-1.7);

\node(a)at(7.9,1.7){   $\ell^{-1}(C)$};
\node(a)at(7.9,-1.7){   $\ell^{-1}(C')$};

  \draw[ line width=0.4pt]  (3,0.75) .. controls  ( 3.5,0.95) and (4 ,0.95)  .. (4.5,0.75) ;
    \draw[ line width=0.4pt]  (3,-0.25) .. controls  ( 3.5,-0.55) and (4 ,-0.55)  .. (4.5,-0.25) ;
     \draw[ line width=0.4pt]  (3,0.75)--(3,-0.25) ;
       \draw[ line width=0.4pt]  (4.5,0.75)--(4.5,-0.25) ;
      \draw[fill=lightgray] (3,0.25)circle(4pt);
        \draw[fill=lightgray] (4.5,0.25)circle(4pt);

         \draw[ line width=0.4pt ]  (1.5,0.6) .. controls  ( 1.3,0.7) and (1 ,0.75)  .. (0.7,0.75) ;
    \draw[ line width=0.4pt ]  (1.5,-0.4) .. controls  ( 1.3,-0.5) and (1 ,-0.6)  .. (0.7,-0.6) ;
     \draw[ line width=0.4pt, dotted]  (0,0.5) .. controls  ( 0.3,0.7) and (0.4 ,0.7)  .. (0.7,0.75) ;
    \draw[ line width=0.4pt, dotted ]  (0,-0.5) .. controls  ( 0.3,-0.7) and (0.4 ,-0.6)  .. (0.7,-0.6) ;

     \draw[ line width=0.4pt ]  (6,0.6) .. controls  ( 6.3,0.7) and (6.5 ,0.7)  .. (6.7,0.7) ;
    \draw[ line width=0.4pt, dotted ]  (6.7,0.7) .. controls  ( 7,0.7) and (7.3 ,0.6)  .. (7.4,0.5) ;
     \draw[ line width=0.4pt]  (6,-0.4) .. controls  ( 6.3,-0.5) and (6.5 ,-0.5)  .. (6.7,-0.5) ;
    \draw[ line width=0.4pt, dotted ]  (6.7,-0.5) .. controls  ( 7,-0.5) and (7.3 ,-0.4)  .. (7.4,-0.4) ;

        \draw[ line width=0.4pt]  (1.5,0.6) .. controls  ( 2,0.9) and (2.5 ,0.9)  .. (3,0.75) ;
    \draw[ line width=0.4pt]  (1.5,-0.4) .. controls  ( 2,-0.5) and (2.5 ,-0.5)  .. (3,-0.25) ;
     \draw[ line width=0.4pt]  (1.5,0.6)--(1.5,-0.4) ;
         \draw[fill=lightgray] (1.5,0.1)circle(4pt);

        \draw[ line width=0.4pt]  (4.5,0.75) .. controls  ( 5,0.9) and (5.5 ,0.9)  .. (6,0.6) ;
    \draw[ line width=0.4pt]  (4.5,-0.25) .. controls  ( 5,-0.5) and (5.5 ,-0.5)  .. (6,-0.4) ;
             \draw[ line width=0.4pt]  (6,0.6)--(6,-0.4) ;
  \draw[fill=lightgray] (6,0.1)circle(4pt);
 
  \draw[ line width=0.7pt]  (1.7,2) .. controls  (1.7,1.8)  and (3.1,1)  .. (3.1,0.8) ;
  \draw[ line width=0.7pt]  (3.1,-0.33)--(3.1,-0.1) ;
     \draw[ line width=0.7pt]  (3.1,-0.1) .. controls  (3.1,0)  and (3.25 ,0)  .. (3.25,0.2) ;
         \draw[ line width=0.7pt]  (3.1,0.55) .. controls  (3.1,0.4)  and (3.25 ,0.4)  .. (3.25,0.2) ;
    \draw[ line width=0.7pt]  (3.1,0.8)--(3.1,0.55) ;
      \draw[ line width=0.7pt]  (3.1,-0.33) .. controls  (3.1,-0.53)  and (4.3,-1.15)  .. (4.3,-1.355) ;

  \draw[>-stealth,->](3.93,-1)--(3.83,-0.92);
     \draw[>-stealth,->](2.3,1.46)--(2.2,1.54);

\end{scope}
        
 \begin{scope}[xshift=-6cm, yshift=-4.5cm]
   \draw[ ] (3,0)circle(1.6cm);
     \draw[ ] (3,0)circle(1.8cm);
       \draw[ ] (3,0)circle(0.6cm);
       
          \draw[ line width=0.4pt]  (0:3)+(90:0.6)--+(90:0.3) ;
      \draw[ line width=0.4pt]  (0:3)+(270:0.6)--+(270:0.3) ;
      
    \draw[fill=lightgray] (0:3)+(90:0.25)circle(2pt);
       \draw[fill=lightgray] (0:3)+(270:0.25)circle(2pt);
    
    \draw[line width=0.7pt ]  ( 2.9,0.6) .. controls  (2.9,0.4) and (2.87,0.3)  .. (2.87,0.25);
\draw[line width=0.7pt ]  ( 3.1,0.6) .. controls  (3.1,0.4) and (3.13,0.3)  .. (3.13,0.25);
\draw[line width=0.7pt ]  (2.87,0.25) .. controls  (2.87,0) and (3.13,0)  .. (3.13,0.25);
 \draw[ line width=0.7pt]   (2.9,0.6) .. controls  (2.6,1.2) and (1.7,-0.1)  .. (1.2,-0.1);
  \draw[ line width=0.7pt]   (3.1,0.6) .. controls  (2.8,1.7) and (1.6,0.1)  .. (1.4,0.1);

 \draw[>-stealth,->](1.89,0.27)--(1.9,0.275);
    \draw[>-stealth,->](2.2,0.77)--(2.19,0.765);
   \node(a)at(-40:3.2){   $C$};
\node(a)at(-24:2.8){   $C'$};
\node(a)at(2.1,0.1){   $\gamma$};
\end{scope}
\end{scope}
  \end{tikzpicture} 
  \caption{Non-matching of Seifert fibers across a special annular piece.}
  \label{fig:A}
\end{figure}
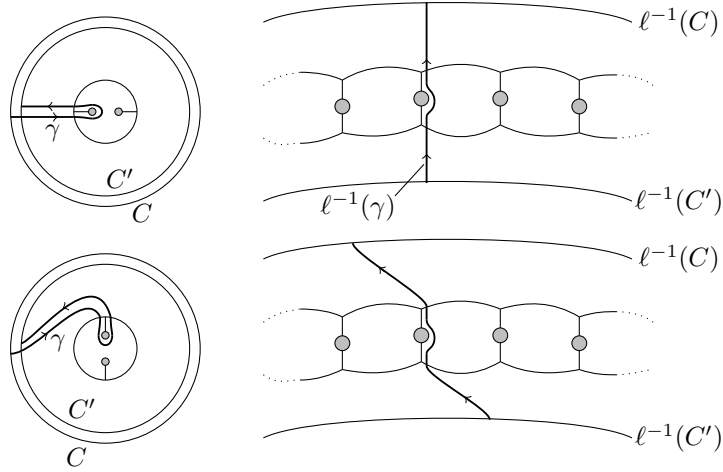

   We have proved the correspondence between the non-annular pieces of
   our decomposition and JSJ-components, and hence with \T-nodes in
   $\Gamma$. The correspondence (ii) can now be seen as follows. A
   string between nodes in $\Gamma$ for which the polar meets an
   exceptional curve corresponds to an annulus of the Nielsen
   decomposition of the Milnor fiber (see Remark \ref{rem:nielsen})
   which meets the polar. This means that the annulus covers
   its image in $V_j\cap\{z=\epsilon\}$ with non-trivial branching. By
   Lemma \ref{le:13.2} this means the image is a disk, so the annulus
   is the Milnor fiber of a special annulus.

This completes the proof that the decomposition graph for the classifying
decomposition is indeed the graph $\Gamma_0$ of Theorem
 \ref{th:classification} and Section \ref{sec:decomp}.  \end{proof}

The vertex-weights $q_\nu$
 are the $q$'s of the model pieces of this constructed
 decomposition. They are the last data we needed to assemble to
 complete the data of Theorem \ref{th:classification}.

 This decomposition of $X\cap B_{\epsilon_0}$ into model pieces 
 matches the construction of the model $\widehat X$ in Section
 \ref{sec:model}, so we have shown:
 \begin{corollary} The bilipschitz model $\widehat X$ constructed
   in Section \ref{sec:model} using the data we have assembled for
   Theorem \ref{th:classification} is bilipschitz
   homeomorphic to $X\cap B_{\epsilon_0}$.
   \qed
 \end{corollary}
This corollary completes the proof of half of Theorem
\ref{th:classification}. It remains to show that the data of that
theorem are bilipschitz invariants.

\begin{remark} Each vertex $\nu$ of $\Gamma_0$ corresponds either to a
  node of $\Gamma$ or, in the case of a special annular piece, to a
  maximal string $\sigma$ of $\Gamma$ which contains a vertex
  corresponding to an exceptional curve which intersect $\Pi^*$ (we
  call the minimal substring of $\sigma$ containing all such vertices
  an \emph{\A-string}, also an \emph{\A-node} if it is a single
  vertex). In discussing examples it is often
  more convenient to associate the weight $q_\nu$ with corresponding
  node or \A-string of $\Gamma$.

  Any rate $q_\nu$ is either $1$ or is a characteristic exponent of a
  branch of the discriminant curve $\Delta$. But in general, not all
  characteristic exponents appear as rates.

  For example, the singularity $(X,0)\subset (\C^3,0)$ with equation
  $x^3+y^4+z^4$ is metrically conical (see \ref{ex:conical}), so the
  only rate $q_\nu$ is $1$ at the single \L-node. But the discriminant
  curve of a generic linear projection consists of four transversal
  cusps each with Puiseux exponent $3/2$. This also
    happens when the contact exponent of a branch of $\Delta$ is
    strictly less than its last characteristic exponent.

  There are also examples where all characteristic exponents
  of $\Delta$ appear as rates $q_\nu$ of \T-nodes or \A-strings. This
  is the case for any $(X,0)$ with equation $z^2-f(x,y)=0$ where $f$
  is irreducible.  Indeed, let $\{(r_k,q_k)\}_{k=1}^n$  be the set of
  Puiseux pairs of $f$ (see \cite[p.~49]{EN}), i.e., the characteristic
  Puiseux exponents of the curve $f=0$ are the rational numbers $p_k =
  \frac{q_k}{r_1 \ldots r_k}$. The projection $\ell =(x,y)$ is generic
  and its discriminant curve is $\{f=0\}$. By Propositions
  4.1 and 4.2 of \cite{MN}, for any $k=1\ldots,n-1$, each connected
  component of $\ell^{-1}(B_k)$ gives rise to a \T-node of $\Gamma$
  with rate $p_k$. Moreover, either $r_n \neq 2$, and then
  $\ell^{-1}(B_n)$ gives rise to a \T-node with rate $p_n$, or $r_n=
  2$, and then $\ell^{-1}(B_n)$ gives rise to a \A-node with rate $p_n$.
\end{remark}

\section{Bilipschitz invariance of 
 the data}\label{sec:invariance}

In this section, we prove that the data \eqref{it:1.9.1},
\eqref{it:1.9.2} and \eqref{it:1.9.3} of Theorem
\ref{th:classification} are determined by the bilipschitz geometry of
$(X,0)$.

We first consider the refined decomposition of $(X,0)$ presented in
Section \ref{sec:decomp}, which is part $\eqref{it:1.9.1}$ of the data
of the Classification Theorem \ref{th:classification}:
$$X^{(\epsilon_0)} = \bigcup_{\nu \in {\mathcal
    N}_\Gamma}M_\nu^{(\epsilon_0)} \cup \bigcup_{\sigma \in {\mathcal
    S}_\Gamma}A_{\sigma}^{(\epsilon_0)}$$
\begin{lemma}\label{le:inv_pieces}
  The bilipschitz geometry of $(X,0)$ determines, up to homeomorphism,
  the above decomposition.
\end{lemma}
\begin{proof}
  The thick-thin decomposition
  $(X,0)=\bigcup_{i=1}^r(Y_i,0)\cup\bigcup_{j=1}^s (Z_j,0)$ is
  determined up to homeomorphism by the bilipschitz geometry of
  $(X,0)$ (Section \ref{sec:uniqueness}).  For each \L-node $\nu$, the
  associated thick piece $Y_\nu$ is defined by
  $Y_\nu=\pi(N(\Gamma_\nu))$ (Section \ref{sec:thick-thin}). Therefore
  the link $M_{\nu}^{(\epsilon_0)}$, which coincides with
  $Y_\nu^{(\epsilon_0)}$ up to a collar neighborhood, is determined
  by the bilipschitz geometry.  For each Tjurina component $\Gamma_j$
  of $\Gamma$, the pieces $M_\nu^{(\epsilon_0)}$ corresponding to \T-nodes in
  $\Gamma_j$ are exactly the Seifert components of the minimal JSJ
  decomposition of $Z_j^{(\epsilon_0)}$. Since the minimal JSJ
    decomposition of a $3$-manifold is unique up to isotopy, the
    pieces $M_\nu$ are determined up to homeomorphism by the
  bilipschitz geometry.

  It remains to prove that the special annular pieces $A_{\sigma}^{(\epsilon_0)}$ are also
  determined by the bilipschitz geometry. This will be done in Lemma
  \ref{le:bilrates} \end{proof}

We next prove the bilipschitz invariance of data \eqref{it:1.9.2} of
Theorem \ref{th:classification}.
\begin{lemma}
  The bilipschitz geometry of $(X,0)$ determines for each thin zone
  $Z_j^{(\epsilon)}$ the foliation of $Z_j^{(\epsilon)}$ by fibers of
  $\zeta_j^{(\epsilon)}$.
\end{lemma}
\begin{proof}
  We begin with a remark: If $ M \to S^1$ is a fibration of a compact
  connected oriented $3$-manifold, we consider $M$ as a foliated
  manifold, with compact leaves which are the connected components of
  the fibers. If the fibration has disconnected fibers, then it is the
  composition with a covering map $S^1\to S^1$ of a fibration with
  connected fiber $F$. This fibration gives the same foliation (and it
  is determined up to isotopy by $F\subset M$). We recall also that
  the isotopy class of a fibration $\zeta\colon M\to S^1$ of a
  $3$-manifold is determined up to isotopy by the homotopy class
  $[\phi]\in[M,S^1]=H^1(M;\Z)$ and has connected fibers if and only if
  this class is primitive. (See e.g., \cite{stallings}.)

  According to Proposition \ref{prop:thin1}, the fibration
  $\zeta_j^{(\epsilon)}\colon Z_j^{(\epsilon)}\to S^1_{\epsilon}$
  varies continuously with $\epsilon$, and the diameter of the fibers
  shrinks faster than linearly with $\epsilon$.  As just described, we
  can modify $\zeta^{(\epsilon)}_j$ if necessary to have connected
  fiber $F_j$ without changing the foliation by fibers. Any fibration
  $ \xi \colon Z_j^{(\epsilon)}\to S^1$ generating a different
  foliation up to isotopy has fibers which map essentially (i.e.,
    not null-homotopically) to $S^1$ by $\zeta_j^{(\epsilon)}$, so
  their diameter cannot shrink super-linearly with $\epsilon$. Hence
  our foliation is determined up to isotopy by the bilipschitz
  geometry.
\end{proof}

Recall that a closed curve (or ``loop'') in $X\setminus\{0\}$ is
essential if it is not null-homotopic in $X\setminus\{0\}$. A
continuous family of essential loops $\{\gamma_\epsilon\colon S^1\to
X^{(\epsilon)}\}_{0<\epsilon\le \epsilon_0}$ is a \emph{fast loop} (of
the first kind) if the lengths shrink faster than linearly in
$\epsilon$.  If $\operatorname{length}(\gamma_\epsilon)=O(\epsilon^q)$
we call $q$ the rate of the fast loop.

\begin{lemma}\label{le:nonparallel}
  For a Tjurina component $\Gamma_j$ of $\Gamma$ and a \T-node $\nu$
  in $\Gamma_j$, denote by $F_{\nu}^{(\epsilon)}$ a fiber of the
  fibration $\zeta_\nu^{(\epsilon)}\colon M_\nu^{(\epsilon)}\to S^1$
  obtained by restricting $\zeta_j^{(\epsilon)}$ to
  $M_\nu^{(\epsilon)}$.  Then $F_{\nu}^{(\epsilon_0)}$ contains an
  essential loop which is not homotopic into $\partial
  F_{\nu}^{(\epsilon_0)}$. Any such loop gives a fast loop of maximal
  rate $q_\nu$, so $q_\nu$ is a bilipschitz invariant.
\end{lemma} 
\begin{proof}
Note that a fast loop can only exist in a
    thin zone, and then an argument similar to the proof of the lemma
    above shows that there is no loss of generality in considering
    only fast loops which are homotopic into  Milnor fibers of the thin zone.
    Indeed, a loop which is not homotopic into a Milnor fiber must map
    essentially to the circle under the fibration $\zeta_j$ and hence
    cannot shrink faster than linearly. 
    
  Let $\gamma$ be a loop in $F_{\nu}^{(\epsilon_0)}$ which is
  neither null-homotopic in $F_{\nu}^{(\epsilon_0)}$ nor homotopic
  into a boundary component of $F_{\nu}^{(\epsilon_0)}$
  (since $\pi_1(F_\nu^{(\epsilon_0)})$ is a nonabelian free group, most
  loops have this property). Let $\delta$ be a boundary component of
  $F_{\nu}^{(\epsilon_0)}$. We can choose $\gamma$ so that, after
  connecting these loops to a base-point, neither of the loops $\gamma$ or
  $\gamma\delta$ is boundary-parallel. If both were  inessential
  then $\delta\sim \gamma^{-1}\gamma\delta$ would be
  inessential, but Proposition \ref{prop:fast boundaries} says each
  boundary component of $F_{\nu}^{(\epsilon_0)}$ is essential. Thus
  at least one of them is the desired essential loop in $F_\nu^{(\epsilon_0)}$.

  Since $M_\nu$ is of type $B(F_\nu,\phi,q_\nu)$, it is clear that
  every essential loop in $F_\nu^{(\epsilon_0)}$
  yields a fast loop of rate at least $q_\nu$.
  Now let $\gamma$ be an essential loop in $M_\nu^{(\epsilon_0)}$
  which is homotopic into  $F_{\nu}^{(\epsilon_0)}$ but not homotopic into a boundary component of $F_{\nu}^{(\epsilon_0)}$. Every
  representative of $\gamma$ intersects $M_{\nu}^{(\epsilon_0)}$, and
  there is a lower bound on the length of the portion of such
a  representative which lies in $M_{\nu}^{(\epsilon_0)}$, which is
realized within a  fiber   $F_{\nu}^{(\epsilon_0)}$. Since
  $F_{\nu}^{(\epsilon_0)}$ shrinks uniformly at rate $q_\nu$, it follows
  that $q_{\nu}$ is an upper bound on the rate of fast loops with
  $\gamma_{\epsilon_0}=\gamma$.
\end{proof}

\begin{lemma} \label{le:bilrates} The bilipschitz geometry of $(X,0)$
  determines the special annular pieces $M_\nu$ and the maximal rate
  $q_\nu$ of each such piece $M_\nu$.
\end{lemma}
\begin{proof} 
  Let $\gamma$ be a boundary component of the fiber $F_\nu$ for a
  \T-node $\nu$ and $\nu'$ the adjacent node for that boundary
  component.  Then $\gamma$ is a fast loop with rate at least
  $\max\{q_\nu,q_{\nu'}\}$. If there is no special annular piece between
  $M_\nu$ and $M_{\nu'}$ the argument of the previous lemma shows the
  maximal rate is exactly   $\max\{q_\nu,q_\nu'\}$, while if there is
  a special annular piece $M_{\nu''}$, then it is $q_{\nu''}$.
\end{proof}

Lemmas  \ref{le:nonparallel} and \ref{le:bilrates} prove 
bilipschitz invariance of all the $q_\nu$'s and complete the proof of Lemma
\ref{le:inv_pieces}. So the proof of Theorem
\ref{th:classification} is complete.

\section{Examples}\label{sec:examples}

In this section, we describe the bilipschitz geometry for several
examples of normal surface singularities $(X,0)$. Namely, for each of
them, we give the dual resolution graph $\Gamma$ of the minimal
resolution $\pi \colon (\widetilde X,E) \rightarrow (X,0)$ considered 
in Section  \ref{sec:decomp} with the additional data: 
 
 \begin{enumerate}
 \item[($1'$)] the nodes (\L-, \T- and \A-) represented by black
   vertices,
 \item[($2'$)] the arrows representing the strict transform of a
   generic linear form
\item[($3'$)]the rate $q_\nu$ weighting each node $\nu$
\end{enumerate}

The data ($1'$) and ($3'$) are equivalent respectively to
\eqref{it:1.9.1} and \eqref{it:1.9.3} of Theorem
\ref{th:classification}.

The data ($2'$) determines the multiplicities $m_\nu$ of the generic
linear form $z_1$ along each exceptional curve $E_\nu$, $\nu \in
\Gamma$.  In all the examples treated here, we are in one of the
following two situations: either the link $X^{(\epsilon)}$ is a
rational homology sphere or for each node $\nu \in \Gamma$,
$$\gcd \big( \gcd(m_\nu,m_\mu), \mu \in V_\nu) \big) =1,$$
where $V_\nu$ denotes the set of neighbor vertices of $\nu$ in
$\Gamma$ including arrows (which have multiplicities $1$).  In these
two situations, the multiplicities $m_\nu$ determine the embedded
topology in $X^{(\epsilon)}$ of the Milnor fiber of $z_1$, and then
data \eqref{it:1.9.2} of Theorem \ref{th:classification} (see for example
\cite{pichon}).

Recall that the $\cal L$-nodes of $\Gamma$ are the vertices carrying
at least one arrow or, equivalently, having rate $1$. The Tjurina
components of $\Gamma$ are the components obtained by removing the
$\cal L$-nodes and adjacent edges. Then, according to Section
\ref{sec:thick-thin} and Theorem \ref{th:main in intro}, the
thick-thin decomposition $(X,0)=\bigl(\bigcup (Y_i,0)\bigr) \cup
\bigl(\bigcup (Z_j,0)\bigr)$ of $(X,0)$ is read from the graph
$\Gamma$ as follows:
 \begin{itemize}
 \item The thin pieces $(Z_j,0)$ of $(X,0)$ correspond to the Tjurina
   components which are not bamboos.
 \item The thick pieces $(Y_i,0)$ of $(X,0)$ are in bijection with the
   $\cal L$-nodes of $\Gamma$. They correspond to the connected
   components of the graph 
   obtained by removing from $\Gamma$ the Tjurina components which are
   not bamboos and their adjacent edges.
\end{itemize}
Each example has been computed as follows: we start by computing the
graph $\Gamma$ with \L-nodes and arrows. Then for each Tjurina
component, we determine the \A-nodes and the rates at \A- and \T-nodes
either by studying the Puiseux expansions of the branches of the
discriminant curve $\Delta$ and the cover $\ell$, or by determining
the strict transform of the polar curve $\Pi^*$ using the equation of
$X$ and the multiplicities of the coordinates functions in the graph,
or by using Hironaka numbers (see Remark \ref{rem:hironaka}).
 
In \ref{ex:superexample} we will treat an example of a
superisolated singularity in all detail. We start by presenting the
thick-thin decomposition for this family of singularities.

\begin{example}[Superisolated singularities] \label{ex:super} A
hypersurface singularity $(X,0)$ with equation $ f_d (x,y,z)+ f_{d+1}
(x,y,z)+ \ldots =0$, where each $f_k$ is an homogeneous polynomial of
degree $k$, is {\it superisolated} if the projective plane curve $C :=
\{f_d=0\} \subset \mathbb P^2$ is reduced with isolated singularities
$\{p_i\}_i$ and these points $p_i$ are not situated on the projective
curve $\{f_{d+1}=0\}$ (see \cite{Lu}, \cite{LMN}). Such a singularity
is resolved by the blow-up $\pi_1$ of the origin of $\C^3$. Therefore,
the thick zones are in bijection with the components of the
projectivized tangent cone $C$.  Moreover, the resolution $\pi \colon
(\widetilde X,E) \rightarrow (X,0)$ introduced in Section 
\ref{sec:thick-thin} is obtained by composing $\pi_1$ with the minimal
resolutions of the plane curve germs $(C,p_i)$, and as the curve $C$
is reduced, no Tjurina component of $\pi$ is a bamboo. Therefore, the
thin zones are in bijection with the points $p_i$.

The embedded topological type of $(X,0)$ does not depend on the $f_k$,
$k >d$ providing $f_{d+1}=0$ does not contain any of the singular
points $p_i$ of $C$.  The previous arguments show the thick-thin
decomposition of $(X,0)$ also only depends on $C$ and does not change
when replacing the equation by $f_d + l^{d+1} = 0$, where $l$ is a
generic linear form.

For example, taking $C$ with smooth irreducible components
intersecting transversely, we obtain a thick-thin decomposition whose
thin zones are all thickened tori, and whose topology comes from the
thick pieces (in fact the genus of the irreducible components of $C$)
and cycles in the resolution graph.
\end{example}
\begin{example}[$(X,0)$ with equation $
  (zx^2+y^3)(x^3+zy^2)+z^7= 0$]
\label{ex:superexample}

This is a superisolated singularity and we compute the resolution
graph with L-nodes and strict transform of the generic linear form by
blowing-up the origin and then resolving the singularities of the
exceptional divisor.  We get two \L-curves intersecting at five points
and one visible Tjurina component in the graph. Blowing-up the five
intersection points, we then obtain five additional Tjurina components
each consisting of one node between the two \L-nodes. We
  obtain the resolution graph of Fig.~\ref{fig:7}.   The numbers in
  parentheses are the multiplicities of the generic linear form while
  the others are the self-intersection numbers (Euler classes) of the
  exceptional curves.

  \begin{figure}[ht]
  \begin{center}
\begin{tikzpicture}
  \draw[thin ](-2,0)--(-1,1);
  \draw[thin ](0:0)--(-1,1);
  \draw[thin ](0:0)--(1,1);
  \draw[thin ](1,1)--(2,0);
   \draw[thin ](1,1)--(1.5,2.5);
   \draw[thin ](-1,1)--(-1.5,2.5);
   \draw[ fill] (1.5,2.5)circle(2.5pt);
   \draw[ fill] (-1.5,2.5)circle(2.5pt);
   \draw[thin ](-1.5,2.5)--(1.5,2.5);
     
\draw[fill] (-1,1)circle(2.5pt);
\draw[fill] (1,1)circle(2.5pt);
\draw[fill] (0,0)circle(2.5pt);
\draw[fill=white] (2,0)circle(2.5pt);
\draw[fill=white] (-2,0)circle(2.5pt);

\draw[thin] (-1.5,2.5)..controls (-0.5,3) and (0.5,3)..(1.5,2.5);
\draw[thin] (-1.5,2.5)..controls (-0.5,2) and (0.5,2)..(1.5,2.5);
\draw[thin] (-1.5,2.5)..controls (-0.5,3.5) and (0.5,3.5)..(1.5,2.5);
\draw[thin] (-1.5,2.5)..controls (-0.5,1.5) and (0.5,1.5)..(1.5,2.5);
\node(a)at(-2,-0.35){-2};
\node(a)at(-2.4,0){(5)};

\node(a)at(0,-0.35){-5};
\node(a)at(0,0.4){(4)};

\node(a)at(2,-0.35){-2};
\node(a)at(2.4,0){(5)};

\node(a)at(-1,0.65){-1};
\node(a)at(-1.5,1){(10)};

\node(a)at(1,0.65){-1};
\node(a)at(1.5,1){(10)};

\node(a)at(-0.4,3.5){-1};
\node(a)at(0.3,3.5){(2)};
\node(a)at(-0.4,1.6){-1};
\node(a)at(0.3,1.5){(2)};
\node(a)at(-1.7,2.2){-23};
\node(a)at(-1.7,2.85){(1)};

\node(a)at(1.7,2.2){-23};
\node(a)at(1.7,2.85){(1)};

  \draw[thin,>-stealth,->](1.5,2.5)--+(1.2,0.4);
       \draw[thin,>-stealth,->](1.5,2.5)--+(1.3,0);
         \draw[thin,>-stealth,->](1.5,2.5)--+(1.2,-0.4);
         
            \draw[thin,>-stealth,->](-1.5,2.5)--+(-1.2,0.4);
       \draw[thin,>-stealth,->](-1.5,2.5)--+(-1.3,0);
         \draw[thin,>-stealth,->](-1.5,2.5)--+(-1.2,-0.4);

   \draw[ fill] (0,2.5)circle(2.5pt);
           \draw[ fill] (0,2.86)circle(2.5pt);
             \draw[ fill] (0,2.12)circle(2.5pt);
               \draw[ fill] (0,1.75)circle(2.5pt);
                 \draw[ fill] (0,3.25)circle(2.5pt);

  \end{tikzpicture} 
  \end{center} 
\caption{}\label{fig:7}
\end{figure}

 Computing the multiplicities of a generic linear
  combination of $f_x, f_y$ and $f_z$ as in Example \ref{ex:E8}, we
  obtain that the polar curve has 14 branches and that the resolution
  graph of the minimal resolution which resolves the basepoints of the
  family of polar curves is as in Fig.~\ref{fig:8} (the number at each
  vertex is the multiplicity of a generic combination $af_x + b f_y +
  cf_z$):
  \begin{figure}[ht]
  \begin{center}
\begin{tikzpicture}
  \draw[thin,>-stealth,->](-1.5,2.5)--+(-1.2,0.4);
   \draw[thin,>-stealth,->](-1.5,2.5)--+(-1.2,0);
    \draw[thin,>-stealth,->](-1.5,2.5)--+(-1.2,-0.4);
    
     \draw[thin,>-stealth,->](1.5,2.5)--+(1.2,0.4);
   \draw[thin,>-stealth,->](1.5,2.5)--+(1.2,0);
    \draw[thin,>-stealth,->](1.5,2.5)--+(1.2,-0.4);

 \draw[thin,>-stealth,->](-2,0)--+(-1.2,0.4);
  \draw[thin,>-stealth,->](2,0)--+(1.2,0.4);
    \draw[thin,>-stealth,->](0,0)--+(0,1.4);

    \draw[thin,>-stealth,->](0,3.25)--+(0.2,1.2);
     \draw[thin,>-stealth,->](0,2.86)--+(0.5,1.1);
                  \draw[thin,>-stealth,->](0,2.5)--+(0.6,1);
                   \draw[thin,>-stealth,->](0,2.12)--+(0.7,0.9);

              \draw[thin,>-stealth,->](0,1.75)--+(0.8,0.7);

 \draw[] (-2,0)circle(2.5pt);
  \draw[thin ](-2,0)--(-1,1);

  \draw[thin ](0:0)--(-1,1);
 
     \draw[thin ](0:0)--(1,1);

        \draw[thin ](1,1)--(2,0);
            \draw[] (2,0)circle(2.5pt);
   \draw[thin ](1,1)--(1.5,2.5);
    \draw[thin ](-1,1)--(-1.5,2.5);
    \draw[ fill] (1.5,2.5)circle(2.5pt);
     \draw[ fill] (-1.5,2.5)circle(2.5pt);
     \draw[thin ](-1.5,2.5)--(1.5,2.5);

\draw[fill] (-1,1)circle(2.5pt);
\draw[fill] (1,1)circle(2.5pt);
\draw[fill] (0,0)circle(2.5pt);
\draw[fill=white] (2,0)circle(2.5pt);
\draw[fill=white] (-2,0)circle(2.5pt);

\draw[thin] (-1.5,2.5)..controls (-0.5,3) and (0.5,3)..(1.5,2.5);
\draw[thin] (-1.5,2.5)..controls (-0.5,2) and (0.5,2)..(1.5,2.5);
\draw[thin] (-1.5,2.5)..controls (-0.5,3.5) and (0.5,3.5)..(1.5,2.5);
\draw[thin] (-1.5,2.5)..controls (-0.5,1.5) and (0.5,1.5)..(1.5,2.5);
\node(a)at(2,-0.35){29};
\node(a)at(1,0.65){57};
\node(a)at(0,-0.35){23};
\node(a)at(-1,0.65){57};
\node(a)at(-2,-0.35){29};

\node(a)at(-0.4,3.5){11};
\node(a)at(-0.4,1.6){11};
\node(a)at(-1.7,2.2){5};
\node(a)at(1.7,2.2){5};
 
     \draw[ fill] (0,2.5)circle(2.5pt);
           \draw[ fill] (0,2.86)circle(2.5pt);
             \draw[ fill] (0,2.12)circle(2.5pt);
               \draw[ fill] (0,1.75)circle(2.5pt);
                 \draw[ fill] (0,3.25)circle(2.5pt);

  \end{tikzpicture} 
  \end{center}
\caption{}\label{fig:8}\end{figure}

  Therefore, the discriminant curve $\Delta$ of a generic linear
  projection $\ell\colon (X,0) \to (C^2,0)$ has $14$ branches and $12$
  distinct tangent lines $L_1,\ldots,L_{12}$.
  
  Let $\Delta_0$ be a branch of $\Delta$ and let $\Pi_0$ be the branch
  of the polar curve such that $\ell(\Pi_0)=\Delta_0$. The
  intersection of $\Delta_0$ with the Milnor fiber of a generic form
  on $\C^2$ is the multiplicity of $\Pi_0$, so it is the denominator
  $d$ of the last characteristic Puiseux exponent of $\Delta_0$. 
  Considering the restriction $\ell|_{\Pi_0} \colon \Pi_0 \to
  \Delta_0 $, we obtain that $d$ divides the multiplicity of the
  generic linear form at the exceptional curve intersecting $\Pi_0^*$,
  which can be read in Fig.~\ref{fig:7}. This gives a bound on
  $d$, which enables one to compute the Puiseux expansion of all the
  branches $\Delta$ using, for example, Maple.
We obtain that $\Delta$     decomposes  as follows:
  \begin{itemize}
\item three branches $\Delta_1, \Delta_2, \Delta_3$ tangent to
  the same line $L_1$, each with a single characteristic Puiseux exponent,
  respectively $6/5, 6/5$ and $5/4$, and Puiseux
  expansions
  \begin{tabbing}
    $\Delta_1$:\qquad\= $u = av +  b v^{6/5} + c v^{7/5} +\ldots $\\
$\Delta_2$:\qquad\>  $u = av +b' v^{6/5} + c' v^{7/5} +\ldots $\\
$ \Delta_3$:\qquad\> $u =  av +b'' v^{5/4} + c'' v^{3/2}+ \ldots $
  \end{tabbing}
\item For each $i=2,\ldots,6$, one branch tangent to $L_i = \{u=a_i
  v\} $ with $3/2$ as single characteristic Puiseux exponent: $ u=a_iv
  + b_i v^{3/2} + \ldots $.
\item $6$ branches, each lifting to a component of the polar whose
  strict transform meets one of the two \L-curves at a smooth point
  (these are ``moving polar components'' in the
    terminology introduced in the proof of Lemma
    \ref{le:2}\eqref{lem1}). So these do not contribute to the bilipschitz
  geometry.
\end{itemize}

The lines $L_2,\ldots,L_6$ correspond to the $5$ thin zones coming from
the $5$ one-vertex Tjurina components (compare with Lemma
\ref{prop:intersecting L}). These five vertex are thus \A-nodes with rate
equal to the characteristic Puiseux exponent $3/2$.

The line $L_1$ is the projection of the exceptional line tangent to
the thin piece $Z_1$ associated with the big Tjurina component
$\Gamma_1$. In order to compute the rate $q_\nu$ at each node $\nu$ of
$\Gamma_1$, we would have to study explicitly the restriction of the
cover $\ell$ to $Z_1$ over a carrousel decomposition of its image and
the reduction to the bilipschitz model, as in Section
\ref{sec:carrousel1}.  The following remark will enable one to
circumvent this technical computation.
\begin{remark}\label{rem:hironaka}
  Let $\ell = (z_1,z_2)\colon (X,0) \to (\C,0)$ be a generic linear
  projection. Denote by $(u,v)=(z_1,z_2)$ the coordinates in the
  target $\C^2$. Let $\Gamma_j$ be a Tjurina component of $\Gamma$ and
  let $\Pi_j$ be a branch of the polar curve of $\ell$ contained in
  the thin piece $Z_j$. Suppose the branch $\Delta_j =
  \ell(\Pi_j)$ of the discriminant curve has a single characteristic
  Puiseux exponent $p$, so the Puiseux expansion of $\Delta_j$ is
  of the form:
$$v=a_1u + a_2 u^2 + \ldots + a_k u^k + b u^{p}+\ldots,$$
where $k$ is the integral part of $p$ and $b \neq
0$. We perform the change of variable $v'=v-\sum_{i=1}^k a_iu^i$. In
other word, we consider the projection $(z_1,z'_2)$ instead of
$(z_1,z_2)$, where $z'_2 = z_2 - \sum_{i=1}^k a_iz_1^i$. This
projection has the same polar and discriminant curve as $\ell$, and
in the new coordinates, the Puiseux expansion of $\Delta_j$ is
$v'=bu^{p}$. Let $\nu$ be a node in $\Gamma_j$ such that $
\Pi_j\subset M_\nu $, i.e. $q_\nu=p$. Then, according to
\cite{le-maugendre-weber, michel}, we have:
$$p=\frac{m_\nu(z'_2)}{m_\nu(z_1)}$$
The quotient $\frac{m_\nu(z'_2)}{m_\nu(z_1)}$ is called the {\it
  Hironaka number} of $\nu$ (with respect to the projection
$(z_1,z'_2)$).
\end{remark} 
We return to the example. Since each branch $\Delta_1$, $\Delta_2$
and $\Delta_3$ has a single characteristic exponent and no inessential
exponent before in the expansion, we can apply the remark above. We
use $z'_2 = x+y$. Let $E_1,\ldots,E_5$ be the exceptional curves
associated with the vertices of the big Tjurina component indexed from
left to right. We first compute the multiplicities of $z_1$, $x$ and
$y$ along the exceptional curves $E_i$. Then $z'_2 =x+y $ is a generic
linear combination of $x$ and $y$ and we obtain its multiplicities as
the minimum of that of $x$ and $y$ on each vertex:
\[
\begin{array}{cccccc}
& E_1&E_2&E_3&E_4& E_5 \\
 z_1: &  5 &10&4&10&5  \\
x: & 7 & 13&5&12&6   \\
y:  &  6&12&5&13&7   \\
x+y:  & 6&12&5&12&6   
\end{array}
\]
Therefore the Hironaka numbers at the vertices of the Tjurina
component $\Gamma_1$ are:
\[
\begin{array}{ccccc}
 E_1&E_2&E_3&E_4& E_5 \\
 6/5 & 6/5& 5/4& 6/5& 6/5
\end{array}
\]

We thus obtain that the rates at the two \T-nodes equal $6/5$, and
that the vertex in the middle is an \A-node with rate $5/4$.
The bilipschitz geometry of $(X,0)$ is then given by the following
graph. The rate at each node is written in italics.
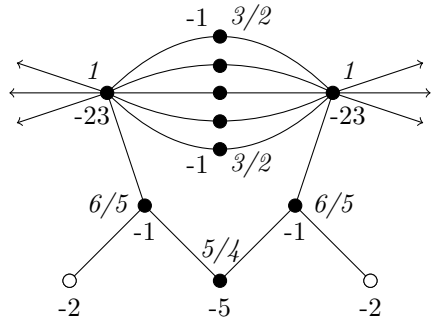
\begin{center}
\begin{tikzpicture}
  \draw[] (-2,0)circle(2.5pt);
  \draw[thin ](-2,0)--(-1,1);

  \draw[thin ](0:0)--(-1,1);
 
     \draw[thin ](0:0)--(1,1);
     
        \draw[thin ](1,1)--(2,0);
            \draw[] (2,0)circle(2.5pt);
   \draw[thin ](1,1)--(1.5,2.5);
    \draw[thin ](-1,1)--(-1.5,2.5);
    \draw[ fill] (1.5,2.5)circle(2.5pt);
     \draw[ fill] (-1.5,2.5)circle(2.5pt);
     \draw[thin ](-1.5,2.5)--(1.5,2.5);
     
\draw[fill ] (-1,1)circle(2.5pt);
\draw[fill  ] (1,1)circle(2.5pt);
\draw[fill] (0,0)circle(2.5pt);
\draw[fill=white] (2,0)circle(2.5pt);
\draw[fill=white] (-2,0)circle(2.5pt);
      
         \draw[ fill] (0,2.5)circle(2.5pt);
           \draw[ fill] (0,2.86)circle(2.5pt);
             \draw[ fill] (0,2.12)circle(2.5pt);
               \draw[ fill] (0,1.75)circle(2.5pt);
                 \draw[ fill] (0,3.25)circle(2.5pt);
     
\draw[thin] (-1.5,2.5)..controls (-0.5,3) and (0.5,3)..(1.5,2.5);
\draw[thin] (-1.5,2.5)..controls (-0.5,2) and (0.5,2)..(1.5,2.5);
\draw[thin] (-1.5,2.5)..controls (-0.5,3.5) and (0.5,3.5)..(1.5,2.5);
\draw[thin] (-1.5,2.5)..controls (-0.5,1.5) and (0.5,1.5)..(1.5,2.5);
\node(a)at(-2,-0.35){-2};
\node(a)at(0,-0.35){-5};
\node(a)at(2,-0.35){-2};
\node(a)at(-1,0.65){-1};
\node(a)at(1,0.65){-1};
\node(a)at(-0.3,3.5){-1};
\node(a)at(0.4,3.5){\it 3/2};
\node(a)at(-0.3,1.55){-1};
\node(a)at(0.4,1.55){\it 3/2};
\node(a)at(-1.7,2.8){\it 1};
\node(a)at(-1.7,2.2){-23};
\node(a)at(1.7,2.8){\it 1};
\node(a)at(1.7,2.2){-23};
\node(a)at(-1.5,1){\it 6/5};
\node(a)at(1.5,1){\it 6/5};
\node(a)at(0,0.4){\it 5/4};

 \draw[thin,>-stealth,->](1.5,2.5)--+(1.2,0.4);
       \draw[thin,>-stealth,->](1.5,2.5)--+(1.3,0);
         \draw[thin,>-stealth,->](1.5,2.5)--+(1.2,-0.4);
         
            \draw[thin,>-stealth,->](-1.5,2.5)--+(-1.2,0.4);
       \draw[thin,>-stealth,->](-1.5,2.5)--+(-1.3,0);
         \draw[thin,>-stealth,->](-1.5,2.5)--+(-1.2,-0.4);
\end{tikzpicture} 
\end{center} 
We now describe the same example via its
  carrousel. The carrousel sections for the six thin zones are
  illustrated somewhat schematically in Fig.~\ref{fig:9}. They consist
  of five carrousel sections $D_1,\dots, D_5$ for the five small thin
  zones containing components of the discriminant with Puiseux
  expansions of the form $y=a_ix+b_ix^{3/2}+\dots$ and one carrousel
  section $D_6$ for the ``large'' thin zone containing two components
  of the discriminant of the form $y=cx+d_ix^{6/5}+\dots$, $i=1,2$ and
  one of the form $y=cx+ex^{5/4}$. There are also six points
  representing the intersection points of components of the
  discriminant coming from the six moving polars in the thick zone.

For each $i=1,\dots,5$, $\ell^{-1}(D_i)$ consists of an annulus which
is a double branched cover of $D_i$ and four disks each of which
simply covers $D_i$ (the total degree must be the multiplicity of $X$,
which is $6$). So for each $i=1,\dots 5$ we see a special annular
piece between the two thick zones plus four $D$-pieces, two
contributing to each of the two thick zones. The carrousel section $D_6$
for the ``large'' thin zone represents two overlapping $B$-pieces
with rate $6/5$ connecting to a piece with rate $5/4$ inside. $\ell^{-1}(D_6)$
consists of two disks and a surface $F$ of genus $5$ which four-fold
covers $D_6$. The two disks correspond again to two $D$-pieces which
contribute to the thick zones, while the surface $F$ is made up of two
annuli over the inner disk of $D_6$ connecting to two surfaces of
genus $2$ over the outer annulus of $D_6$. The two annuli are in fact
two sections by the Milnor fiber of a single special annular piece of
rate $5/4$ connecting the two pieces of rate $6/5$ to form the
``large'' thin zone. 

\begin{figure}[ht]
\centering
\begin{tikzpicture} 
   \begin{scope}[yshift=5.5cm]
    \begin{scope}[yshift=-5.5cm]
     
    \draw[ fill=lightgray ] (2,0)circle(6pt);
  \draw[fill=black] ( 2.1,0)circle(0.5pt);
   \draw[] ( 1.9,0)circle(1.5pt);
 \draw[] ( 2.1,0)circle(1.5pt);
   \draw[fill=black] ( 1.9,0)circle(0.5pt);
   \node(a)at(2.55,0){$3/2$};
   
 \begin{scope} [xshift=1.5cm] 
 \draw[ fill=lightgray] (2,0)circle(6pt);
  \draw[fill=black] ( 2.1,0)circle(0.5pt);
   \draw[fill=black] ( 1.9,0)circle(0.5pt);
  \draw[] ( 2.1,0)circle(1.5pt);
   \draw[] ( 1.9,0)circle(1.5pt);
   \node(a)at(2.55,0){$3/2$};
  \end{scope}
  
   \begin{scope} [xshift=3cm] 
 \draw[ fill=lightgray] (2,0)circle(6pt);
  \draw[fill=black] ( 2.1,0)circle(0.5pt);
   \draw[fill=black] ( 1.9,0)circle(0.5pt);
  \draw[] ( 2.1,0)circle(1.5pt);
   \draw[] ( 1.9,0)circle(1.5pt);
   \node(a)at(2.55,0){$3/2$};
  \end{scope}
  
   \begin{scope} [xshift=4.5cm] 
 \draw[ fill=lightgray] (2,0)circle(6pt);
  \draw[fill=black] ( 2.1,0)circle(0.5pt);
   \draw[fill=black] ( 1.9,0)circle(0.5pt);
  \draw[] ( 2.1,0)circle(1.5pt);
   \draw[] ( 1.9,0)circle(1.5pt);
   \node(a)at(2.55,0){$3/2$};
  \end{scope}
  
   \begin{scope} [xshift=6cm] 
 \draw[ fill=lightgray] (2,0)circle(6pt);
  \draw[fill=black] ( 2.1,0)circle(0.5pt);
   \draw[fill=black] ( 1.9,0)circle(0.5pt);
  \draw[] ( 2.1,0)circle(1.5pt);
   \draw[] ( 1.9,0)circle(1.5pt);
   \node(a)at(2.55,0){$3/2$};
  \end{scope}

   \begin{scope}[yshift=-2.7cm]
    \draw[ ] (0:5) circle(2cm);

       \draw[ fill=lightgray] (5,0)circle(12pt);
  \draw[fill=black] ( 4.85,0.15)circle(0.5pt);
    \draw[fill=black] ( 4.85,-0.15)circle(0.5pt);
      \draw[fill=black] ( 5.15,0.15)circle(0.5pt);
    \draw[fill=black] ( 5.15,-0.15)circle(0.5pt);
  \draw[] ( 4.85,0.15)circle(2pt);
    \draw[] ( 4.85,-0.15)circle(2pt);
      \draw[] ( 5.15,0.15)circle(2pt);
    \draw[] ( 5.15,-0.15)circle(2pt);
    
       \draw[ fill=lightgray] (0:5)+(0:1.2)circle(3pt);
  \draw[fill=black] (0:5)+(0:1.2)circle(0.5pt);
  
       \draw[ fill=lightgray] (0:5)+(72:1.2)circle(3pt);
  \draw[fill=black] (0:5)+(72:1.2)circle(0.5pt);
  
       \draw[ fill=lightgray] (0:5)+(144:1.2)circle(3pt);
  \draw[fill=black] (0:5)+(144:1.2)circle(0.5pt);
  
       \draw[ fill=lightgray] (0:5)+(216:1.2)circle(3pt);
  \draw[fill=black] (0:5)+(216:1.2)circle(0.5pt);
  
       \draw[ fill=lightgray] (0:5)+(-72:1.2)circle(3pt);
  \draw[fill=black] (0:5)+(-72:1.2)circle(0.5pt);

       \draw[ fill=darkgray] (0:5)+(180:1.2)circle(3pt);
  \draw[fill=black] (0:5)+(180:1.2)circle(0.5pt);
  
       \draw[ fill=darkgray] (0:5)+(252:1.2)circle(3pt);
  \draw[fill=black] (0:5)+(252:1.2)circle(0.5pt);
  
       \draw[ fill=darkgray] (0:5)+(-36:1.2)circle(3pt);
  \draw[fill=black] (0:5)+(-36:1.2)circle(0.5pt);
  
       \draw[ fill=darkgray] (0:5)+(36:1.2)circle(3pt);
  \draw[fill=black] (0:5)+(36:1.2)circle(0.5pt);

       \draw[ fill=darkgray] (0:5)+(-252:1.2)circle(3pt);
  \draw[fill=black] (0:5)+(-252:1.2)circle(0.5pt);

     \node(a)at(-5:6.5){$6/5$};
          \node(a)at(1.6,-1){${\small 5/4}$};
          \draw(1.9,-1)--(4.8,0);
 
   \draw[fill=black] (1,1)circle(0.5pt);
           \draw[fill=black] (0.5,1)circle(0.5pt);
          \draw[fill=black] (0,1)circle(0.5pt);
             \draw[fill=black] (1,0)circle(0.5pt);
           \draw[fill=black] (0.5,0)circle(0.5pt);
          \draw[fill=black] (0,0)circle(0.5pt);

        \node(a)at(2,1.5){${\small 1}$};
    \end{scope}  
     \end{scope}  
          \end{scope}  
    \end{tikzpicture}  
\caption{}\label{fig:9}
\end{figure}

\end{example}
\begin{example}
  [Simple singularities] The bilipschitz geometry of the
  A-D-E singularities is given in Table \ref{tab:simple}. The right column
  gives the numbers of thin pieces and thick pieces.  According to
  Corollary \ref{cor:conical1}, when $(X,0)$ doesn't have a thin
  piece, then it has a single thick piece and $(X,0)$ is metrically
  conical. Otherwise $(X,0)$ is not metrically conical.
  \end{example}
\begin{table}[ht]
    \centering
\begin{tabular}{lcr}
Singularity & bilipschitz geometry & thick-thin pieces \\
&&\\
$A_1$ & 
  \begin{tikzpicture}
 \draw[fill ] (0,0)circle(2pt);   
 \draw[thin,>-stealth,->](0,0)--+(-0.6,0.6);
  \draw[thin,>-stealth,->](0,0)--+(0.6,0.6);
    \node(a)at(-0.4,0){\it 1};
 \end{tikzpicture} 
 & metrically conical \\

&& \\
 $ A_{2n+1}, n \geq 1$    &       \begin{tikzpicture}
 
\draw[fill ] (0,0)circle(2pt);
 \draw[thin,>-stealth,->](0,0)--+(-0.6,0.6);
  \draw[thin,>-stealth,->](4.4,0)--+(0.6,0.6);

    \draw[thin ](0,0)--(0.7,0);
      \draw[thin ](0.7,0)--(0.9,0);
     \draw[thin, dotted ](0.9,0)--(1.3,0);
        \draw[thin ](1.3,0)--(1.5,0);
          \draw[thin ](1.5,0)--(2.2,0);
           \draw[thin ](2.2,0)--(2.9,0);
           
   \draw[thin ](2.9,0)--(3.1,0);
     \draw[thin, dotted ](3.1,0)--(3.5,0);
        \draw[thin ](3.5,0)--(3.7,0);
              \draw[thin ](3.7,0)--(4.4,0);

\draw[fill=white ] (0.7,0)circle(2pt);
\draw[fill=white ] (1.5,0)circle(2pt);
\draw[fill  ] (2.2,0)circle(2pt);
\draw[fill=white ] (2.9,0)circle(2pt);
\draw[fill=white ] (3.7,0)circle(2pt);
\draw[fill  ] (4.4,0)circle(2pt);

\node(a)at(0,0.3){\it 1};
\node(a)at(2.2,0.3){\it n+1};
\node(a)at(4.4,0.3){\it 1};
 \end{tikzpicture} 
 
& 2 thick, 1 thin \\
 &&\\

   $A_{2n}, n\geq 1$   &     
  
   \begin{tikzpicture}
 
\draw[fill ] (0,0)circle(2pt);
 \draw[thin,>-stealth,->](0,0)--+(-0.6,0.6);
  \draw[thin,>-stealth,->](5.8,0)--+(0.6,0.6);

    \draw[thin ](0,0)--(0.7,0);
      \draw[thin ](0.7,0)--(0.9,0);
     \draw[thin, dotted ](0.9,0)--(1.3,0);
        \draw[thin ](1.3,0)--(1.5,0);
          \draw[thin ](1.5,0)--(2.2,0);
           \draw[thin ](2.2,0)--(2.9,0);
           
   \draw[thin ](2.9,0)--(3.6,0);
     \draw[thin ](3.6,0)--(3.8,0);
        \draw[thin, dotted](3.8,0)--(4.2,0);
              \draw[thin ](4.2,0)--(5.8,0);

\draw[fill=white ] (0.7,0)circle(2pt);
\draw[fill=white ] (1.5,0)circle(2pt);
\draw[fill =white  ] (2.2,0)circle(2pt);

\draw[fill = ] (2.9,0)circle(2pt);

\draw[fill=white  ] (3.6,0)circle(2pt);
\draw[fill=white ] (4.4,0)circle(2pt);
\draw[fill =white ] (5.1,0)circle(2pt);
\draw[fill = ] (5.8,0)circle(2pt);

\node(a)at(0,0.3){\it 1};
\node(a)at(2.9,0.3){\it n+1/2};
\node(a)at(4.4,0.3){\it 1};

\node(a)at(2.2,-0.3){ -3};
\node(a)at(2.9,-0.3){ -1};
\node(a)at(3.6,-0.3){ -3};

 \end{tikzpicture} 

&  2 thick, 1 thin \\
&  & \\
$D_4$ & 

 \begin{tikzpicture}
 \draw[fill =white ] (0,0)circle(2pt);
  \draw[fill =white ] (0.7,0.7)circle(2pt);
   \draw[thin ](0,0)--(0.7,0);
      \draw[thin ](0.7,0)--(0.7,0.7);
        \draw[thin ](0.7,0)--(1.4,0);

  \draw[fill   ] (0.7,0)circle(2pt);

\draw[thin,>-stealth,->](0.7,0)--+(0.6,0.6);
  \draw[fill =white ] (0,0)circle(2pt);
  \draw[fill =white ] (0.7,0.7)circle(2pt);
    \draw[fill =white ] (1.4,0)circle(2pt);

\node(a)at(0.7,-0.3){ \it 1};

  \end{tikzpicture}  & metrically conical \\
& & \\
$ D_{k}, k\geq 5$ &   
 \begin{tikzpicture}
 \draw[fill =white ] (0,0)circle(2pt);
  \draw[fill =white ] (0.7,0.7)circle(2pt);
   \draw[thin ](0,0)--(0.7,0);
      \draw[thin ](0.7,0)--(0.7,0.7);
         \draw[thin ](0.7,0)--(1.4,0);
          \draw[thin ](1.4,0)--(1.6,0);
           \draw[thin, dotted ](1.6,0)--(2,0);
            \draw[thin ](2,0)--(2.2,0);
             \draw[thin ](2.2,0)--(3.6,0);

  \draw[fill   ] (0.7,0)circle(2pt);
    \draw[fill   ] (2.9,0)circle(2pt);
 
\draw[thin,>-stealth,->](2.9,0)--+(0.6,0.6);
  \draw[fill =white ] (0,0)circle(2pt);
  \draw[fill =white ] (0.7,0.7)circle(2pt);
    \draw[fill =white ] (1.4,0)circle(2pt);
      \draw[fill =white ] (2.2
      ,0)circle(2pt);
\draw[fill =white ] (3.6,0)circle(2pt);

\node(a)at(0.7,-0.3){ \it k/2-1};
\node(a)at(2.9,-0.3){ \it 1};
  \end{tikzpicture} 
 &
 1 thick, 1 thin
 \\
 && \\
$E_6$ &
 \begin{tikzpicture}
 \draw[fill   ] (0,0)circle(2pt);
   \draw[fill   ] (0.7,0)circle(2pt);
 
   \draw[thin ](0,0)--(2.1,0);
      \draw[thin ](0.7,0)--(0.7,1.4);
    \draw[fill=white   ] (1.4,0)circle(2pt);
  \draw[fill=white  ] (2.1,0)circle(2pt);
     \draw[fill =white ] (0.7,0.7)circle(2pt);
        \draw[fill =white ] (0.7,1.4)circle(2pt);
 
\draw[thin,>-stealth,->](0,0)--+(-0.6,0.6);

\node(a)at(1,0.3){ \it 4/3};
\node(a)at(0.1,0.3){ \it 1};

  \end{tikzpicture} & 1 thick, 1 thin \\
 && \\
 $E_7$ & 
 \begin{tikzpicture}
  \draw[fill   ] (-0.7,0)circle(2pt);
   \draw[fill   ] (0.7,0)circle(2pt);
   \draw[thin ](-0.7,0)--(2.8,0);
      \draw[thin ](0.7,0)--(0.7,0.7);
  \draw[fill=white   ] (1.4,0)circle(2pt);
  \draw[fill=white  ] (2.1,0)circle(2pt);
     \draw[fill =white ] (0.7,0.7)circle(2pt);      
        \draw[fill =white ] (2.8,0)circle(2pt);
  \draw[fill=white   ] (0,0)circle(2pt);
\draw[thin,>-stealth,->](-0.7,0)--+(-0.6,0.6);
 
\node(a)at(1,0.3){ \it 3/2};
\node(a)at(-0.6,0.3){ \it 1};

  \end{tikzpicture} 
    & 1 thick, 1 thin \\
 &&\\
 $E_8$ & 
 \begin{tikzpicture}
  \draw[fill   ] (-2.1,0)circle(2pt);
   \draw[fill   ] (0.7,0)circle(2pt);
   \draw[thin ](-2.1,0)--(2.1,0);
      \draw[thin ](0.7,0)--(0.7,0.7);
  \draw[fill=white   ] (1.4,0)circle(2pt);
  \draw[fill=white  ] (2.1,0)circle(2pt);
     \draw[fill =white ] (0.7,0.7)circle(2pt);      
          \draw[fill=white   ] (0,0)circle(2pt);
\draw[thin,>-stealth,->](-2.1,0)--+(-0.6,0.6);
  \draw[fill=white   ] (-1.4,0)circle(2pt);
    \draw[fill=white   ] (-0.7,0)circle(2pt);

\node(a)at(1,0.3){ \it 5/3};
\node(a)at(-2,0.3){ \it 1};

  \end{tikzpicture} 
  &  1 thick, 1 thin
\\&&
\end{tabular}
     \caption{Bilipschitz geometry of simple singularities}~\\[-16pt]
    \label{tab:simple}
  \end{table}
\begin{example}
  [Hirzebruch-Jung singularities]
The Hirzebruch-Jung singularities (surface cyclic quotient
singularities) were described by Hirzebruch as the normalizations of
the singularities $(Y,0)$ with equation $z^p - x^{p-q}y =0$ in $\mathbb
C^3$, where $1\leq q<p$ are coprime positive integers. The link
of the normalization $(X,p)$ of $(Y,0)$ is the lens space $L(p,q)$
(see \cite{H1,HNK}) and the minimal resolution graph is a bamboo

 
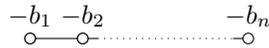
\begin{center}
  \begin{tikzpicture}
  \draw[fill=white] (0,0)circle(2pt);
   \draw[fill =white  ] (0.7,0)circle(2pt);
   \draw[thin ](0,0)--(0.9,0);
      \draw[thin, dotted ](0.9,0)--(2.7,0);
       \draw[thin ](2.7,0)--(2.9,0);
  \draw[fill=white  ] (2.9,0)circle(2pt);
  
      \draw[fill=white] (0,0)circle(2pt);
   \draw[fill =white  ] (0.7,0)circle(2pt);

\node(a)at(0,0.3){ $-b_1$};
\node(a)at(0.7,0.3){ $-b_2$};
\node(a)at(2.9,0.3){ $-b_n$};

  \end{tikzpicture} 
  \end{center}
  where $p/q = [b_1,b_2,\ldots,b_n]$.  The thick-thin decomposition of
  $(X,0)$ has been already described in the proof of Theorem
  \ref{th:thin fast loops}: the generic linear form has multiplicity
  $1$ along each $E_{i}$, and the \L-nodes are the two extremal
  vertices of the bamboo and any vertex with $b_i\ge 3$.  To get the
  adapted resolution graph of section \ref{sec:thick-thin} we blow up
  once between any two adjacent \L-nodes. Then the subgraph $\Gamma_j$
  associated with a thin piece $Z_j$ is either a $(-1)$-weighted
  vertex or a maximal string $\nu_i,\nu_{i+1},\dots,\nu_k$ of
  $(-2)$-weighted vertices
  excluding $\nu_1$ and $\nu_n$.
  There are no \T-nodes, so each Tjurina component is an \A-string
giving a special annulus.. The rate of the special
    annulus is $3/2$ if the Tjurina component is a $(-1)$-vertex and
    $(m+3)/2$ for a string of $m$ $(-2)$-vertices.
\end{example}
\begin{example}
[Cusp singularities] Let $(X,0)$ be a cusp singularity
(\cite{H2}). The dual graph of the minimal resolution of $(X,0)$
consists of a cycle of rational curves with Euler weights $\le -2$ and
at least one weight $\le -3$. Since such a singularity is taut, its
maximal cycle coincides with its fundamental cycle, and the generic
linear form has multiplicity $1$ along each irreducible component
$E_i$. Therefore, the \L-nodes are the vertices which carry Euler
class $\le -3$, and the argument is now similar to the lens space case
treated above: we blow up once between any two adjacent \L-nodes. Then
the subgraph $\Gamma_j$ associated with a thin piece $Z_j$ is either a
$(-1)$-weighted vertex or a maximal string
$\nu_i,\nu_{i+1},\dots,\nu_k$ of $(-2)$-weighted vertices.
As the distance $\epsilon$ to the origin goes to zero, we then see the
link $X^{(\epsilon)}$ as a ``necklace" of thick zones separated by
thin zones, each of the latter vanishing to a circle. Again,
the thin zones are special annuli with rates as in the
  previous example.
\end{example}
\begin{example}
[Brian\c con-Speder examples]

  The Brian\c con-Speder example we will consider is the
  $\mu$-constant family of singularities $$x^5+z^{15}+y^7z+txy^6=0$$
  depending on the parameter $t$. We will see that the bilipschitz
  geometry changes very radically when $t$ becomes $0$ (the same is
  true for the other Brian\c con-Speder families).

The minimal resolution graph of $(X,0)$ is
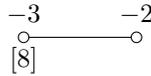
\begin{center}
 \begin{tikzpicture}
     \draw[thin ](0,0)--(1.5,0);
   
        \draw[fill=white] (0,0)circle(2pt);
   \draw[fill=white   ] (1.5,0)circle(2pt);

\node(a)at(0,0.3){ $-3$};
\node(a)at(0,-0.3){ [8]};
\node(a)at(1.5,0.3){ $-2$};

  \end{tikzpicture} 
\end{center}

For any $t$ the curves $x=0$ and $y=0$ are 

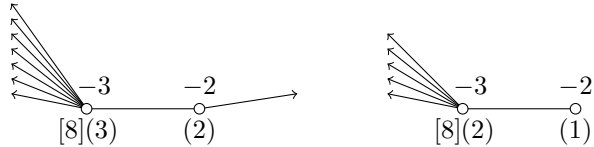
\begin{center}
 \begin{tikzpicture}
 
   \draw[thin ](0,0)--(1.5,0);
  
\draw[thin,>-stealth,->](0,0)--+(-1,0.2);
\draw[thin,>-stealth,->](0,0)--+(-1,0.4);
\draw[thin,>-stealth,->](0,0)--+(-1,0.6);
\draw[thin,>-stealth,->](0,0)--+(-1,0.8);
\draw[thin,>-stealth,->](0,0)--+(-1,1);
\draw[thin,>-stealth,->](0,0)--+(-1,1.2);
\draw[thin,>-stealth,->](0,0)--+(-1,1.4);

\draw[thin,>-stealth,->](1.5,0)--+(1.3,0.2);

 \draw[fill =white  ] (0,0)circle(2pt);
   \draw[fill=white   ] (1.5,0)circle(2pt);

\node(a)at(0.1,0.3){ $-3$};
\node(a)at(-0.2,-0.3){ [8]};
\node(a)at(0.2,-0.3){ (3)};
\node(a)at(1.5,-0.3){ (2)};
\node(a)at(1.5,0.3){ $-2$};

  \draw[thin ](5,0)--(6.5,0);

   \draw[thin,>-stealth,->](5,0)--+(-1,0.2);
\draw[thin,>-stealth,->](5,0)--+(-1,0.4);
\draw[thin,>-stealth,->](5,0)--+(-1,0.6);
\draw[thin,>-stealth,->](5,0)--+(-1,0.8);
\draw[thin,>-stealth,->](5,0)--+(-1,1);

 \draw[fill=white   ] (5,0)circle(2pt);
   \draw[fill=white   ] (6.5,0)circle(2pt);
   
\node(a)at(5.1,0.3){ $-3$};
\node(a)at(5.2,-0.3){ (2)};
\node(a)at(4.8,-0.3){ [8]};
\node(a)at(6.5,0.3){ $-2$};
\node(a)at(6.5,-0.3){(1)};

  \end{tikzpicture} 
  \end{center}

The curve $z=0$ is 
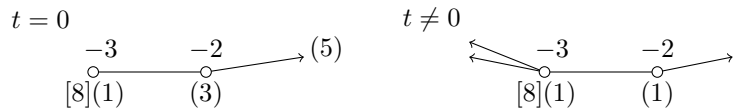
\begin{center}
 \begin{tikzpicture}
 
   \draw[thin ](0,0)--(1.5,0);
\draw[thin,>-stealth,->](1.5,0)--+(1.3,0.2);
\node(a)at(3.1,0.3){ (5)};

 \draw[fill =white  ] (0,0)circle(2pt);
   \draw[fill =white  ] (1.5,0)circle(2pt);
  
\node(a)at(0.1,0.3){ $-3$};
\node(a)at(-0.2,-0.3){ [8]};
\node(a)at(0.2,-0.3){ (1)};
\node(a)at(1.5,-0.3){ (3)};
\node(a)at(1.5,0.3){ $-2$};
\node(a)at(-.7,.7){ $t=0$};

  \draw[thin ](6,0)--(7.5,0);
   
   \draw[thin,>-stealth,->](6,0)--+(-1,0.2);
\draw[thin,>-stealth,->](6,0)--+(-1,0.4);

 \draw[thin,>-stealth,->](7.5,0)--+(1,0.2);

 \draw[fill =white  ] (6,0)circle(2pt);
   \draw[fill =white   ] (7.5,0)circle(2pt);

\node(a)at(6.1,0.3){ $-3$};
\node(a)at(6.2,-0.3){ (1)};
\node(a)at(5.8,-0.3){ [8]};
\node(a)at(7.5,0.3){ $-2$};
\node(a)at(7.5,-0.3){(1)};
\node(a)at(4.5,.7){ $t\neq0$};

\end{tikzpicture} 
\end{center}

A generic linear form for $t\ne0$ is $z=0$, given by the right hand
graph above.  To resolve the linear system of hyperplane sections one
must blow up twice at the left node, giving three \L-nodes:
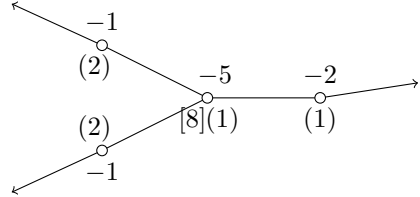
\begin{center}
 \begin{tikzpicture}
 
   \draw[thin ](0,0)--(1.5,0);
    \draw[thin ](0,0)--+(-1.4,0.7);
      \draw[thin ](0,0)--+(-1.4,-0.7);

\draw[thin,>-stealth,->](1.5,0)--+(1.3,0.2);
\draw[thin,>-stealth,->](-1.4,0.7)--+(-1.2,0.55);
\draw[thin,>-stealth,->](-1.4,-0.7)--+(-1.2,-0.555);

 \draw[fill=white   ] (0,0)circle(2pt);
   \draw[fill=white   ] (1.5,0)circle(2pt);
   
     \draw[fill=white    ] (-1.4,0.7)circle(2pt);
       \draw[fill=white   ] (-1.4,-0.7)circle(2pt);
   
\node(a)at(0.1,0.3){ $-5$};
\node(a)at(-0.2,-0.3){ [8]};
\node(a)at(0.2,-0.3){ (1)};
\node(a)at(1.5,-0.3){ (1)};
\node(a)at(1.5,0.3){ $-2$};

\node(a)at(-1.4,1){ $-1$};
\node(a)at(-1.4,-1){ $-1$};

\node(a)at(-1.5,0.4){ (2)};
\node(a)at(-1.5,-0.4){ (2)};
\end{tikzpicture} 
\end{center}
The generic linear for $t=0$ , with the linear system of hyperplane
sections resolved, is

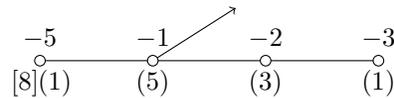
\begin{center}
\begin{tikzpicture}
  \draw[thin ](0,0)--(4.5,0);
  
  \draw[thin,>-stealth,->](1.5,0)--+(1.1,0.7);
  
   \draw[fill=white   ] (0,0)circle(2pt);
    \draw[fill =white ] (1.5,0)circle(2pt);
     \draw[fill=white   ] (3,0)circle(2pt);
      \draw[fill=white   ] (4.5,0)circle(2pt);
      
      \node(a)at(0,0.3){ $-5$};
        \node(a)at(0.2,-0.3){ (1)};
         \node(a)at(-0.2,-0.3){ [8]};

        \node(a)at(1.5,0.3){ $-1$};
        \node(a)at(1.5,-0.3){ (5)};
\node(a)at(3,0.3){ $-2$};
        \node(a)at(3,-0.3){ (3)};

\node(a)at(04.5,0.3){ $-3$};
        \node(a)at(4.5,-0.3){ (1)};
\end{tikzpicture} 
\end{center}

So there is just one \L-node, and it is equal to neither of the
original vertices of the resolution. We compute the rate at the
\T-node from the Puiseux expansion of the branches of $\Delta$ and we
obtain the following bilipschitz description:

\begin{center}
 \begin{tikzpicture}
   \draw[thin ](0,0)--(1.5,0);
    \draw[thin ](0,0)--+(-1.4,0.7);
      \draw[thin ](0,0)--+(-1.4,-0.7);

\draw[thin,>-stealth,->](1.5,0)--+(1.3,0.2);
\draw[thin,>-stealth,->](-1.4,0.7)--+(-1.2,0.55);
\draw[thin,>-stealth,->](-1.4,-0.7)--+(-1.2,-0.555);

 \draw[fill    ] (0,0)circle(2pt);
   \draw[fill   ] (1.5,0)circle(2pt);
   
     \draw[fill    ] (-1.4,0.7)circle(2pt);
       \draw[fill   ] (-1.4,-0.7)circle(2pt);
   
\node(a)at(-0.2,0.3){ $-5$};
\node(a)at(0.3,0.3){ [8]};
    \node(a)at(0,-0.3){ $\it 2$};
\node(a)at(1.5,-0.3){ \it 1};
\node(a)at(1.5,0.3){ $-2$};

\node(a)at(-1.4,1){ $-1$};
\node(a)at(-1.4,-1){ $-1$};

\node(a)at(-1.5,0.4){ \it 1};
\node(a)at(-1.5,-0.4){ \it 1};

 \node(a)at(0.8,1.3){ $t\neq0$};

\begin{scope}[xshift=4.5cm]
 \node(a)at(1.0,1.3){ $t=0$};
  \draw[thin ](0,0)--(4.5,0);
  
  \draw[thin,>-stealth,->](1.5,0)--+(1.1,0.7);
  
   \draw[fill   ] (0,0)circle(2pt);
    \draw[fill   ] (1.5,0)circle(2pt);
     \draw[fill=white   ] (3,0)circle(2pt);
      \draw[fill=white   ] (4.5,0)circle(2pt);
      
      \node(a)at(-0.3,0.3){ $-5$};
         \node(a)at(0.2,0.3){ [8]};
          \node(a)at(0,-0.3){ $\it 2$};

        \node(a)at(1.5,0.3){ $-1$};
        \node(a)at(1.5,-0.3){ \it 1};
\node(a)at(3,0.3){ $-2$};

\node(a)at(04.5,0.3){ $-3$};

\end{scope}
 \end{tikzpicture} 
  \end{center}
\end{example}

\end{document}